%% file: main.tex
\documentclass{article}

\usepackage{iclr2025_conference,times}
\usepackage[colorlinks,linkcolor={blue!75!black},urlcolor=red,citecolor={blue!75!black}]{hyperref}
\input{packages}
\input{commands}

\title{Stochastic Polyak Step-sizes and Momentum: Convergence Guarantees and Practical \\ Performance}

\author{Dimitris Oikonomou \\
CS \& MINDS \\
Johns Hopkins University \\
\texttt{doikono1@jh.edu} \\
\And
Nicolas Loizou \\
AMS \& MINDS \\
Johns Hopkins University \\
\texttt{nloizou1@jh.edu} \\
}

\iclrfinalcopy

\begin{document}

\maketitle

\begin{abstract}
    Stochastic gradient descent with momentum, also known as Stochastic Heavy Ball method (SHB), is one of the most popular algorithms for solving large-scale stochastic optimization problems in various machine learning tasks. In practical scenarios, tuning the step-size and momentum parameters of the method is a prohibitively expensive and time-consuming process. In this work, inspired by the recent advantages of stochastic Polyak step-size in the performance of stochastic gradient descent (SGD), we propose and explore new Polyak-type variants suitable for the update rule of the SHB method. In particular, using the Iterate Moving Average (IMA) viewpoint of SHB, we propose and analyze three novel step-size selections: MomSPS$_{\max}$, MomDecSPS, and MomAdaSPS. For MomSPS$_{\max}$, we provide convergence guarantees for SHB to a neighborhood of the solution for convex and smooth problems (without assuming interpolation). If interpolation is also satisfied, then using MomSPS$_{\max}$, SHB converges to the true solution at a fast rate matching the deterministic HB. The other two variants, MomDecSPS and MomAdaSPS, are the first adaptive step-size for SHB that guarantee convergence to the exact minimizer - without a priori knowledge of the problem parameters and without assuming interpolation. Our convergence analysis of SHB is tight and obtains the convergence guarantees of stochastic Polyak step-size for SGD as a special case. We supplement our analysis with experiments validating our theory and demonstrating the effectiveness and robustness of our algorithms. 
\end{abstract}

\vspace{-2mm}
\section{Introduction}
\vspace{-1mm}

We consider the unconstrained finite-sum optimization
problem,
\begin{equation}
    \label{eq:main-problem}
    \min_{x\in\R^d}\left[f(x)=\frac{1}{n}\sum_{i=1}^nf_i(x)\right],
\end{equation}
where each $f_i:\R^d\to\R$ is convex, smooth, and lower bounded by $\ell_i^*$. Let $X^*$ be the set of minimizers of (\ref{eq:main-problem}). We assume that $X^*\neq\emptyset$ and we fix $x^*\in X^*$. This problem is the cornerstone of machine learning tasks, \citep{hastie2009elements}, where $x$ corresponds to the model parameters, $f_i(x)$ represents the loss on the training point $i$, and the aim is to minimize the average loss $f(x)$ across training points. 

When $n$ is large, stochastic gradient methods are the preferred methods for solving \eqref{eq:main-problem} mainly because of their cheap per iteration cost. One of the most popular stochastic algorithms for solving such large-scale machine learning optimization problems is stochastic gradient descent (SGD) with momentum, \citep{sutskever2013importance}, a.k.a. stochastic heavy ball method (SHB) given by: 
\begin{equation}
    \label{eq:shb}
    x^{t+1}=x^t-\g_t \nabla f_{S_t}(x^t) +\b_t(x^t-x^{t-1}).\tag{SHB}
\end{equation}
where $S_t\subseteq[n]$ a random subset of data-points (mini-batch) with cardinality $B$ sampled independently at each iteration $t$, and $\nabla f_{S_t}(x^t)=\frac{1}{B}\sum_{i\in S_t}\nabla f_i(x^t)$ is the mini-batch gradient. Here $\g_t>0$ is the step-size/learning rate at iteration $t$ while $\b_t\geq0$ represents the momentum parameter. When the momentum parameter $\b_t=0, \forall t\geq 0$, then the update rule \ref{eq:shb} is equivalent to the well-studied mini-batch SGD, $x^{t+1}=x^t-\g_t \nabla f_{S_t}(x^t)$, \citep{robbins1951stochastic}, which has been efficiently analyzed under different properties of problem \eqref{eq:main-problem} and different step-size selections $\gamma_t$ \citep{NemYudin1983book,nemirovski2009robust,HardtRechtSinger-stability_of_SGD,needell2014stochastic,nguyen2018sgd,gower2019sgd,gower2021sgd}. Additionally, when the cardinality of the random subset $S_t$ is $B=n$, then the update rule of \ref{eq:shb} is equivalent to the deterministic heavy ball method (HB) proposed by \citet{polyak1964some}, as a way to improve the convergence behavior of deterministic Gradient Descent (GD). 

There is a rich literature on the convergence of \ref{eq:shb} and HB in different scenarios. In \citet{polyak1964some}, it was proved that for a specific choice of the step-size $\g$ and the momentum parameter $\b$, the HB method enjoys an accelerated linear convergence when minimizing strongly convex quadratic functions while more recently, \citet{ghadimi2015global} proved a global sublinear convergence guarantee for HB for convex and smooth functions. In the stochastic setting, several works focus on convergence guarantees of \ref{eq:shb} under constant step-size and momentum parameters \citep{ma2018quasi,kidambi2018insufficiency,yan2018unified,gitman2019understanding,liu2020improved}. However, in practical scenarios, these choices suffer from a prohibitively expensive and time-consuming hyper-parameter tuning process. This has motivated a large body of research on the development of adaptive \ref{eq:shb} - a method that adapts their parameters using information collected during the iterative process. Such analysis is challenging, and the current adaptive versions of \ref{eq:shb} either focus on the full batch setting (deterministic) \citep{barre2020complexity,saab2022adaptive,wang2023generalized}, or assume that an interpolation condition is satisfied \citep{schaipp2023momo} or focus on moving averaged gradient (a different form of momentum)~\citep{wang2023generalized}. 

Previous studies in the fully stochastic (non-interpolated) scenario have predominantly concentrated on analyzing adaptive versions of SGD, with limited attention given to developing adaptive variants for the \ref{eq:shb}. In this work, we take inspiration from the recently introduced and highly efficient Polyak-type adaptive step-sizes for SGD and investigate its applicability and extension to \ref{eq:shb}.

\vspace{-2mm}
\subsection{Main Contributions}
\vspace{-1mm}

Our main contributions are summarized below.

\textbf{$\diamond$ Efficient Polyak Step-sizes via IMA viewpoint.} 
We explain and illustrate by experiment (see \Cref{fig:mosps_vs_naive-snip})  why naively using SPS$_{\max}$ of \citet{loizou2021stochastic} as a step-size $\gamma_t$ in the update rule of \ref{eq:shb} is not robust, leading to divergence even in simple problems. To resolve this issue, we provide an alternative way of selecting Polyak-type step-sizes for \ref{eq:shb} via the Iterate Moving Average viewpoint from \citet{sebbouh2021almost}. Through our approach, we propose three novel adaptive step-size selections, namely \ref{eq:mospsmax}, \ref{eq:modecsps}, \ref{eq:moadasps}. Each of the proposed step-sizes depends on the choice of the momentum parameter $\beta$ adding further stability to \ref{eq:shb} and comes with specific benefits over their constant step-size counterparts or other adaptive variants of \ref{eq:shb}. 

\textbf{$\diamond$ MomSPS$_{\max}$: Convergence of SHB in non-interpolated setting.} 
Our first step-size selection of \ref{eq:shb} is \ref{eq:mospsmax}, which has a similar structure to SPS$_{\max}$ of \citet{loizou2021stochastic} but includes also $(1-\beta)$ in its expression. For this choice, we provide convergence guarantees for \ref{eq:shb} to a neighborhood of the solution for convex and smooth problems. Our analysis provides the first convergence guarantees of adaptive \ref{eq:shb} using Polyak-type step-size. Previous works on Polyak step-size with momentum in the stochastic setting have guarantees either only under the interpolation setting \citep{schaipp2023momo} or for a moving averaged gradient momentum \citep{wang2023generalized}. In addition, as a corollary of our main theoretical results, we show that \ref{eq:mospsmax} under the interpolation setting and in deterministic scenarios (full batch) converges to the true solution at a fast rate matching the deterministic HB. 

\textbf{$\diamond$ Convergence of SHB to exact solution via MomDecSPS and MomAdaSPS.} 
Inspired by two recent Polyak-type step-size selections for SGD, the DecSPS in \citet{orvieto2022dynamics} and AdaSPS in \citet{jiang2023adaptive}, we propose two new ways for tuning the step-size for \ref{eq:shb}. These are \ref{eq:modecsps} and \ref{eq:moadasps}. Our analysis provides the first $O(1/\sqrt{T})$ convergence guarantees to the exact solution in the non-interpolated regime for a Polyak-type adaptive variant of \ref{eq:shb}. Our proposed update rules converge for any choice of momentum parameter $\b\in[0,1)$, which makes them particularly useful in practical scenarios.  

\textbf{$\diamond$ Tight Convergence Guarantees.} 
All of our convergence guarantees are \emph{true} generalizations of the theoretical analysis of SGD using SPS$_{\max}$, DecSPS, and AdaSPS. That is, if $\beta=0$ (no momentum) in the update rule of \ref{eq:shb}, our theorems obtain as a special case the best-known convergence rates of Polyak-type step-size for SGD, highlighting the tightness of our analysis. See also \Cref{tab:results} for a summary of our main complexity results and a comparison with closely related works. 

\textbf{$\diamond$ Further Convergence Results.} 
As a byproduct of our theoretical analysis, we provide two interesting corollaries: a \textit{novel analysis of constant step-size \ref{eq:shb}} (as a corollary of our Theorem on \ref{eq:mospsmax}) and the first \textit{robust convergence of \ref{eq:shb}} via our theorem on \ref{eq:moadasps}. For a constant step-size, the \Cref{cor:shb-const} of our \Cref{thm:shb-sps-max} allows larger step-sizes than the analysis of \ref{eq:shb} in \citet{liu2020improved} and provides convergence without assuming the restrictive bounded variance condition (there exist $q>0$ such that $\E\|\nabla f_i(x)-\nabla f(x)\|^2 < q$). In addition, via \Cref{thm:shb-ada-sps} we provide the first robust convergence of adaptive \ref{eq:shb} that guarantees convergence to the exact solution and automatically adapts to whether our problem is interpolated or not. That is, if interpolation is assumed, then the rate of \ref{eq:shb} with \ref{eq:moadasps} is the same as the rate of \ref{eq:shb} with \ref{eq:mospsmax} (or constant step-size), and if no interpolation is assumed, then it matches the rate of \ref{eq:shb} with \ref{eq:modecsps}. The analysis achieves the best-known rates in both settings. 

\textbf{$\diamond$ Numerical Evaluation.} 
In \Cref{sec:experiments}, we verify our theoretical results via numerical experiments on various problems, demonstrating the effectiveness and practicality of our approach. An open-source implementation of our method is available at \url{https://github.com/dimitris-oik/MomSPS}.

\vspace{-4mm}
\setlength{\tabcolsep}{2.5pt}
\begin{table*}[t]
    \small
    \centering
    \begin{tabular}{llccl}
        \toprule
        \textbf{Step-size} & \textbf{Assumptions / Setting} & \textbf{Adaptive} & \textbf{Exact Convergence} & \textbf{Rate} \\
        \midrule
        Constant \citep{liu2020improved} & Knowledge of $L$ & \xmark & \xmark & $O(\frac{1}{T}+\hat{\s}^2)$ \\
        IMA \citep{sebbouh2021almost} & Knowledge of $L$ & \xmark & \cmark &  $O(\frac{1}{T}+\hat{\s}^2)$ \\
        ALR-SMAG \citep{wang2023generalized} & Moving Averaged Gradient & \cmark & \xmark & $O(\frac{1}{T}+\s^2)$\\
        \rowcolor{green!30}\ref{eq:mospsmax} (Thm \ref{thm:shb-sps-max}) & Restriction on $\b$ & \cmark & \xmark & $O(\frac{1}{T}+\s^2)$ \\
        \rowcolor{green!30}\ref{eq:modecsps} (Thm \ref{thm:shb-dec-sps}) & Bounded Iterates & \cmark & \cmark & $O(\frac{1}{\sqrt{T}})$ \\
        \rowcolor{green!30}\ref{eq:moadasps} (Thm \ref{thm:shb-ada-sps}) & Bounded Iterates & \cmark & \cmark & $O(\frac{1}{T}+\frac{\s}{\sqrt{T}})$ \\
        \bottomrule
    \end{tabular}
    \caption{Summary of the considered step-sizes and the corresponding theoretical results in the stochastic setting. All the rates are given for convex and smooth functions. The quantity of convergence in our rates is $\E[f(\overline{x}^T)-f(x^*)]$, where $\overline{x}^T=\frac{1}{T}\sum_{t=0}^Tx^t$. Here $\hat{\s}^2=\E\|\nabla f_i(x^*)\|^2<\infty$ and $\s^2=\E[f_i(x^*)-\ell_i^*]<\infty$. The \textquote{Exact Convergence} column refers to convergence to the exact solution $x^*$ with no interpolation assumption. }
    \label{tab:results}
\end{table*}

\section{Exploring the Interplay of SPS and Heavy Ball Momentum}
\vspace{-1mm}

In this section, we present the expression of Stochastic Polyak Step-size (SPS) and its different variants. We illustrate that a naive combination of SPS with momentum is not robust, leading to divergence in simple problems. We explain how we resolve this issue using the Iterate Moving Average viewpoint of \ref{eq:shb} and propose three adaptive Polyak-type step-sizes for \ref{eq:shb}.

\vspace{-1mm}
\subsection{Background on Stochastic Polyak Step-size}
\vspace{-1mm}
\label{sec:sgd-polyak}

The deterministic Polyak step-size (PS) was first introduced by \citet{polyak1987introduction} as an efficient step-size selection for GD for solving convex optimization problems. It has the following expression: $\g_t=\frac{f(x^t)-f(x^*)}{\|\nabla f(x^t)\|^2}$ which is obtained by minimizing an upper bound of the quantity $\|x^{t+1}-x^*\|^2$ in the analysis of GD. Since its original proposal, the PS has been successfully used in the analysis of deterministic subgradient methods in different settings with favorable convergence guarantees \citep{boyd2003subgradient, davis2018subgradient, hazan2019revisiting}. PS requires the prior knowledge of $f(x^*)$, which might look like a strong assumption. However, as shown in \citet{boyd2003subgradient}, this is known in several applications, including finding a point in the intersection of convex sets and positive semi-definite matrix completion.

Inspired by the convergence of PS in the deterministic setting, \citet{loizou2021stochastic} has effectively modified the Polyak step-size for the stochastic setting, achieving convergence rates comparable to those of finetuned SGD. The proposed stochastic Polyak step-size (SPS) has several benefits, including independence on parameters of the problem (e.g., $L$-smoothness or $\mu$ strong convexity) and competitive performance in over-parametrized models.  More specifically, \citet{loizou2021stochastic} proposed the SPS$_{\max}$ given below\footnote{Originally \citet{loizou2021stochastic} use $f_{S_t}^*$ instead of the lower bound $\ell_{S_t}^*$. The lower bound is due to \citet{orvieto2022dynamics}, which proves that the more relaxed lower bound can still lead to the same convergence guarantees.}:
\vspace{-3mm}
\begin{equation}
    \label{eq:spsmax}
    \g_t=\min\left\{\frac{f_{S_t}(x^t)-\ell_{S_t}^*}{c\|\nabla f_{S_t}(x^t)\|^2},\g_b\right\}.\tag{SPS$_{\max}$}
\end{equation}
Here, $\g_b>0$ is a bound that restricts SPS from being very large and is essential to ensure convergence to a small neighborhood around the solution and $c>0$ is a positive constant that depends on the function class the objective $f$ belongs to. 

As mentioned in \citet{orvieto2022dynamics}, the SPS$_{\max}$ comes with strong convergence guarantees and competitive performance; however, it has one main drawback when used in non-over-parameterized regimes: It can guarantee convergence only to a neighborhood of the solution. For this reason, \citet{orvieto2022dynamics} suggests a decreasing variant of the original SPS$_{\max}$ named DecSPS, given by:
\begin{equation}
    \label{eq:decsps}
    \g_t=\frac{1}{c_t}\min\left\{\frac{f_{S_t}(x^t)-\ell_{S_t}^*}{\|\nabla f_{S_t}(x^t)\|^2},c_{t-1}\g_{t-1}\right\},\tag{DecSPS}
\end{equation}
where $c_t$ is an increasing sequence of positive real numbers and $c_{-1}:=c_0$ and $\g_{-1}=\g_b>0$. The authors proved that SGD with \ref{eq:decsps} and $c_t=\sqrt{t+1}$, converges with a sublinear rate $O(1/\sqrt{T})$ for convex and smooth functions with bounded iterates (i.e., $D^2 := \max_{t \in [T-1]}\|x^t-x^*\|^2<\infty$).

More recently, in \citet{jiang2023adaptive}, another decreasing variant of SPS was introduced, named AdaSPS:
\begin{equation}
   \textstyle
    \label{eq:adasps}
    \g_t=\min\left\{\frac{f_{S_t}(x^t)-\ell_{S_t}^*}{c\|\nabla f_{S_t}(x^t)\|^2}\frac{1}{\sqrt{\sum_{s=0}^tf_{S_s}(x^s)-\ell_{S_s}^*}},\g_{t-1}\right\},\tag{AdaSPS}
\end{equation}
where $\g_{-1}=+\infty$ and $c>0$. For convex and $L$-smooth functions with bounded iterates, it can be shown that SGD with AdaSPS converges to an exact solution with a rate $O\left(\frac{\t^2}{T}+\frac{\t\s}{\sqrt{T}}\right)$, where $\t=2cLD^2+\frac{1}{c}$. The interesting aspect of the convergence of AdaSPS is that it provides a robust result for SGD, meaning that the method recovers the best bounds for both the interpolated ($\sigma=0$) and non-interpolated regimes.

\vspace{-1mm}
\subsection{Naive SPS in SHB}
\vspace{-1mm}
\label{sec:naive-sps}

All of the above variants of SPS were proposed and analyzed for SGD (\ref{eq:shb} with no momentum). With the increased popularity of momentum in machine learning, one can naturally ask the following question: \textit{Is it possible to combine SPS with momentum?}
This question was partially answered in \citet{wang2023generalized} when they proposed a Polyak-type step-size for a moving averaged gradient (MAG) momentum but not for the \ref{eq:shb}. Their MAG framework is not equivalent to \ref{eq:shb} when their stepsizes are adaptive and while their proposed stepsize reduces to the original \ref{eq:spsmax} stepsize when $\b=0$, their guarantees do not reduce to the original \ref{eq:spsmax} guarantees for convex functions. 

\begin{wrapfigure}{r}{0.55\textwidth}
\begin{minipage}{0.55\textwidth}
\vspace{-7mm}
\begin{figure}[H]
	\centering
    \begin{subfigure}{0.45\columnwidth}
		\includegraphics[width=\textwidth]{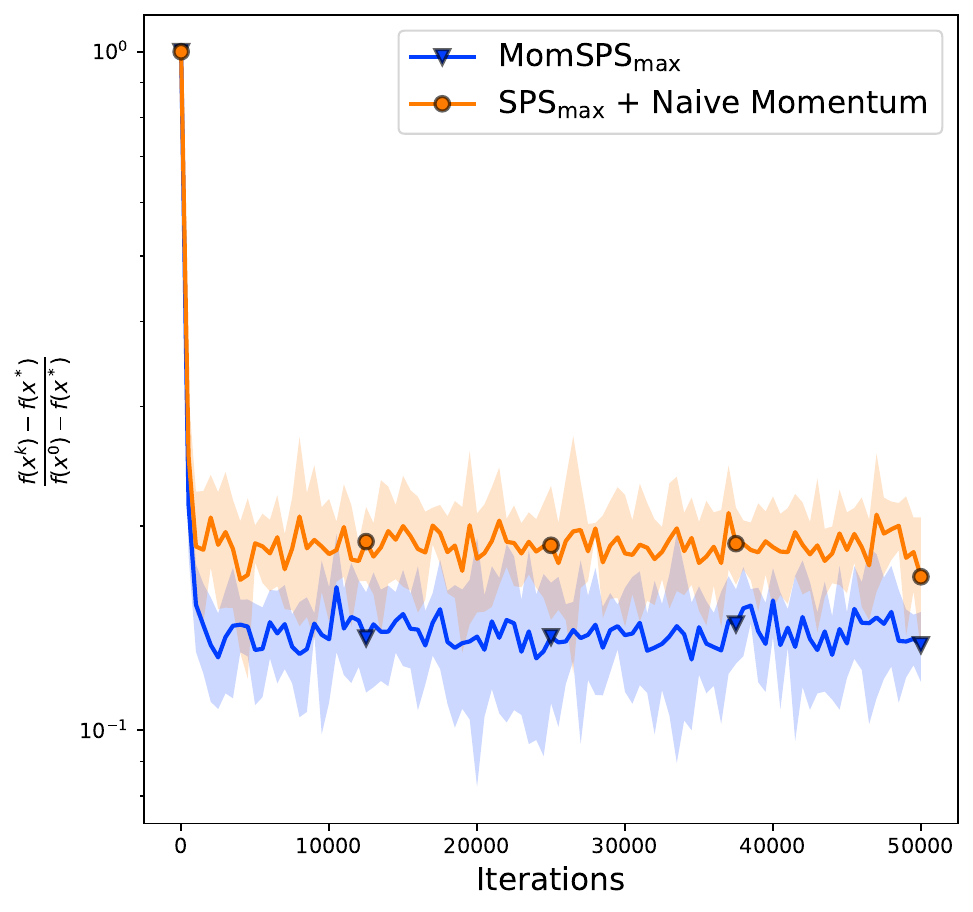}
	\end{subfigure}
    ~
	\begin{subfigure}{0.45\columnwidth}
		\includegraphics[width=\textwidth]{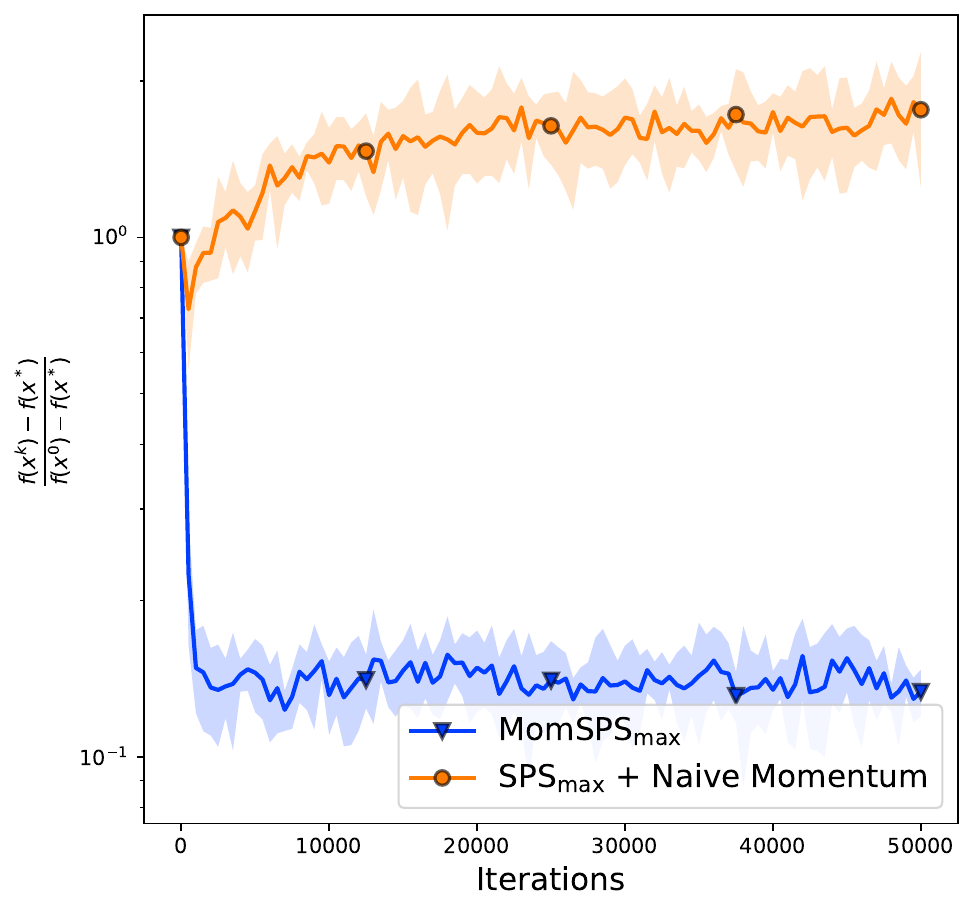}
	\end{subfigure}
	\caption{Comparison of \ref{eq:mospsmax} versus SPS$_{\max}$ with naive momentum for different momentum parameters $\b$ on a logistic regression problem. Left: $\b=0.2$, Right: $\b=0.5$}
	\label{fig:mosps_vs_naive-snip}
\end{figure}
\end{minipage}
\end{wrapfigure}

The most straightforward approach is to directly apply the SPS$_{\max}$ in the \ref{eq:shb} update rule. Let us call this update rule \textit{SPS$_{\max}$ with naive momentum}. Unfortunately, this approach does not necessarily lead to convergence for natural choices of momentum parameter $\b\in(0,1)$, as shown in \Cref{fig:mosps_vs_naive-snip}.\footnote{See also \Cref{fig:mosps_vs_naive} in Appendix for more values of $\b$.} In this experiment, even for simple convex and smooth problems like a logistic regression with synthetic data, the naive rule fails to converge when $\b$ gets larger (a typical choice for momentum parameter is $\b=0.9$). This indicates that a more careful step-size selection is needed, which may also depend on the momentum parameter. In \Cref{fig:mosps_vs_naive-snip}, we compare the SPSmax with naive momentum with one of our proposed and analyzed step-size selection \ref{eq:mospsmax}. As seen in \Cref{fig:mosps_vs_naive-snip}, when $\b$ is small, then both SPS$_{\max}$ with naive momentum and \ref{eq:shb} with \ref{eq:mospsmax} have similar behavior. When $\b=0$, the two methods have identical performance as both are reduced to the same method: SGD with SPS$_{\max}$. However, as $\b$ gets larger, the performance of SPS$_{\max}$ with naive momentum gets worse and less stable, and at some point, it diverges. On the other hand, in this example, \ref{eq:shb} with \ref{eq:mospsmax} converges for any of the selected $\b$.

\vspace{-1mm}
\subsection{Iterative Moving Average: Balance between SPS and Momentum}
\vspace{-1mm}

Having explained how naively combining SPS$_{\max}$ with heavy ball momentum can lead to divergence, in this section, we leverage the iterate moving-average (IMA) viewpoint of \ref{eq:shb} from \citet{sebbouh2021almost} to propose SPS-type adaptations that depend on momentum parameter as well for \ref{eq:shb}. As we will see later, this viewpoint provides stability and robustness to the proposed update rules.

\textbf{Iterate Moving Average (IMA).}
\citet{sebbouh2021almost} provide the IMA as an alternative way of expressing the update rule of \ref{eq:shb} and explain how the IMA formulation is crucial in comparing \ref{eq:shb} and SGD as it allows to establish connections between the step-sizes of the two methods. The Iterate Moving Average  method is given by the following update rule:
\begin{equation*}
    \label{eq:ima}
    z^{t+1}=z^t-\h_t\nabla f_{S_t}(x^t),\q x^{t+1}=\frac{\l_{t+1}}{\l_{t+1}+1}x^t+\frac{1}{\l_{t+1}+1}z^{t+1},
\end{equation*}
where $z^0=x^0$, $\h_t>0$ and $\l_t\geq0$. 
As proved in \citet{sebbouh2021almost}, if for any $t\in\N$, holds $1+\l_{t+1}=\frac{\l_t}{\b_t}\text{ and }\h_t=(1+\l_{t+1})\g_t$ then the $x_t$ iterates of the IMA method are equal to the $x_t$ iterates produced by the \ref{eq:shb} method. For completeness, we include the proof of this statement in \Cref{sec:ima}. 

In our proposed methods, we select the most common setting of constant momentum ($\b_t=\b$). Considering the equivalence between IMA and \ref{eq:shb} under the correct parameter selection, let us present the following proposition that will allow us to obtain convergence guarantees for \ref{eq:shb} via a convergence analysis of IMA. 
\begin{proposition}
    \label{pro:ima-ss}
    If $\l_t=\l\geq0$ then assuming that $\b=\frac{\l}{1+\l}$ and $\g_t=(1-\b)\h_t$ the $x_t$ iterates of the IMA method are equal to the $x_t$ iterates produced by the \ref{eq:shb} method. 
\end{proposition}
Using the above proposition, let us provide the following corollary that explains the derivation of our proposed step-size selections for \ref{eq:shb}. 

\begin{corollary}
    \label{cor:trans-app}
    Let the step-size $\h_t$ in IMA be one of the previously proposed Polyak step-sizes: \ref{eq:spsmax}, \ref{eq:decsps} or \ref{eq:adasps} and let $\l_t=\l$. Then via \Cref{pro:ima-ss}, the \ref{eq:shb} has constant momentum parameter $\b_t=\b$ and the step-sizes are \ref{eq:mospsmax}, \ref{eq:modecsps} or \ref{eq:moadasps} respectively.
\end{corollary}

\textbf{Three New Adaptive Step-sizes for SHB.} 
Using the IMA viewpoint and \Cref{pro:ima-ss} and \Cref{cor:trans-app} let us present three Polyak-type step-size for \ref{eq:shb}.

\textbf{(i) MomSPS$_{\max}$.}
Let us start with the variant associated with the SPS$_{\max}$. That is, via \Cref{cor:trans-app} if we use SPS$_{\max}$ as a step-size for IMA, then this is equivalent to using the following step-size in the update rule of \ref{eq:shb}:
\vspace{-2mm}
\begin{equation}
    \label{eq:mospsmax}
    \g_t=(1-\b)\min\left\{\frac{f_{S_t}(x^t)-\ell_{S_t}^*}{c\|\nabla f_{S_t}(x^t)\|^2}, \g_b\right\}.\tag{MomSPS$_{\max}$}
\end{equation}
Here, the parameter $\g_b>0$ has the same purpose as in the original SPS$_{\max}$, and it is a bound that restricts \ref{eq:mospsmax} from being very large and is essential to ensure convergence to a neighborhood around the solution. It is clear from its expression that \ref{eq:mospsmax} has the same form as SPS$_{\max}$ but multiplied by a \emph{correcting momentum factor} $1-\b$. This follows from \Cref{cor:trans-app}. In practice, this small change allows the \ref{eq:shb} to be more \textquote{stable} for different momentum parameters, as we show in \Cref{fig:mosps_vs_naive-snip}. Using \ref{eq:mospsmax} and a restriction on the momentum parameter $\beta$ in \Cref{sec:convergence}, we establish $O(1/T)$ convergence for \ref{eq:shb} up to a neighborhood of the solution for convex and smooth functions. 

Our following two proposed step-size selections are variants of the two decreasing variants of Polyak step-sizes \ref{eq:decsps} and \ref{eq:adasps}. As such, we can prove \ref{eq:shb}'s convergence to the exact solution instead of a neighborhood. More importantly, using \ref{eq:modecsps} or \ref{eq:moadasps}, we can provide convergence guarantees for any choice of the momentum coefficients $\b\in[0,1)$ making \ref{eq:shb} fully adaptive (no tuning necessary). 

\textbf{(ii) MomDecSPS.}
Firstly, we propose an adaptation of DecSPS with momentum. We call this step-size selection MomDecSPS, and it is given by:
\begin{equation}
    \label{eq:modecsps}
    \g_t=\min\left\{\frac{(1-\b)[f_{S_t}(x^t)-\ell_{S_t}^*]}{c_t\|\nabla f_{S_t}(x^t)\|^2},\frac{\g_{t-1}c_{t-1}}{c_t}\right\},\tag{MomDecSPS}
\end{equation}
where $\g_{-1}:=\g_b>0$ is a step-size bound and $c_t=c\sqrt{t+1}$ for $t\geq0$ with $c_{-1}:=c_0=c>0$ be a constant to regulate the step-size. Note that $\g_t$ is indeed decreasing. This step-size is adaptive and, in its definition, does not require knowledge of any function properties (e.g., smoothness/strong convexity constants). In \Cref{sec:convergence} we prove that \ref{eq:shb} with \ref{eq:modecsps} and constant momentum $\beta\in [0,1)$ converges to exact solution with a rate $O(1/\sqrt{T})$. 

\textbf{(iii) MomAdaSPS.} 
Similar to the previous two steps, let us propose the following adaptation of AdaSPS with momentum. We call this MomAdaSPS, and is given by: 
\begin{equation}
    \textstyle
    \label{eq:moadasps}
    \g_t=\min\left\{\frac{(1-\b)[f_{S_t}(x^t)-\ell_{S_t}^*]}{c\|\nabla f_{S_t}(x^t)\|^2\sqrt{\sum_{s=0}^tf_{S_s}(x^s)-\ell_{S_s}^*}},\g_{t-1}\right\}.\tag{MomAdaSPS}
\end{equation}
By definition, this is a decreasing step-size ($\g_t \leq \g_{t-1}$). In \Cref{sec:convergence}, we prove that \ref{eq:shb} with \ref{eq:moadasps} also converges to the exact solution for any choice of $\b\in[0,1)$. More specifically, we show that \ref{eq:moadasps} converges to the exact solution with a rate $O\left(\frac{1}{T}+\frac{\s}{\sqrt{T}}\right)$. Thus, when we are in the interpolation regime (i.e., $\s^2=0$), the convergence of \ref{eq:shb} with \ref{eq:moadasps} matches  \ref{eq:shb} with \ref{eq:mospsmax} while when the setting does not satisfy interpolation condition, this became equivalent to the $O(1/\sqrt{T})$ of \ref{eq:modecsps}.  Following the terminology of \citet{jiang2023adaptive}, our results are the first \emph{robust} adaptive step-size selection for \ref{eq:shb}, as it can automatically adapt to the optimization setting (interpolation vs. non-interpolation). 

Let us close this section by mentioning two remarks related to the above three step-size selections:

\begin{remark}
    The \textquote{correcting factor} $1-\b$ is outside the minimum in the \ref{eq:mospsmax} while for the decreasing variants \ref{eq:modecsps} and \ref{eq:moadasps}, it only appears in the first term of the minimum. This follows from \Cref{pro:ima-ss} and the derivation of \Cref{cor:trans-app} and is explained in detail in \Cref{sec:ima}. In \Cref{sec:other-choices}, via experiments, we also test other variants of the above step-sizes with the correcting factor outside of the minimum. We observe that the other choices do not have good practical performance. 
\end{remark}

\begin{remark}
    For no momentum, i.e., $\b=0$, all three proposed step-size choices, \ref{eq:mospsmax}, \ref{eq:modecsps} and \ref{eq:moadasps} are reduced to \ref{eq:spsmax}, \ref{eq:decsps} and \ref{eq:adasps} and our convergence analysis recovers the convergence guarantees for SGD showing the tightness of our approach. 
\end{remark}

\vspace{-2mm}
\section{Convergence Analysis}
\vspace{-1mm}
\label{sec:convergence}

In this section, we present the convergence results for \ref{eq:shb} with all of the proposed Polyak step-sizes. For the formal definitions and helpful lemmas, see \Cref{sec:lemmas}. The proofs can be found in \Cref{sec:proofs}. For all of our results, we make the following assumption: 
\vspace{-2mm}
\begin{assumption}[Finite optimal objective difference]
    \begin{equation*}
        \s^2:=\E_{S_t}[f_{S_t}(x^*)-\ell_{S_t}^*]=f(x^*)-\E_{S_t}[\ell_{S_t}^*]<\infty
    \end{equation*}
\end{assumption}

This assumption was first introduced in \citet{loizou2021stochastic} and \citet{orvieto2022dynamics}. Note that $\s^2<\infty$ when all $f_i$ are lower bounded like we have assumed. We say that problem (\ref{eq:main-problem}) is \emph{interpolated} if $\s^2=0$. If the interpolation condition is satisfied then there exists $x^*\in X^*$ such that $f(x^*)=f_{S_t}(x^*)=\ell_{S_t}^*$ for all $S_t\subseteq[n]$. Many modern machine learning models satisfy this condition. Examples include non-parametric regression \citep{liang2020just} and over-parameterized deep neural networks \citep{zhang2021understanding,ma2018power}.

 \vspace{-2mm}
\subsection{Convergence to a neighborhood of the solution}
\vspace{-1mm}

We start with the analysis of \ref{eq:shb} with \ref{eq:mospsmax}. 
\begin{theorem}
    \label{thm:shb-sps-max}
    Assume that each $f_i$ is convex and $L_i$-smooth. Then, the iterates of \ref{eq:shb} with \ref{eq:mospsmax} with $c=1$ and $\b\in\left[0,\frac{\a}{2\g_b-\a}\right)$ where $\a=\min\left\{\frac{1}{2L_{\max}},\g_b\right\}$ and $L_{\max}=\max_i\{L_i\}$, converge as 
    \vspace{-3mm}
    \begin{equation*}
        \E[f(\overline{x}^T)-f(x^*)]\leq\frac{C_1\|x^0-x^*\|^2}{T}+C_2\s^2,
    \end{equation*}
    \vspace{-1mm}
    where $\overline{x}^T=\frac{1}{T}\sum_{t=0}^{T-1}x^t$ and the constants $C_1=\frac{1-\b}{\a\b+\a-2\b\g_b}$ and $C_2=\frac{2\g_b-\a\b-\a}{\a\b+\a-2\b\g_b}$. 
\end{theorem}

Firstly, let us note that both constants $C_1$ and $C_2$ in \Cref{thm:shb-sps-max} are positive since the denominator of both is positive because $\b\in[0,1)$ and the numerator of $C_2$ is positive because $\g_b\geq\a>0$. The above result shows that \ref{eq:mospsmax} has a sublinear convergence to a neighborhood matching the best-known rate for \ref{eq:shb} in the convex setting, see \citet{liu2020improved}. Moreover, when $\b=0$, the update rule of \ref{eq:shb} with \ref{eq:mospsmax} becomes equivalent to SGD with SPS$_{\max}$ (no momentum) and the result of \Cref{thm:shb-sps-max} reduces to convergence guarantees provided in \citet{loizou2021stochastic}, only with a slightly better neighborhood of convergence (here when $\b=0$ we have $C_2=(2\g_b-\a)/\a$ while in the original it is $C_2=2\g_b/\a$).  

In addition, note that in \Cref{thm:shb-sps-max}, there is a restriction in the momentum coefficient. This stems from the technique used in the proof (which forces $C_1$ and $C_2$ to be positive). However, this restriction is not vital as in practical scenarios, the method converges even when $\b$ is selected outside the given interval. 

Moreover, let us highlight that the constants $C_1$ and $C_2$ are increasing when viewed as functions of $\b$ in the given interval (see also \Cref{fig:consts}). This means that our theory suggests that $\b=0$ (no momentum) is the best theoretical choice, which is typical in many works on stochastic methods with momentum \citep{wang2023generalized,loizou2020momentum}. This is not ideal but does not undermine the importance of \Cref{thm:shb-sps-max}, as this is the first result showing the convergence of \ref{eq:shb} with SPS in a non-interpolated setting. 

From \Cref{thm:shb-sps-max}, we can also deduce rates for the deterministic and interpolated regimes. 

\begin{corollary}[Deterministic Heavy Ball Method]
    \label{cor:hb-ps-max}
    Assume that $f$ is convex and $L$-smooth. Then HB with $\b\in[0,\a/(2\g_b-\a))$ and 
    \begin{equation}
        \label{eq:mopsmax}
        \g_t=(1-\b)\min\left\{\frac{f(x^t)-f(x^*)}{\|\nabla f(x^t)\|^2},\g_b\right\},\tag{MomPS$_{\max}$}
    \end{equation}
    where $\a=\min\left\{\frac{1}{2L},\g_b\right\}$, converges as $\min_{t\in[T]}\{f(x^t)-f(x^*)\}\leq\frac{(1-\b)\|x^0-x^*\|^2}{(\a\b+\a-2\b\g_b)T}$.
\end{corollary}

Note that in the setting of \Cref{cor:hb-ps-max}, we have used $f(x^*)$ instead of a lower bound $\ell^*$. This is in accordance with previous work in the deterministic setting \citep{polyak1987introduction}. We highlight that in \citet{hazan2019revisiting} (PS) and \citet{wang2023generalized} (ALR-MAG), there is no upper bound in the step-size for the deterministic setting. This is due to the fact that in these works, the step-size is derived as a minimizer of an upper bound of the quantity $\|x^t-x^*\|^2$. Our step-size \ref{eq:mopsmax} is \emph{not} a minimizer of such bound when $\b>0$. Nevertheless, when $\b=0$ (no momentum), by setting $\g_b=\infty$ we recover the corresponding result in \citet{hazan2019revisiting}. In our setting, we cannot set $\g_b=\infty$ when $\b>0$ because the term $\b\g_b$ appears in the denominator of our convergence result. 

\begin{corollary}[Interpolation]
    \label{cor:shb-sps-max-inter}
    Assume interpolation ($\s^2=0$) and let all assumptions of \Cref{thm:shb-sps-max} be satisfied. Then \ref{eq:shb} with \ref{eq:mospsmax} and $\b\in[0,\a/(2\g_b-\a))$ where $\a=\min\left\{\frac{1}{2L_{\max}},\g_b\right\}$ converges as $\E[f(\overline{x}^T)-f(x^*)]\leq\frac{(1-\b)\|x^0-x^*\|^2}{(\a\b+\a-2\b\g_b)T}$, where $\overline{x}^T=\frac{1}{T}\sum_{t=0}^{T-1}x^t$. 
\end{corollary}

In the closely related works \citet{loizou2021stochastic} (SPS$_{\max}$) and \citet{wang2023generalized} (ALR-SMAG), there is no upper bound in the proposed step-sizes for the interpolated setting. In our analysis, we focus on the fully stochastic setting (interpolation is only a corollary), and the result of \Cref{cor:shb-sps-max-inter} shows a saddle connection between the momentum parameter and the parameter $\g_b$ (we cannot select $\g_b=\infty$ when momentum is used). Nevertheless, via our corollary, when $\b=0$ by setting $\g_b=\infty$, we can recover the corresponding result in \citet{loizou2021stochastic}. 

Another Corollary of \Cref{thm:shb-sps-max} is a novel analysis for \ref{eq:shb} with constant step-size. 
\begin{corollary}[Constant Step-size]
    \label{cor:shb-const}
    Let all assumptions of \Cref{thm:shb-sps-max} be satisfied. If $\g_b\leq\frac{1}{2L_{\max}}$, then \ref{eq:shb} with $\g\leq\frac{1-\b}{2L_{\max}}$ and $\b\in[0,1)$ converges as $\E[f(\overline{x}^T)-f(x^*)]\leq\frac{\|x^0-x^*\|^2}{T\g}+\s^2$, where $\overline{x}^T=\frac{1}{T}\sum_{t=0}^{T-1}x^t$. 
\end{corollary}

Note that in \Cref{cor:shb-const}, there is no restriction that depends on the smoothness parameter in the momentum parameter $\b$. We allow to have $\b\in[0,1)$. A similar result is obtained in \citet{liu2020improved}, where the authors establish the convergence rate $O\left(\frac{f(x^0)-f(x^*)}{T}+\frac{L\g\hat{\s}^2}{1-\b}\right)$ for the quantity $\frac{1}{T}\sum_{t=0}^{T-1}\E[\|\nabla f_{S_t}(x^t)\|^2]$ when $\g\leq\frac{(1-\b)^2}{L}\min\left\{\frac{1}{4-\b+\b^2}, \frac{1}{2\sqrt{2\b+2\b^2}}\right\}$. Here $\hat{\s}^2$ is assumed to be a bound of the variance $\E_t\|\nabla f_{S_t}(x^t)-\nabla f(x^t)\|^2\leq\hat{\s}^2$. In comparison, in our analysis, we provide the same asymptotic rate $O(1/T)$ for the common in the convex setting quantity $\E[f(\overline{x}^T)-f(x^*)]$. By comparing the two results, our step-size is larger when $\b\geq\sqrt{5}-2\approx 0.236$. However, in our result, the neighborhood of convergence is constant and does not depend on the step-size $\g$ or the momentum coefficient. For a numerical comparison of \ref{eq:shb} with the two different constant step-sizes, see \Cref{sec:shb-const}.

\vspace{-2mm}
\subsection{Convergence to the exact solution}
\vspace{-1mm}

In this section, we provide theoretical results for the adaptive decreasing step-sizes, \ref{eq:modecsps} and \ref{eq:moadasps}. The main advantage of these decreasing step-sizes is that we can guarantee convergence to the exact solution while keeping the main adaptiveness properties. Due to the nature of the decreasing step-sizes, there is no restriction in the momentum parameter, as was the case with \Cref{thm:shb-sps-max}. For the results of this section, we make the extra assumption of \emph{bounded iterates}, i.e., we assume that $D^2=\max_{t\in[T]}\|x^t-x^*\|^2$ is finite. This is a standard assumption for several adaptive step-sizes, see: \citep{reddi2019convergence,ward2020adagrad,orvieto2022dynamics,jiang2023adaptive}. 

\begin{theorem}
    \label{thm:shb-dec-sps}
    Assume that each $f_i$ is convex and $L_i$-smooth. Let $\overline{x}^T=\frac{1}{T}\sum_{t=0}^{T-1}x^t$, $\a=\min\left\{\frac{1}{2L},\g_b\right\}$ and $D^2=\max_{t\in[T]}\|x^t-x^*\|^2$. Then, the iterates of \ref{eq:shb} with \ref{eq:modecsps} with $c=1$ and $\b\in[0,1)$ converge as: 
    \begin{equation*}
        \E[f(\overline{x}^T)-f(x^*)]\leq\frac{2\b[f(x^0)-f(x^*)]}{(1-\b)T}+\frac{(1+\b)D^2}{(1-\b)\a\sqrt{T}}+\frac{2\s^2}{\sqrt{T}}.
    \end{equation*}
\end{theorem}

Notice that when there is no momentum, i.e., for $\b=0$, \Cref{thm:shb-dec-sps} recovers the original result of SGD with DecSPS from \citet{orvieto2022dynamics}. To our knowledge, the only other result that guarantees convergence to the exact solution for \ref{eq:shb} in the non-interpolated regime appears in \citet{sebbouh2021almost}. In particular, \citet{sebbouh2021almost} proves that for $\g_t=\frac{\h_t}{1+\l_{t+1}}$ and $\b_t=\frac{\l_t}{1+\l_t}$ where $\h_t=\frac{\h}{\sqrt{t+1}}$ with $0\leq\h\leq\frac{1}{4L_{\max}}$ and $\l_t=\frac{1}{2}\sum_{i=0}^{t-1}\sqrt{\frac{t+1}{i+1}}$ with $\l_0=0$, \ref{eq:shb} converges as $\E[f(x^t)-f(x^*)]\leq O(\log t/\sqrt{t})$. This result only holds for a specific choice of $\b_t$, it needs knowledge of $L_{\max}$, and only guarantees $O(\log t/\sqrt{t})$ rate. However, it does not make a bounded iteration assumption, and it shows convergence of the last iterate. In contrast, our convergence of \ref{eq:shb} with \ref{eq:modecsps} holds for any choice of $\b\in[0,1)$, the step-size is adaptive and does not depend on $L_{\max}$, and it achieves a rate of $O(1/\sqrt{k})$. 

Next, we present the theoretical guarantees for \ref{eq:moadasps}. 
\newpage
\begin{theorem}
    \label{thm:shb-ada-sps}
    Assume that each $f_i$ is convex and $L_i$-smooth. Then, the iterates of \ref{eq:shb} with \ref{eq:moadasps} and $\b\in[0,1)$ converge as 
    \begin{equation*}
        \E[f(\overline{x}^T)-f(x^*)]\leq\frac{\t^2}{T}+\frac{\t\s}{\sqrt{T}},
    \end{equation*}
    where $\overline{x}^T=\frac{1}{T}\sum_{t=0}^{T-1}x^t$, $\t=\frac{\b\sqrt{f(x^0)-f(x^*)}}{1-\b}+\frac{(1+\b)cLD^2}{2(1-\b)}+\frac{1}{2c}$ and $D^2=\max_{t\in[T]}\|x^t-x^*\|$. 
\end{theorem}

As in our previous theorems, when there is no momentum ($\b=0$), \Cref{thm:shb-ada-sps} recovers the convergence guarantees of AdaSPS from \citet{jiang2023adaptive}. Let us highlight that our convergence guarantees of \ref{eq:shb} with \ref{eq:moadasps}  offers the first \emph{robust} step-size selection for \ref{eq:shb}, in the sense that it can automatically adapt to the optimization setting. More specifically, for no interpolation, our result has a rate of $O(1/\sqrt{T})$, which matches the best-known rate for \ref{eq:shb} while if we assume interpolation ($\s^2=0$), then we are able to achieve a rate of $O(1/T)$ which matches the best-known rate for \ref{eq:shb} in the convex setting (see \citet{liu2020improved} and \Cref{thm:shb-sps-max}). Furthermore, under interpolation, \Cref{thm:shb-ada-sps} improves the convergence of \Cref{cor:shb-sps-max-inter}. In particular, when $\s^2=0$, \Cref{thm:shb-ada-sps} reaches the rate $O(1/T)$ for \emph{any} $\b\in[0,1)$, while in \Cref{cor:shb-sps-max-inter} this is possible only under tighter restrictions on $\b$.

\vspace{-2mm}
\section{Numerical Experiments}
\label{sec:experiments}
\vspace{-1mm}

In this section, we test our proposed algorithms in deterministic and stochastic convex problems as well as in training popular deep neural networks (DNNs). Our experiments are designed to highlight the benefits of our momentum variants over the vanilla (no-momentum) SPS. 

\begin{wrapfigure}{r}{0.6\textwidth}
\begin{minipage}{0.6\textwidth}
\begin{figure}[H]
	\centering
    \begin{subfigure}{0.45\columnwidth}
		\includegraphics[width=\textwidth]{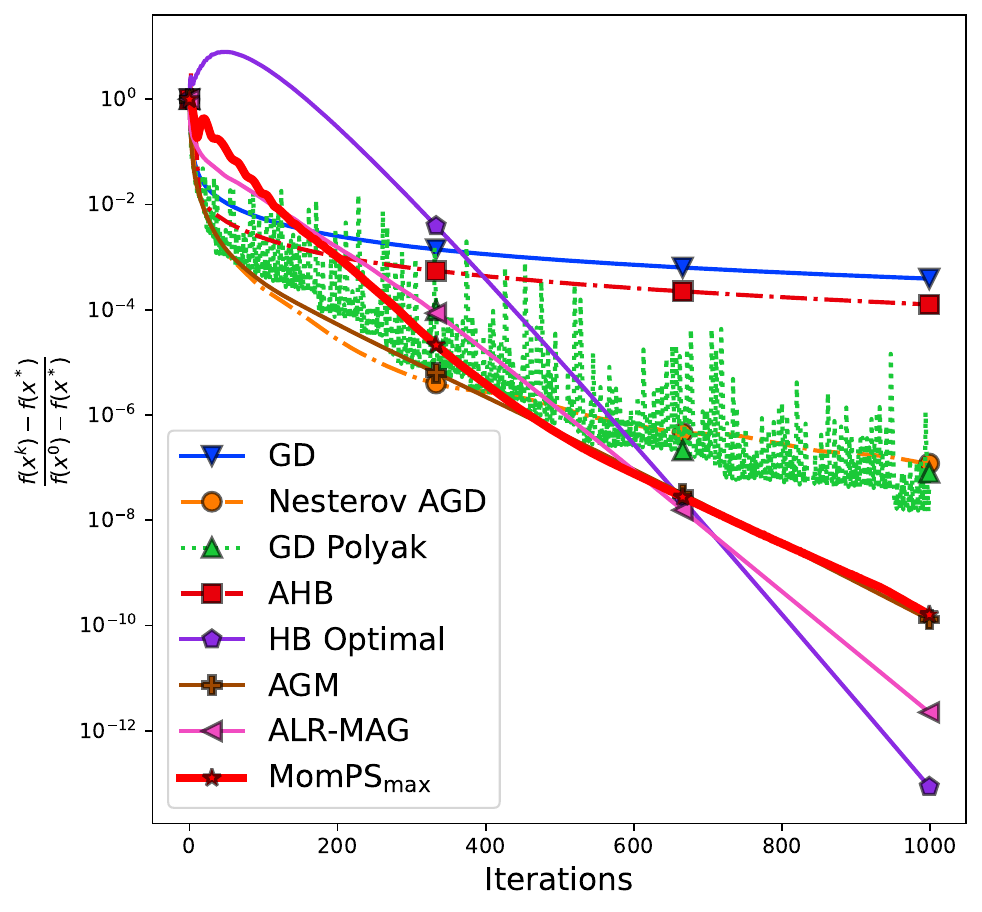}
	\end{subfigure}
    ~
	\begin{subfigure}{0.45\columnwidth}
		\includegraphics[width=\textwidth]{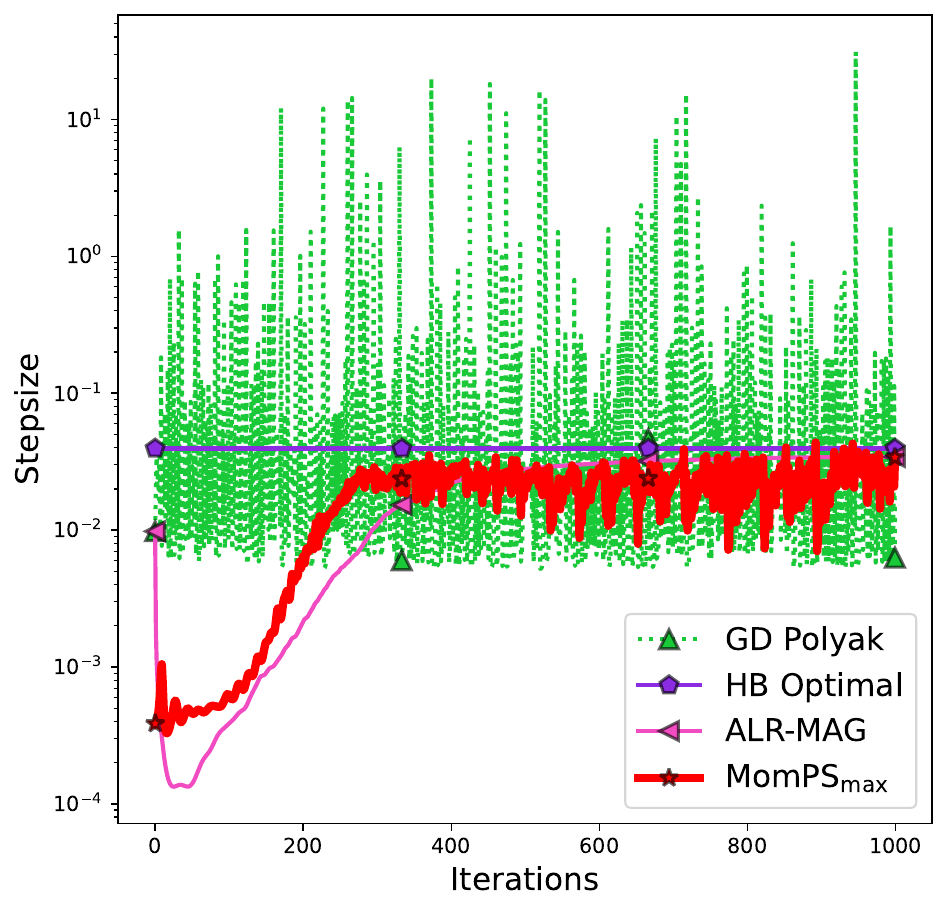}
	\end{subfigure}
	\caption{Comparison of deterministic algorithms for least squares. Left: Relative error, Right: Step-sizes.}
	\label{fig:det_ls}
\end{figure}
\end{minipage}
\end{wrapfigure}

\textbf{Deterministic Setting.}
For the deterministic setting, we focus on the least squares problem. The loss function in this case is given by $f(x)=\frac{1}{2}\|Ax-b\|^2$, where $A\in\R^{n\times d}$ and $b\in\R^n$. In our experiments, we follow the setting of \citet{wang2023generalized} and choose $n=d=1000$ while the matrix $A$ has been generated according to \citet{lenard1984randomly} such that the condition number of $A^TA$ is $10^4$. We test the following algorithms: Gradient Descent (GD), Nesterov's Accelerated Gradient Descent (AGD) \citep{nesterov2018lectures}, Gradient Descent with Polyak's step-size \citep{polyak1987introduction}, Adaptive Heavy Ball (AHB) \citep{saab2022adaptive}, Polyak's Heavy Ball (HB) \citep{polyak1964some}, Accelerated Gradient Method (AGM) \citep{barre2020complexity}, ALR-MAG \citep{wang2023generalized} and finally our step-size from \Cref{cor:hb-ps-max}, called (\ref{eq:mopsmax}), which guarantees $O(1/T)$ rate. For the least squares problem, there is an established theory from \citet{polyak1964some} that guarantees acceleration of HB with optimal step-size and momentum coefficient given by $\b^*=(\sqrt{L}-\sqrt{\m})^2/(\sqrt{L}+\sqrt{\m})^2$ and $\g^*=(1+\sqrt{\b^*})/L$. For ALR-MAG and \ref{eq:shb} with \ref{eq:mopsmax}, we have used the optimal momentum $\b^*$ for direct comparison with optimal HB. Moreover, for GD Polyak, ALR-MAG, and \ref{eq:mopsmax}, we have used the precomputed $f(x^*)$ (via GD), and for \ref{eq:mopsmax}, we have chosen $\g_b=100$ (for other choices, see \Cref{sec:upper-bound-need}). The results in \Cref{fig:det_ls} show that \ref{eq:mopsmax} has performance comparable to AGM and ALR-MAG which are faster than vanilla Polyak step-size after the initial iterations. Furthermore, to explore the behavior of the various adaptive step-sizes, we plot the Polyak, the ALR-SMAG, the \ref{eq:mopsmax}, and the optimal HB step-sizes. We can see that the original Polyak step-size oscillates around the optimal HB step-size while \ref{eq:mopsmax} converges from below to the optimal. 

\textbf{Stochastic Setting: Convex Problems.} Here, we compare our step-sizes with previous works in the stochastic setting. As noticed initially for SPS$_{\max}$~\citep{loizou2021stochastic}, the value of the upper bound $\g_b$ that results in good convergence depends on the problem and requires careful parameter tuning. To alleviate this problem, \citet{loizou2021stochastic} uses a smoothing procedure that prevents large fluctuations in the step-size across iterations. That is,  $\g_b^t=\t^{b/n}\g_{t-1}$ for each iteration $t$ where $\t=2$, $b$ is the batch-size, and $n$ is the number of examples. We use the same smoothing trick for $\g_b$ in the implementation of the proposed methods. Similar smoothing procedures have been used in \citet{tan2016barzilai,vaswani2019painless}. 

We considered multi-class logistic regression applied to commonly used benchmark datasets from the LIBSVM repository \citep{chang2011libsvm}. We separate our experiments between two classes of step-sizes: non-decreasing step-sizes (which might guarantee convergence to a neighborhood of the solution) and decreasing step-sizes (which guarantee convergence to the exact solution).  For the first class of step-sizes we test the following algorithms: SGD \citep{nesterov2018lectures}, \ref{eq:shb} with the constant step-size and $\b=0.9$, ADAM as described in \citet{kingma2014adam}, the warm-up version of ALR-SMAG with the hyper-parameters suggested in \citet{wang2023generalized}, 
the SPS$_{\max}$, the SPS$_{\max}$ with naive momentum as described in \Cref{sec:naive-sps} and our proposed \ref{eq:mospsmax}. For the Polyak-based step-sizes such as ALR-SMAG, SPS$_{\max}$, and \ref{eq:mospsmax}, we select $\ell_i^*=0$ and $c=1$. For SHB, ALR-SMAG, \ref{eq:mospsmax},  SPS$_{\max},$ with naive momentum,  we use momentum $\b=0.9$. 
All the convex experiments are run for 100 epochs and for 5 trials. We plot the average of the trials and the standard deviation. Since there are no standard train/test splits, and due to the small sizes of the datasets, we present training loss and accuracy curves only. We present the comparison of these methods in \Cref{fig:m_logreg_set_vowel_bs_52_e_100}. In all experiments, we observe that the performance of our proposed \ref{eq:mospsmax} significantly outperforms the other methods in both training loss and test accuracy. For the decreasing variants, we test AdaGrad-Norm \citep{duchi2011adaptive}, SGD with \ref{eq:decsps}, SGD with \ref{eq:adasps}, and our proposed \ref{eq:shb} with \ref{eq:modecsps} and \ref{eq:shb} with \ref{eq:moadasps}. We present the outcome of this comparison in \Cref{fig:dec_m_logreg_set_letter_bs_1500_e_100}. \ref{eq:modecsps} and \ref{eq:moadasps} outperform their no-momentum counterparts in both training loss and test accuracy. 

\vspace{-2mm}
\begin{minipage}{0.45\textwidth}
\begin{figure}[H]
	\centering
	\includegraphics[width=\textwidth]{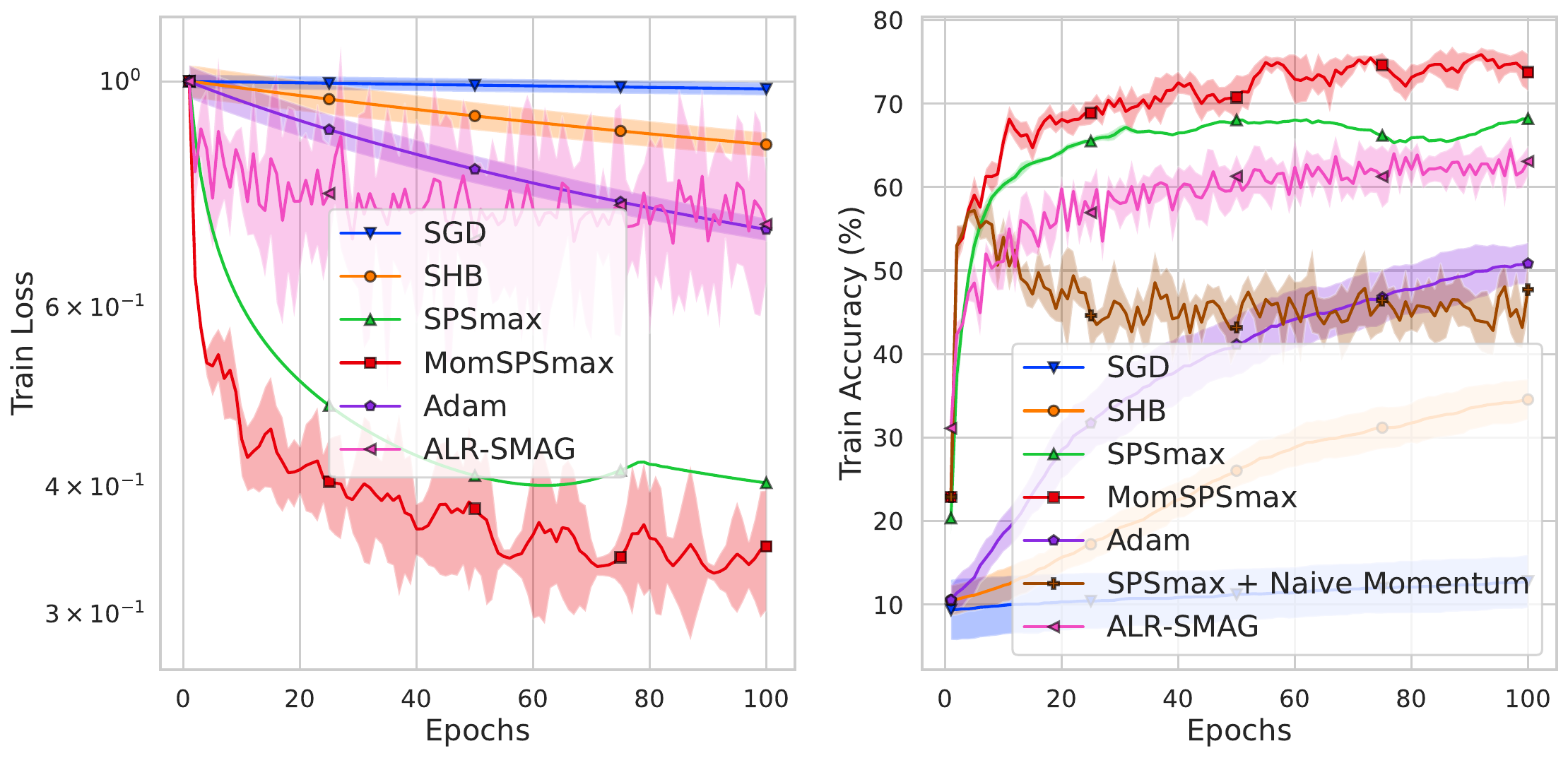}
	\caption{\small LibSVM dataset: vowel, Batch size: 52}
	\label{fig:m_logreg_set_vowel_bs_52_e_100}
\end{figure}
\end{minipage}
~
\begin{minipage}{0.45\textwidth}
\begin{figure}[H]
	\centering
	\includegraphics[width=\textwidth]{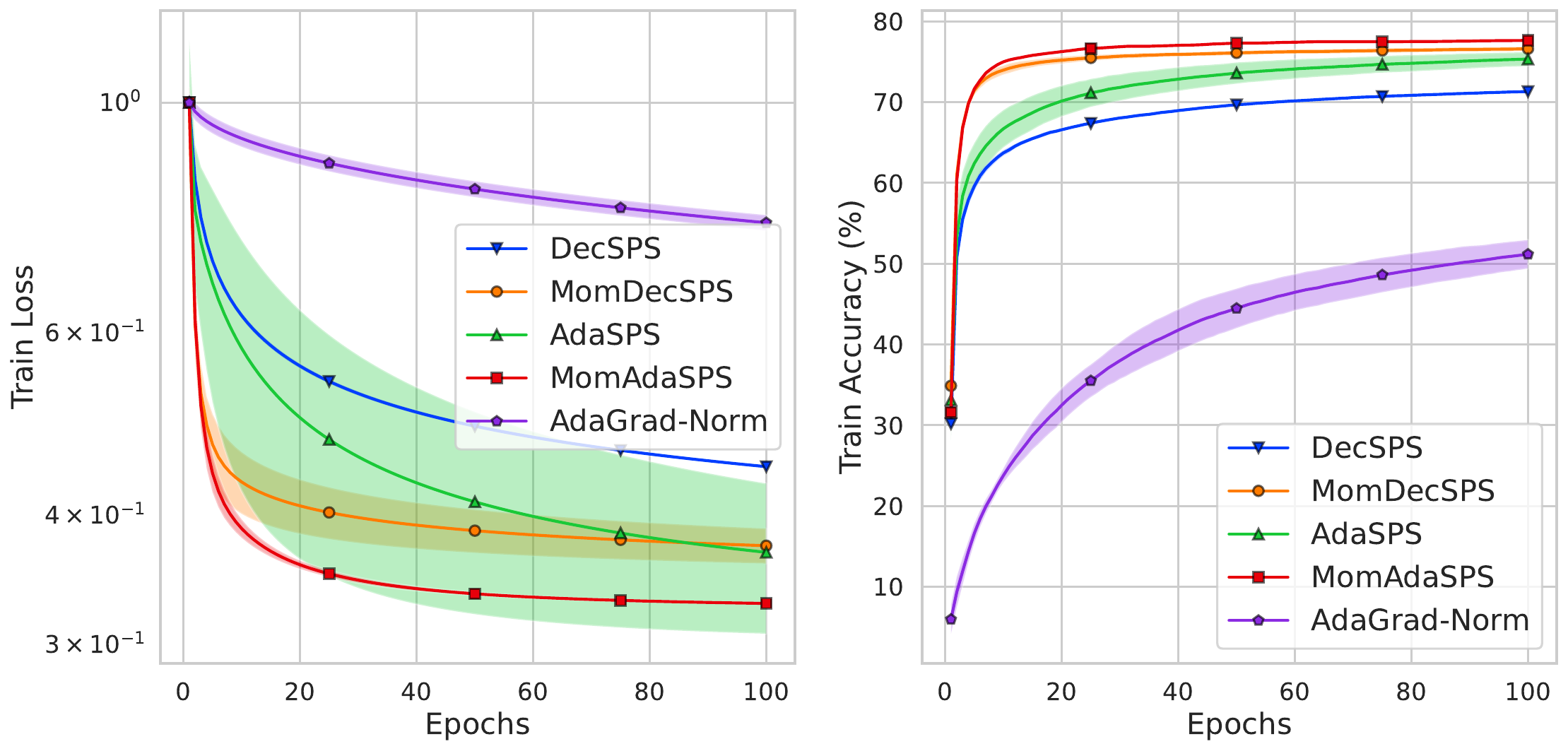}
	\caption{\small LibSVM dataset: letter, Batch size: 1500}
	\label{fig:dec_m_logreg_set_letter_bs_1500_e_100}
\end{figure}
\end{minipage}

\vspace{-2mm}
\textbf{MomSPS for DNNs.} In our final experiment, we go beyond the convex setting of our convergence guarantees, and we test \ref{eq:mospsmax} on the training of DNNs. We consider non-convex minimization for multi-class classification using deep network models on the CIFAR 10 and CIFAR 100 datasets \citep{krizhevsky2009learning}. We use two standard image-classification architectures: ResNet18 and ResNet34, \citep{he2016deep}. For space concerns, we report only the ResNet34 experiments in the main paper and relegate the ResNet18 to the Appendix. In these experiments we run \ref{eq:mospsmax} with $c=0.4$. We present the results for CIFAR 10 in \Cref{fig:resnet34_set_cifar10_bs_256_e_100} and for CIFAR 100 in \Cref{fig:hresnet34_set_cifar100_bs_256_e_100}. We observe that \ref{eq:mospsmax} consistently outperforms its no-momentum counterpart SPS$_{\max}$ and has competitive generalization performance compared to other popular optimizers. 

\vspace{-3mm}
\begin{minipage}{0.45\textwidth}
\begin{figure}[H]
	\centering
	\includegraphics[width=\textwidth]{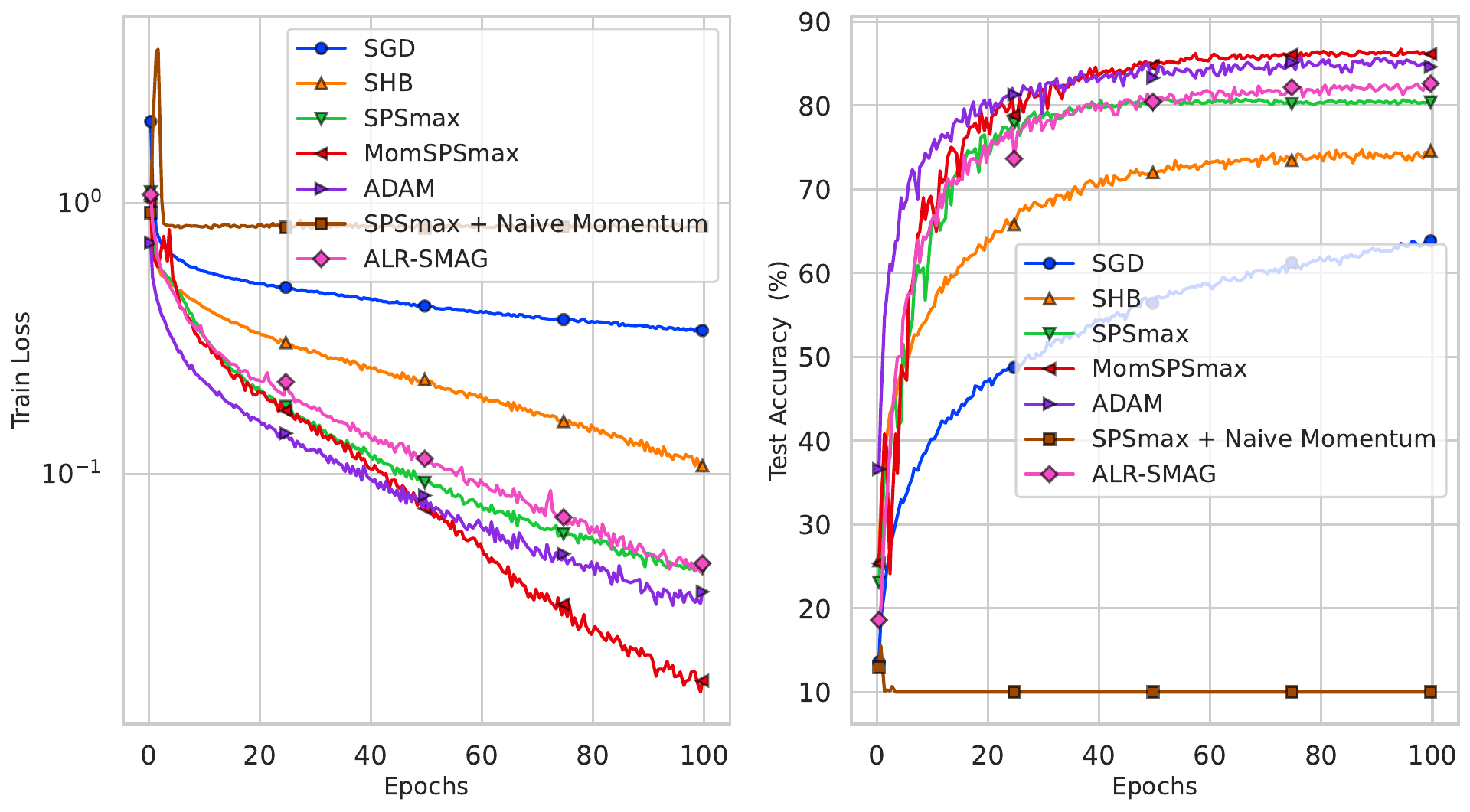}
    \caption{Resnet 34 on CIFAR 10}
    \label{fig:resnet34_set_cifar10_bs_256_e_100}
\end{figure}
\end{minipage}
~
\begin{minipage}{0.45\textwidth}
\begin{figure}[H]
	\centering
	\includegraphics[width=\textwidth]{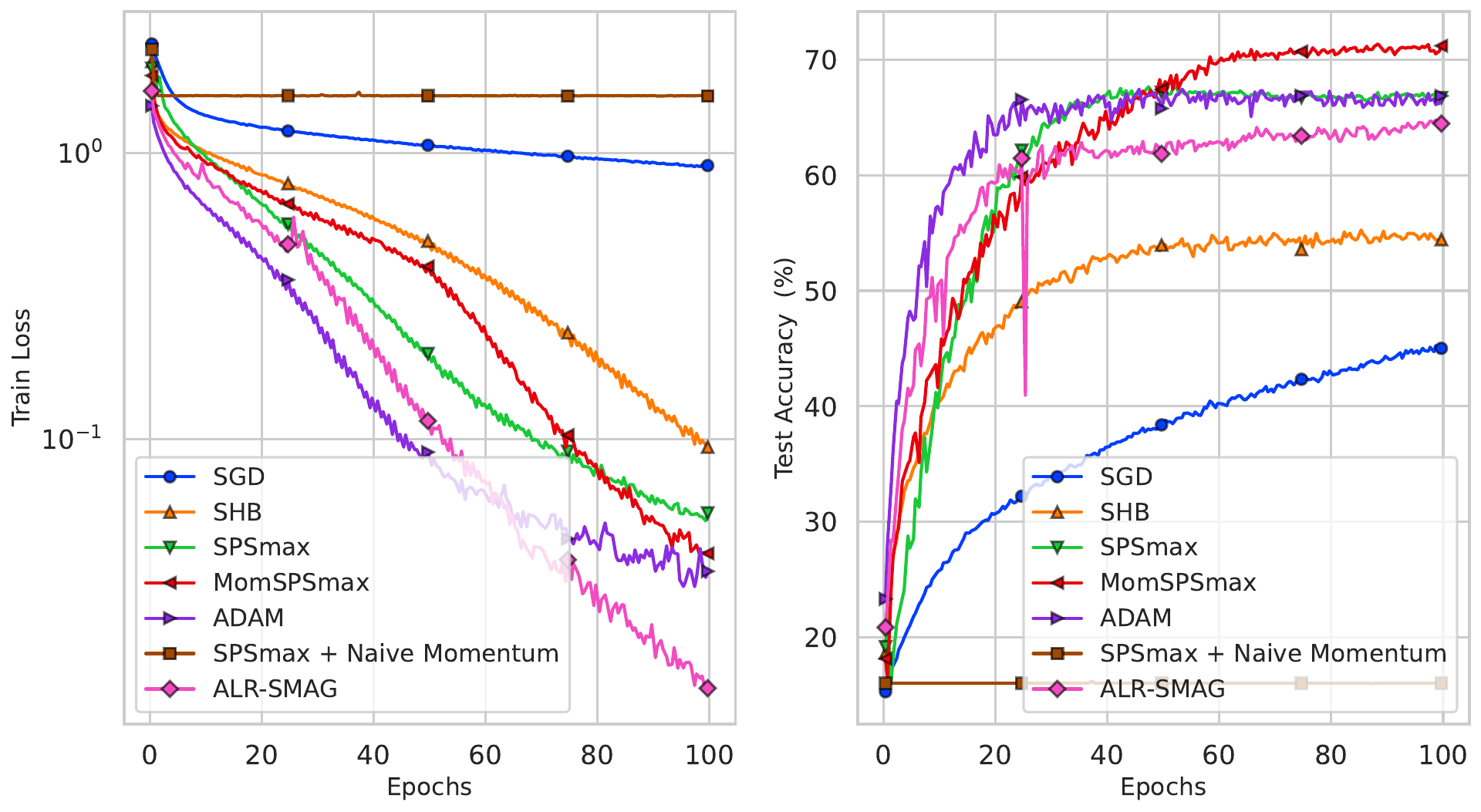}
    \caption{ResNet 34 on CIFAR 100}
	\label{fig:hresnet34_set_cifar100_bs_256_e_100}
\end{figure}
\end{minipage}

\newpage
\bibliographystyle{plainnat}
\bibliography{main}

\newpage
\appendix

\newpage
\part*{Supplementary Material}

The Supplementary Material is organized as follows: In \Cref{sec:furth-related-work}, we have more details on related work on adaptive and momentum methods. In \Cref{sec:pseudocode}, we include the main pseudo-codes for our algorithms. In \Cref{sec:lemmas}, we give the basic definitions and lemmas as well as the basic theory of IMA. \Cref{sec:proofs} presents the proofs of the theoretical guarantees from the main paper. In \Cref{sec:add-exps}, we describe in detail our experimental setup and provide additional experiments. 

{\small\tableofcontents}

\newpage
\section{Further Related Work}
\label{sec:furth-related-work}

\paragraph{Adaptive methods.}
Two of the first adaptive algorithms are AdaGrad \citep{duchi2011adaptive} and RMSProp \citep{hinton2012neural}. The very popular algorithm Adam \citep{kingma2014adam} was introduced as a momentum extension of RMSProp, and its more recent weight-decay variant AdamW \citep{loshchilov2017decoupled} is the de facto optimizer for deep neural networks. Investigating the convergence guarantees of adaptive methods across various settings continues to be an active area of research \citep{vaswani2020adaptive, ward2020adagrad,li2019convergence, shi2022ai, defossez2020simple,choudhury2024remove,defazio2022momentumized}. 

\medskip
\noindent
More recently, a new line of work for adaptive algorithms has appeared, inspired by Polyak step-sizes. Some attempts for efficiently generalizing the Polyak step-size from the deterministic setting to the stochastic were made in \citet{rolinek2018l4,oberman2019stochastic,berrada2020training}. \citet{loizou2021stochastic} was the first work proposing \ref{eq:spsmax} and providing strong convergence guarantees in different settings, including strongly convex, convex, and non-convex functions. Further extensions with decreasing variants were proposed in \citet{orvieto2022dynamics,jiang2023adaptive}. Many recent works propose SPS-type update rules for solving optimization problems in different settings. For example, extensions of SPS to proximal setting analyzed in \citet{schaipp2023stochastic}, SPS variants for mirror descent presented in \citet{d2023stochastic}, and SPS for federated learning proposed in \citet{mukherjee2023locally}. There are also strong connections between SPS and Model-Based approaches \citep{asi2019importance,asi2019stochastic,chadha2022accelerated}. For further results related to SPS, see \citet{gower2022cutting,li2022sp2,garrigos2023function} and \citet{abdukhakimov2023sania}.

\paragraph{Momentum with constant step-size.}
\citet{polyak1964some} introduced the concept of momentum for the Gradient Descent (GD) method. He showed that for strongly convex quadratic problems, the momentum method provably accelerates GD with optimal momentum coefficient $\b^*=(\sqrt{L}-\sqrt{\m})^2/(\sqrt{L}+\sqrt{\m})^2$ and $\g^*=(1+\sqrt{\b^*})^2/L$ where $\m$ is the strong convexity constant and $L$ the smoothness constant. \citet{ghadimi2015global} proved a global sublinear convergence guarantee for HB for convex and smooth functions and global linear convergence when the function is also smooth and strongly convex. Moreover, \citet{nesterov1983method} introduced Nesterov's Accelerated Gradient Descent (AGD) method, where he shows acceleration over GD for strongly convex and smooth as well as convex and smooth objectives. Note that both of these methods require knowledge of both $\m$ and $L$. In the stochastic setting, \citet{liu2020improved} provides new convergence guarantees for \ref{eq:shb} with constant step-size. In particular, it shows that for both strongly convex and non-convex objectives, \ref{eq:shb} enjoys the same convergence bound as SGD. \citet{sebbouh2021almost}, among other things, shows that in the smooth and convex setting \ref{eq:shb} converges in expectation at the last iterate with a rate of $O(1/T)$ to a neighborhood of the minimum and at a $O(\log T/\sqrt{T})$ rate to the minimum exactly. For these guarantees, one needs knowledge of the problem parameters (such as the smoothness constant), and their results hold for a specific selection of the momentum parameter. Momentum has also been successfully applied in linear systems, see \citet{loizou2020momentum} where it provides convergence guarantees of SGD, stochastic Newton, stochastic proximal point, and stochastic dual subspace ascent with momentum for consistent linear systems. 

\paragraph{Adaptive Momentum.}
In this paragraph, we review the bibliography for adaptive methods with momentum. Momentum can be either SHB or Nesterov's acceleration, and in this case, the adaptivity can be either the step-size or the momentum coefficient. 
In the deterministic setting, \citet{barre2020complexity} proposed the Accelerated Gradient Method (AGM), an adaptive algorithm built upon AGD that guarantees acceleration over GD. AGM approximates the strong convexity constant $\m$ by the inverse of the classical Polyak Step-size. However, their algorithm still requires the knowledge of $L$ (but not $\m$). Furthermore, \citet{saab2022adaptive} suggests an Adaptive Heavy Ball (AHB) algorithm where it approximates the constants $\m$ and $L$ iteratively in each step based on the formula of the optimal constants of HB, but it does not show any acceleration over HB or GD. In the stochastic setting, \citet{wang2023generalized} propose an adaptive algorithm for a moving averaged gradient momentum, named ALR-SMAG. It shows that ALR-SMAG enjoys a linear convergence rate for semi-strongly convex and smooth functions. \citet{schaipp2023momo} proposes a new adaptive learning rate that can be combined with any momentum-based method. Finally, \citet{zeng2023adaptive} proposes adaptive variants for \ref{eq:shb} in linear systems solvers. 

\paragraph{On Technical Assumptions for Convergence.}
In the literature of stochastic optimization problems, there are several assumptions on the noise of the stochastic estimators that typically are made on top of smoothness and convexity to prove the convergence of stochastic optimization algorithms. 

\medskip
\noindent
For example, many works \citep{recht2011hogwild,rakhlin2012making,shamir2013stochastic,nguyen2018sgd} assume bounded gradients, i.e., that is, there exists a $M\in \R$ such that $\E\|\nabla f_i(x)\|^2\leq M$. While this might look like a natural assumption, in the unconstrained setting, it contradicts the assumption of strong convexity leading to convergence guarantees that hold for an empty set of problems \citep{nguyen2018sgd,gower2019sgd,gower2021sgd}. A much more relaxed assumption used in the literature is the growth condition on the stochastic gradients, \citep{bertsekas1996neuro,schmidt2013fast,vaswani2019fast}. It states that there exist constants $\r\in\R$ and $\d\in\R$ such that $\E\|\nabla f_i(x)\|^2\leq\r\|\nabla f(x)\|^2+\d$. Based on a strong growth condition $(\d=0)$, \citet{schmidt2013fast} were the first to establish linear convergence of SGD, with \citet{vaswani2019fast} later showing that SGD can find a first-order stationary point as efficiently as full gradient descent in non-convex settings. Similar conditions have also been proved and used in the analysis of decentralized variants of SGD \citep{koloskova2020unified,assran2019stochastic}. More recently, a line of work that uses smoothness (via expected smoothness/residual conditions) to provide closed-form expressions for the values of $\r$ and $\d$ of growth condition was able to provide tight convergence analysis of several stochastic algorithms, including SGD \citep{gower2019sgd,gower2021sgd}, variance reduced methods \citep{khaled2023unified}, stochastic algorithms for min-max optimization \citep{loizou2021stochastic,loizou2020stochastic,gorbunov2022stochastic,choudhury2024single}.

\medskip
\noindent
In the convex regime, the original analysis of SGD with SPS of \citet{loizou2021stochastic} was one of the first papers that did not require any additional assumptions to guarantee convergence for SGD. Following the convergence guarantees of \citet{loizou2021stochastic},  we highlight that our proposed analysis of \ref{eq:shb} with Polyak step-size does not require any additional assumptions for guaranteeing convergence. To the best of our knowledge, as mentioned in the main paper, our approach provides the first analysis of \ref{eq:shb} without the restrictive bounded variance and growth conditions. 

\medskip
\noindent
For proving convergence of \ref{eq:shb}, \citet{liu2020improved} makes the strong assumption of bounded variance and assumes knowledge of the smoothness $L$ for tuning the parameters. Another work on theoretical guarantees for \ref{eq:shb} is \citet{sebbouh2021almost}, where a last iterate convergence is proved for a specific choice of $\g_t$ and $\b_t$, but also requires knowledge of $L$ and only achieves rate $O(\log T/\sqrt{T})$ for the decreasing step-size variant. Furthermore, the authors of \citet{wang2022provable} propose a Polyak-based stepsize for a moving averaged gradient momentum variant, different from \ref{eq:shb}, and provide guarantees for the deterministic and stochastic setting for strongly convex objectives. Finally, \citet{schaipp2023momo} assumes a stochastic convex and \emph{interpolated} regime for their results.

\newpage
\section{Pseudo-codes}
\label{sec:pseudocode}

In this section, we include the pseudo-codes of the \ref{eq:shb} with the proposed stochastic Polyak step-sizes for a better understanding of the methods and easier comparison with other algorithms. 

\paragraph{SHB with MomSPSmax.} We start with the pseudo-code for \ref{eq:shb} with \ref{eq:mospsmax}. In \Cref{thm:shb-sps-max} we established convergence to a neighborhood for $\b\in\left[0,\frac{\a}{2\g_b-\a}\right)$ where $\a=\min\left\{\frac{1}{2L_{\max}},\g_b\right\}$ and for $c=1$. However, in practice, it seems to have convergence even for different choices of $c$ or even if $\b$ is chosen outside of this bound, so we provide the pseudo-code for a general $\b\in[0,1)$ and $c$. Thus, the user must provide the momentum coefficient $\b$, the upper bound $\g_b$, the constant $c$, and the lower bounds $\ell_{S_t}^*$. In practice, the usual choices are $\b=0.9$, $c=1$, and $\ell_{S_t}^*=0$. 

\begin{algorithm}[H]
    \caption{\ref{eq:shb} with \ref{eq:mospsmax}}
    \begin{algorithmic}[1]
        \State \textbf{Parameters:} $\b\in[0,1)$, $\g_b>0$, $c>0$, $\ell_{S_t}^*$ lower bounds
        \State \textbf{Initialization:} $x^0\in\R^d$, $x^{-1}=x^0$
        \For{$t=0,1,2,\dots$}
            \State Choose uniformly at random $S_t\subseteq\{1,\dots,n\}$
            \State $\g_t=(1-\b)\min\left\{\frac{f_{S_t}(x^t)-\ell_{S_t}^*}{c\|\nabla f_{S_t}(x^t)\|^2},\g_b\right\}$
            \State $x^{t+1}=x^t-\g_t\nabla f_{S_t}(x^t)+\b(x^t-x^{t-1})$
        \EndFor
    \end{algorithmic}
    \label{alg:mospsmax}
\end{algorithm}

\paragraph{SHB with MomDecSPS.}
For the decreasing step-sizes, \Cref{thm:shb-dec-sps,thm:shb-ada-sps} make no assumption on $\b\in[0,1)$. For \ref{eq:shb} with \ref{eq:modecsps}, the user needs to provide the momentum coefficient $\b$, the upper bound $\g_b$, the constant $c$ and the lower bounds $\ell_{S_t}$, like the previous algorithm. In practice, the usual choices are $c=1$ and $\ell_{S_t}^*=0$. Note that when $t=0$, the step-size is equal to 
\begin{align*}
    \g_0=\min\left\{\frac{(1-\b)[f_{S_0}(x^0)-\ell_{S_0}^*]}{c\|\nabla f_{S_0}(x^0)\|^2},\frac{\g_bc}{c}\right\}=(1-\b)\min\left\{\frac{f_{S_0}(x^0)-\ell_{S_0}^*}{c\|\nabla f_{S_0}(x^0)\|^2},\g_b\right\},
\end{align*}
which is equal to \ref{eq:mospsmax} for $t=0$. 

\begin{algorithm}[H]
    \caption{\ref{eq:shb} with \ref{eq:modecsps}}
    \begin{algorithmic}[1]
        \State \textbf{Parameters:} $\b\in[0,1)$, $\g_{-1}=\g_b>0$, $c>0$, $c_{-1}=c>0$, $c_t=c\sqrt{t+1}$ for $t\geq0$, $\ell_{S_t}^*$ lower bounds
        \State \textbf{Initialization:} $x^0\in\R^d$, $x^{-1}=x^0$
        \For{$t=0,1,2,\dots$}
            \State Choose uniformly at random $S_t\subseteq\{1,\dots,n\}$
            \State $\g_t=\min\left\{\frac{(1-\b)[f_{S_t}(x^t)-\ell_{S_t}^*]}{c_t\|\nabla f_{S_t}(x^t)\|^2},\frac{\g_{t-1}c_{t-1}}{c_t}\right\}$
            \State $x^{t+1}=x^t-\g_t\nabla f_{S_t}(x^t)+\b(x^t-x^{t-1})$
        \EndFor
    \end{algorithmic}
    \label{alg:modecsps}
\end{algorithm}

\paragraph{SHB with MomAdaSPS.}
Finally, for \ref{eq:shb} with \ref{eq:moadasps} the requirements are less since the user only needs to provide the momentum coefficient $\b$, the constant $c$ and the lower bounds $\ell_{S_t}^*$. Furthermore, according to \citet{jiang2023adaptive} one can set $c=\frac{1}{\sqrt{f_{S_0}(x^0)-\ell_{S_0}^*}}$ after randomly choosing $S_0\subseteq[n]$ in the first iteration. Moreover, the minimum with respect to $\g_{t-1}$ is to ensure that the step-size is decreasing. 

\begin{algorithm}[H]
    \caption{\ref{eq:shb} with \ref{eq:moadasps}}
    \begin{algorithmic}[1]
        \State \textbf{Parameters:} $\b\in[0,1)$, $\g_{-1}=+\infty$, $c>0$, $\ell_{S_t}^*$ lower bounds
        \State \textbf{Initialization:} $x^0\in\R^d$, $x^{-1}=x^0$
        \For{$t=0,1,2,\dots$}
            \State Choose uniformly at random $S_t\subseteq\{1,\dots,n\}$
            \State $\g_t=\min\left\{\frac{(1-\b)[f_{S_t}(x^t)-\ell_{S_t}^*]}{c\|\nabla f_{S_t}(x^t)\|^2\sqrt{\sum_{s=0}^tf_{S_s}(x^s)-\ell_{S_s}^*}},\g_{t-1}\right\}$
            \State $x^{t+1}=x^t-\g_t\nabla f_{S_t}(x^t)+\b(x^t-x^{t-1})$
        \EndFor
    \end{algorithmic}
    \label{alg:moadasps}
\end{algorithm}

\medskip
\noindent
Note that for all the above algorithms, one needs only lower bounds $\ell_{S_t}^*$, which can be chosen equal to $\ell_{S_t}^*=0$ in practice. Of course, if the true infima $f_{S_t}^*=\inf_{x\in\R^d}f_{S_t}$ are known, then one should use them because it will lead to better convergence in terms of the neighborhood. Furthermore, notice that we start by initializing $x^{-1}=x^0$, where $x^0$ is the given starting point. This is because the first step of \ref{eq:shb} (i.e. for $t=0$) is just an iteration of SGD because we have no previous iterate $x^{-1}$, that is $x^1=x^0-\g_0\nabla f_{S_0}(x^0)+\b(x^0-x^{-1})=x^0-\g_0\nabla f_{S_0}(x^0)$.

\newpage
\section{Technical Preliminaries}
\label{sec:lemmas}

\subsection{Basic Definitions}

In this section, we present some basic definitions we use throughout the paper. 

\begin{definition}[Convexity]
    A differentiable function $f:\R^d\to\R$ is convex if 
    \begin{align*}
        f(x)\geq f(y)+\langle\nabla f(y),x-y\rangle,
    \end{align*}
    for all $x,y\in\R^d$. 
\end{definition}

\begin{definition}[$L$-smooth]
    A differentiable function $f:\R^d\to\R$ is $L$-smooth if there exists a constant $L>0$ such that 
    \begin{align*}
        \|\nabla f(x)-\nabla f(y)\|\leq L\|x-y\|,
    \end{align*}
    or equivalently 
    \begin{align*}
        |f(x)-f(y)-\langle\nabla f(y),x-y\rangle|\leq\frac{L}{2}\|x-y\|^2
    \end{align*}
    for all $x,y\in\R^d$. 
\end{definition}

\subsection{Basic Lemmas}

Here we have gathered useful lemmas for various Polyak related step-sizes. These lemmas are used repeatedly in the proofs of the main results. 

\begin{lemma}[{\citep{loizou2021stochastic}}]
    \label{lem:sps-bound}
    Suppose each $f_i$ is $L_i$-smooth, then the step-size of SPS$_{\max}$ satisfies: 
    \begin{align}
        \label{eq:lem:sps-bound}
        \a=\min\left\{\frac{1}{2cL_{\max}},\g_b\right\}\leq\g_t\leq\g_b\text{ and }\g_t^2\|\nabla f_{S_t}(x^t)\|^2\leq\g_t[f_{S_t}(x^t)-\ell_{S_t}^*],
    \end{align}
    where $L_{\max}=\max_{i\in[n]}\{L_i\}$. 
\end{lemma}

\begin{lemma}[Lemma 1, \citep{orvieto2022dynamics}]
    \label{lem:decsps-bound}
    Suppose each $f_i$ is $L_i$-smooth, then the step-size of DecSPS satisfies 
    \begin{align}
        \label{eq:lem:decsps-bound}
        \min\left\{\frac{1}{2c_tL_{\max}},\frac{c_0\g_b}{c_t}\right\}\leq\g_t\leq\frac{c_0\g_b}{c_t}\text{ and }\g_{t-1}\leq\g_t\text{ and }\g_t^2\|\nabla f_{S_t}(x^t)\|^2\leq\frac{\g_t}{c_t}[f_{S_t}(x^t)-\ell_{S_t}^*],
    \end{align}
    where $L_{\max}=\max_{i\in[n]}\{L_i\}$. 
\end{lemma}

\begin{lemma}[Lemma 12, {\citep{jiang2023adaptive}}]
    \label{lem:adasps-trick}
    For any non-negative sequence $a_0,\dots,a_T$ , the following holds: 
    \begin{align}
        \label{eq:lem:adasps-trick-1}
        \sqrt{\sum_{t=0}^Ta_t}\leq\sum_{t=0}^T\frac{a_t}{\sqrt{\sum_{s=0}^ta_s}}
        \leq2\sqrt{\sum_{t=0}^Ta_t}.
    \end{align}
    If $a_0\geq1$, then the following holds: 
    \begin{align}
        \label{eq:lem:adasps-trick-2}
        \sum_{t=0}^T\frac{a_t}{\sqrt{\sum_{s=0}^ta_s}}\leq\log\left(\sum_{t=0}^T
        a_t\right)+1.
    \end{align}
\end{lemma}

\begin{lemma}[Lemma 13, {\citep{jiang2023adaptive}}]
    \label{lem:adasps-finish}
    If $x^2\leq a(x+b)$ for $a\geq0$ and $b\geq0$, then it holds that $x\leq a+
    \sqrt{ab}$. 
\end{lemma}

\begin{lemma}[Lemma 16, {\citep{jiang2023adaptive}}]
    \label{lem:adasps-bounds}
    Suppose each $f_i$ is $L_i$-smooth, then the step-size of AdaSPS satisfies 
    \begin{align}
        \label{eq:lem:adasps-bounds}
        \frac{1}{2cL_{\max}}\frac{1}{\sqrt{\sum_{s=0}^tf_{S_s}(x^s)-\ell_{S_s}^*}}\leq
        \h_t\leq\frac{f_{S_t}(x^t)-\ell_{S_t}^*}{c\|\nabla f_{S_t}(x^t)\|^2}\frac{1}
        {\sqrt{\sum_{s=0}^tf_{S_s}(x^s)-\ell_{S_s}^*}},
    \end{align}
    where $L_{\max}=\max_{i\in[n]}\{L_i\}$. 
\end{lemma}

\subsection{Iterate Moving Average}
\label{sec:ima}

As mentioned in the main paper, the iterates of the \ref{eq:shb} method are equal to the \ref{eq:ima} method. The following proposition is from \citet{sebbouh2021almost}. We include the proof for completeness. 

\begin{proposition}[Proposition 1.6, \citep{sebbouh2021almost}]
    \label{prop:ima-app}
    Let $\h_t>0$ and $\l_t\geq0$. Consider the Iterate Moving Average (IMA) method given by the following update rule:
    \begin{align}
        &z^{t+1}=z^t-\h_t\nabla f_{S_t}(x^t)\label{eq:ima-1}\\
        &x^{t+1}=\frac{\l_{t+1}}{\l_{t+1}+1}x^t+\frac{1}{\l_{t+1}+1}z^{t+1},\label{eq:ima-2}
    \end{align}
    where $z^0=x^0$. If the following equations hold for any $t\in\N$ 
    \begin{align}
        &1+\l_{t+1}=\frac{\l_t}{\b_t}\label{eq:ima-hb-equiv-1}\\
        &\h_t=(1+\l_{t+1})\g_t,\label{eq:ima-hb-equiv-2}
    \end{align}
    then the $x_t$ iterates of the \ref{eq:ima} method are equal to the $x_t$ iterates produced by the \ref{eq:shb} method. 
\end{proposition}

\begin{proof}
    We will start from the \ref{eq:ima} update rule and we will prove that its $(x^t)$ iterates satisfy the \ref{eq:shb} update rule. Suppose that we have the $(z^{t+1})$ and $(x^{t+1})$ iterates as well as the $(\l_t)$ and $(\h_t)$ sequences. Define the sequences $(\b_t)$ and $(\g_t)$ using \cref{eq:ima-hb-equiv-1,eq:ima-hb-equiv-2}. We will show that $(x^t)$ iterates satisfy the SHB update rule. 
    Solving for $z^{t+1}$ in \cref{eq:ima-2} we get 
    \begin{align}
        \label{eq:z-def}
        z^{t+1}:=x^{t+1}+\l_{t+1}(x^{t+1}-x^t)
    \end{align}
    Now substituting \cref{eq:ima-1,eq:z-def} in \cref{eq:ima-2} we have 
    \begin{align*}
        x^{t+1}
        &=\frac{\l_{t+1}}{\l_{t+1}+1}x^t+\frac{1}{\l_{t+1}+1}z^{t+1}\\
        \overset{(\ref{eq:ima-1})}&{=}\frac{\l_{t+1}}{\l_{t+1}+1}x^t+\frac{1}{\l_{t+1}+1}\left(z^t-\h_t\nabla f_{S_t}(x^t)\right)\\
        \overset{(\ref{eq:z-def})}&{=}\frac{\l_{t+1}}{\l_{t+1}+1}x^t+\frac{1}{\l_{t+1}+1}\left(x^t+\l_t(x^t-x^{t-1})-\h_t\nabla f_{S_t}(x^t))\right)\\
        &=x^t-\frac{\h_t}{\l_{t+1}+1}\nabla f_{S_t}(x^t)+\frac{\l_t}{\l_{+t+1}+1}(x^t-x^{t-1})\\
        \overset{(\ref{eq:ima-hb-equiv-1},\ref{eq:ima-hb-equiv-2})}&{=}x^t-\g_t\nabla f_{S_t}(x^t)+\b_t(x^t-x^{t-1}),
    \end{align*}
    as wanted. 
\end{proof}

\medskip
\noindent
The above proposition essentially states that momentum is a convex combination of SGD under a certain transformation. In this paper, we work only with a constant momentum coefficient, i.e., $\b_t=\b$. For this case, we also assume constant $\l_t$, and we have \Cref{pro:ima-ss} and \Cref{cor:trans-app}. 

\begin{proof}[Proof of \Cref{pro:ima-ss}]
    Setting $\l_t=\l$ in \Cref{eq:ima-hb-equiv-1} we have 
    \begin{align*}
        1+\l=\frac{\l}{\b_t}\Rightarrow\b_t=\frac{\l}{1+\l}=:\b.
    \end{align*}
    Then solving for $\l$ we get $\l=\frac{\b}{1-\b}$. Hence 
    \begin{align*}
        \g_t&=\frac{\h_t}{1+\l}\\
        &=\frac{\h_t}{1+\frac{\b}{1-\b}}\\
        &=(1-\b)\h_t
    \end{align*}
\end{proof}

\begin{proof}[Proof of \Cref{cor:trans-app}]
    \begin{itemize}
        \item Let $\h_t=\tn{SPS}_{\max}$. Then 
        \begin{align*}
            \g_t&=(1-\b)\h_t\\
            &=(1-\b)\min\left\{\frac{f_{S_t}(x^t)-\ell_{S_t}^*}{c\|\nabla f_{S_t}(x^t)\|^2}, \h_b\right\}\\
            &=(1-\b)\min\left\{\frac{f_{S_t}(x^t)-\ell_{S_t}^*}{c\|\nabla f_{S_t}(x^t)\|^2}, \g_b\right\}.
        \end{align*}
        Here $\g_b=\h_b$. 

        \item Let $\h_t=\tn{DecSPS}$. Then 
        \begin{align*}
            \g_t&=(1-\b)\h_t\\
            &=(1-\b)\min\left\{\frac{f_{S_t}(x^t)-\ell_{S_t}^*}{c_t\|\nabla f_{S_t}(x^t)\|^2},\frac{\h_{t-1}c_{t-1}}{c_t}\right\}\\
            &=\min\left\{\frac{(1-\b)[f_{S_t}(x^t)-\ell_{S_t}^*]}{c_t\|\nabla f_{S_t}(x^t)\|^2},\frac{(1-\b)\h_{t-1}c_{t-1}}{c_t}\right\}\\
            &=\min\left\{\frac{(1-\b)[f_{S_t}(x^t)-\ell_{S_t}^*]}{c_t\|\nabla f_{S_t}(x^t)\|^2},\frac{\g_{t-1}c_{t-1}}{c_t}\right\},
        \end{align*}
        since $\g_{t-1}=(1-\b)\h_{t-1}$ from \Cref{pro:ima-ss}. 

        \item Let $\h_t=\tn{AdaSPS}$. Then 
        \begin{align*}
            \g_t&=(1-\b)\h_t\\
            &=(1-\b)\min\left\{\frac{f_{S_t}(x^t)-\ell_{S_t}^*}{c\|\nabla f_{S_t}(x^t)\|^2\sqrt{\sum_{s=0}^tf_{S_s}(x^s)-\ell_{S_s}^*}},\h_{t-1}\right\}\\
            &=\min\left\{\frac{(1-\b)[f_{S_t}(x^t)-\ell_{S_t}^*]}{c\|\nabla f_{S_t}(x^t)\|^2\sqrt{\sum_{s=0}^tf_{S_s}(x^s)-\ell_{S_s}^*}},(1-\b)\h_{t-1}\right\}\\
            &=\min\left\{\frac{(1-\b)[f_{S_t}(x^t)-\ell_{S_t}^*]}{c\|\nabla f_{S_t}(x^t)\|^2\sqrt{\sum_{s=0}^tf_{S_s}(x^s)-\ell_{S_s}^*}},\g_{t-1}\right\}.
        \end{align*}
    \end{itemize}
\end{proof}

\subsection{Connection of MomSPS\texorpdfstring{$_{\max}$}{max} and SPS\texorpdfstring{$_{\max}$}{max}}

In this section, we explore the connection between the \ref{eq:mospsmax} and the \ref{eq:spsmax} step-sizes. In particular, one can view the \ref{eq:mospsmax} step-size as the \ref{eq:spsmax} step-size with $c=1/(1-\b)$. Here, we explain why these are two different step-sizes. Recall from \Cref{pro:ima-ss} and \Cref{cor:trans-app} that our proposed step-sizes for the \ref{eq:shb} setting are scaled versions of Polyak step-sizes for SGD. In particular, let's assume that 
\begin{align*}
    \g_t^{SPSmax}=\min\left\{\frac{f_{S_t}(x^t)-f_{S_t}^*}{c_{SPSmax}\|\nabla f_{S_t}(x^t)\|^2},\g_b^{SPSmax}\right\}
\end{align*}
and 
\begin{align*}
    \g_t^{MomSPSmax}=(1-\b)\min\left\{\frac{f_{S_t}(x^t)-f_{S_t}^*}{c_{MomSPSmax}\|\nabla f_{S_t}(x^t)\|^2},\g_b^{MomSPSmax}\right\}
\end{align*}
so we have $\g^t_{SPSmax}=\g^t_{MomSPSmax}$ if and only if 
\begin{align*}
    \frac{1-\b}{c_{MomSPSmax}}=\frac{1}{c_{SPSmax}}\text{ and }\g_b^{SPSmax}=(1-\b)\g_b^{MomSPSmax}.
\end{align*}
Without loss of generalization, we can assume that $c_{MomSPSmax}=1$ (this is also what \Cref{thm:shb-sps-max} guarantees). In this case we get $\b=1-\frac{1}{c_{SPSmax}}$. Hence, we can restate \Cref{thm:shb-sps-max} as follows: 
\begin{theorem}
    Assume that each $f_i$ is convex and $L_i$-smooth. Then, the iterates of \ref{eq:shb} with SPS$_{\max}$ with $c\in\left[1,\frac{2\g_b-2\a}{2\g_b-\a}\right)$ and $\b=1-\frac{1}{c}$ where $\a=\min\left\{\frac{1}{2L_{\max}},\g_b\right\}$ and $L_{\max}=\max_{i\in[n]}\{L_i\}$, converge as 
    \begin{equation*}
        \E[f(\overline{x}^T)-f(x^*)]\leq\frac{C_1\|x^0-x^*\|^2}{T}+C_2\s^2,
    \end{equation*}
    where $\overline{x}^T=\frac{1}{T}\sum_{t=0}^{T-1}x^t$ is the Cesaro average and the constants $C_1=\frac{1-\b}{\a\b+\a-2\b\g_b}$ and $C_2=\frac{2\g_b-\a\b-\a}{\a\b+\a-2\b\g_b}$. 
\end{theorem}
This means that \ref{eq:mospsmax} can be viewed as a special case of SPS$_{\max}$. However, note that the update rules are different since in \ref{eq:shb} we have the extra term $+\b(x^t-x^{t-1})$ so a different analysis is needed. 

\medskip
\noindent
The same line of thought can be applied to other works related to HB. For example, in \citet{ghadimi2015global}, the first global guarantees for HB were established. The authors show that for a convex and $L$-smooth function $f$ HB with any $\b\in[0,1)$ and $0<\g<\frac{2(1-\b)}{L}$ converges with a rate of $O(1/T)$. Also, recall that GD for convex and $L$-smooth function converges for $0<\g<\frac{2}{L}$ with rate $O(1/T)$. Hence, we can restate the result of \citet{ghadimi2015global} as follows: Let $f$ be a convex and $L$-smooth function. Then HB with $\g=\frac{c}{L}$, where $c\in(0,2)$, and $\b=1-\frac{c}{2}$ converges with rate $O(1/T)$. 

\medskip
\noindent
Similarly, for SGD in the convex and smooth setting, we are guaranteed convergence $O(1/T)$ to a neighborhood of the solution when $\g\leq\frac{1}{L}$. In \citet{liu2020improved}, the authors guarantee that \ref{eq:shb} with any $\b\in(0,1)$ and $\g\leq\frac{(1-\b)^2}{L}\min\left\{\frac{1}{4-\b+\b^2},\frac{1}{2\sqrt{2\b^2+2\b}}\right\}$ converges with rate $O(1/T)$ to a neighborhood. This can be restated as follows: \ref{eq:shb} with $\g=\frac{c}{L}$, where $c\in(0,1)$, and $\b$ given by the solution of $c=\min\left\{\frac{(1-\b)^2}{4-\b+\b^2},\frac{(1-\b)^2}{2\sqrt{2\b^2+2\b}}\right\}$, converges with rate $O(1/T)$ to a neighborhood. 

\newpage
\section{Proofs of the main results}
\label{sec:proofs}

In this section we present the proofs of the main theoretical results presented in the main paper, i.e., the convergence analysis of \ref{eq:shb} with \ref{eq:mospsmax}, \ref{eq:modecsps} and \ref{eq:moadasps} for convex and smooth functions $f_i$ and $f$ of Problem \ref{eq:main-problem}. The idea for all the proofs is that we firstly show the corresponding result in the \ref{eq:ima} setting, and then we transfer it to the \ref{eq:shb} setting with the help of \Cref{pro:ima-ss}. 

\medskip
\noindent
Compared to the analysis of SGD with SPS, the \ref{eq:shb} update rule requires taking into consideration extra terms related to the previous iterate $x^{t-1}$. This needs new algebraic tricks in order to use convexity, and after that, one needs to deal with the two quantities $f(x^t)$ and $f(x^{t-1})$ at the same time. We deal with this using the \ref{eq:ima} framework, similarly with \citet{sebbouh2021almost}. However, we highlight that we finish our proofs with a different approach. More specifically, in the proof of Theorem G.1 in \citet{sebbouh2021almost} the authors choose the momentum coefficients $\lambda_t$ in such a way that the main sum telescopes and the final result only has the expectation of the last iterate. In our proof, we make no such simplifications, and instead, we use Jensen’s inequality to finish. This approach allows us to have more freedom in the selection of $\lambda_t$. This is more prominent in the proofs of the decreasing variants. 

\medskip
\noindent
Compared to the classical analysis of constant step-size \ref{eq:shb}, the use of SPS requires an adaptive step-size that uses the loss and stochastic gradient estimates at an iterate, resulting in correlations of the step-size with the gradient, i.e. we might have $\mathbb{E}[\gamma_t\nabla f_i(x^t)]\neq\gamma_t\mathbb{E}[\nabla f_i(x^t)]$. One of the technical challenges in the proofs is to carefully analyze the \ref{eq:shb} iterates, taking these correlations into account. Moreover, since we try to be adaptive to the Lipschitz constant, we can not use any standard descent lemmas (implied by the smoothness and the \ref{eq:shb} update) that are used in the classical analysis of SGD and constant step-size \ref{eq:shb}. This makes the convex proof more challenging than the standard analysis of \ref{eq:shb}.

\subsection{Proof of \texorpdfstring{\Cref{thm:shb-sps-max}}{Theorem 3.1}}

First, we state and prove the corresponding theorem for SPS$_{\max}$ in the \ref{eq:ima} setting. 

\begin{theorem}
    \label{thm:ima-sps-max}
    Assume that each $f_i$ is convex and $L_i$-smooth. Then, the iterates of \ref{eq:ima} with 
    \begin{align*}
        \h_t=\min\left\{\frac{f_{S_t}(x^t)-\ell_{S_t}^*}{\|\nabla f_{S_t}(x^t)\|^2},\h_b\right\}\text{ and }\l_t=\l\in\left[0,\frac{\a}{2(\h_b-\a)}\right),
    \end{align*}
    where $\a=\min\left\{\frac{1}{2L_{\max}},\h_b\right\}$ and $L_{\max}=\max_{i\in[n]}\{L_i\}$, converge as 
    \begin{align*}
        \E[f(\overline{x}^T)-f(x^*)]
        \leq\frac{\|x^0-x^*\|^2}{(2\a\l+\a-2\h_b\l)T}+\frac{(2\h_b+2\h_b\l-2\a\l-\a)\s^2}{2\a\l+\a-2\h_b\l},
    \end{align*}
    where $\overline{x}^T=\frac{1}{T}\sum_{t=0}^{T-1}x^t$ and $\s^2=\E[f_i(x^*)-f_i^*]$. 
\end{theorem}

\begin{proof}
    We have 
    \begin{align*}
        \|z^{t+1}-x^*\|^2
        &=\|z^t-x^*\|^2-2\h_t\langle\nabla f_{S_t}(x^t),z^t-x^*\rangle+\h_t^2\|\nabla f_{S_t}(x^t)\|^2\\
        \overset{(\ref{eq:z-def})}&{=}\|z^t-x^*\|^2-2\h_t\langle\nabla f_{S_t}(x^t),x^t-x^*\rangle-2\h_t\l_t\langle\nabla f_{S_t}(x^t),x^t-x^{t-1}\rangle\\
        &\q+\h_t^2\|\nabla f_{S_t}(x^t)\|^2\\
        \overset{\text{convexity}}&{\leq}\|z^t-x^*\|^2-2\h_t[f_{S_t}(x^t)-f_{S_t}(x^*)]-2\h_t\l_t[f_{S_t}(x^t)-f_{S_t}(x^{t-1})]\\
        &\q+\h_t^2\|\nabla f_{S_t}(x^t)\|^2\\
        \overset{(\ref{eq:lem:sps-bound})}&{\leq}\|z^t-x^*\|^2-2\h_t[f_{S_t}(x^t)-f_{S_t}(x^*)]-2\h_t\l_t[f_{S_t}(x^t)-f_{S_t}(x^{t-1})]\\
        &\q+\h_t[f_{S_t}(x^t)-\ell_{S_t}^*]\\
        &=\|z^t-x^*\|^2-2\h_t[f_{S_t}(x^t)-\ell_{S_t}^*]+2\h_t[f_{S_t}(x^*)-\ell_{S_t}^*]-2\h_t\l_t[f_{S_t}(x^t)-\ell_{S_t}^*]\\
        &\quad+2\h_t\l_t[f_{S_t}(x^{t-1})-\ell_{S_t}^*]+\h_t[f_{S_t}(x^t)-\ell_{S_t}^*]\\
        &=\|z^t-x^*\|^2-2\h_t\left(1+\l_t-\frac{1}{2}\right)[f_{S_t}(x^t)-\ell_{S_t}^*]+2\h_t\l_t[f_{S_t}(x^{t-1})-\ell_{S_t}^*]\\
        &\quad+2\h_t[f_{S_t}(x^*)-\ell_{S_t}^*]\\
        \overset{(\ref{eq:lem:sps-bound})}&{\leq}\|z^t-x^*\|^2-2\a\left(\l_t+\frac{1}{2}\right)[f_{S_t}(x^t)-\ell_{S_t}^*]+2\h_b\l_t[f_{S_t}(x^{t-1})-\ell_{S_t}^*]\\
        &\quad+2\h_b[f_{S_t}(x^*)-\ell_{S_t}^*]\\
        &\leq\|z^t-x^*\|^2-2\a\left(\l_t+\frac{1}{2}\right)[f_{S_t}(x^t)-f_{S_t}(x^*)]-2\a\left(\l_t+\frac{1}{2}\right)[f_{S_t}(x^*)-\ell_{S_t}^*]\\
        &\q+2\h_b\l_t[f_{S_t}(x^{t-1})-f_{S_t}(x^*)]+2\h_b\l_t[f_{S_t}(x^*)-\ell_{S_t}^*]+2\h_b[f_{S_t}(x^*)-\ell_{S_t}^*].
    \end{align*}
    Now take expectation conditional on $x^t$ to get 
    \begin{align*}
        \E_{S_t}\|z^{t+1}-x^*\|^2
        &\leq\E_{S_t}\|z^t-x^*\|^2-2\a\left(\l_t+\frac{1}{2}\right)\E_{S_t}[f_{S_t}(x^t)-f_{S_t}(x^*)]\\
        &\q-2\a\left(\l_t+\frac{1}{2}\right)\E_{S_t}[f_{S_t}(x^*)-\ell_{S_t}^*]+2\h_b\l_t\E_{S_t}[f_{S_t}(x^{t-1})-f_{S_t}(x^*)]\\
        &\q+2\h_b\l_t\E_{S_t}[f_{S_t}(x^*)-\ell_{S_t}^*]+2\h_b\E_{S_t}[f_{S_t}(x^*)-\ell_{S_t}^*]\\
        &=\|z^t-x^*\|^2-2\a\left(\l_t+\frac{1}{2}\right)[f(x^t)-f(x^*)]-2\a\left(\l_t+\frac{1}{2}\right)\s^2\\
        &\quad+2\h_b\l_t[f(x^{t-1})-f(x^*)]+2\h_b\l_t\s^2+2\h_b\s^2\\
        &=\|z^t-x^*\|^2-2\a\left(\l_t+\frac{1}{2}\right)[f(x^t)-f(x^*)]+
        2\h_b\l_t[f(x^{t-1})-f(x^*)]\\
        &\quad+2\left(\h_b\left(1+\l_t\right)-\a\left(\l_t+\frac{1}{2}\right)\right)\s^2.\numberthis\label{eq:thm:ima-sps-max-coeff}
    \end{align*}
    Note that in \cref{eq:thm:ima-sps-max-coeff} the coefficient $2\left(\h_b\left(1+\l_t\right)-\a\left(\l_t+\frac{1}{2}\right)\right)$ of $\s^2$ is positive since $\h_b\geq\a$. Taking expectation again and using the tower property we have 
    \begin{align*}
        &2\a\left(\l_t+\frac{1}{2}\right)\E[f(x^t)-f(x^*)]-2\h_b\l_t\E[f(x^{t-1})
        -f(x^*)]\\
        &\leq\E\|z^t-x^*\|^2-\E\|z^{t+1}-x^*\|^2+2\left(\h_b\left(1+\l_t\right)-
        \a\left(\l_t+\frac{1}{2}\right)\right)\s^2.\numberthis\label{eq:thm:ima-sps-max-sum}
    \end{align*}
    Summing \cref{eq:thm:ima-sps-max-sum} for $t=0,\dots,T-1$ and telescoping we get 
    \begin{align*}
        \sum_{t=0}^{T-1}&\left[2\a\left(\l_t+\frac{1}{2}\right)\E[f(x^t)-f(x^*)]-2\h_b\l_t\E[f(x^{t-1})-f(x^*)]\right]\\
        &\leq\sum_{t=0}^{T-1}\left[\E\|z^t-x^*\|^2-\E\|z^{t+1}-x^*\|^2\right]+\sum_{t=0}^{T-1}2\left(\h_b\left(1+\l_t\right)-\a\left(\l_t+\frac{1}{2}\right)\right)\s^2\\
        &=\E\|z^0-x^*\|^2-\E\|z^T-x^*\|^2+\sum_{t=0}^{T-1}2\left(\h_b\left(1+\l_t\right)-\a\left(\l_t+\frac{1}{2}\right)\right)\s^2\\
        \overset{z^0=x^0}&{\leq}\|x^0-x^*\|^2+\sum_{t=0}^{T-1}2\left(\h_b\left(1+\l_t\right)-\a\left(\l_t+\frac{1}{2}\right)\right)\s^2.\numberthis\label{eq:thm:ima-sps-max-main}
    \end{align*}
    The LHS of \cref{eq:thm:ima-sps-max-main} can be bounded as follows (using $x^{-1}=x^0$) 
        \begin{align*}
        &\sum_{t=0}^{T-1}\left[2\a\left(\l_t+\frac{1}{2}\right)\E[f(x^t)-f(x^*)]-2\h_b\l_t\E[f(x^{t-1})-f(x^*)]\right]\\
        &=\sum_{t=0}^{T-1}2\a\left(\l_t+\frac{1}{2}\right)\E[f(x^t)-f(x^*)]-\sum_{t=0}^{T-1}2\h_b\l_t\E[f(x^{t-1})-f(x^*)]\\
        &=\sum_{t=0}^{T-1}2\a\left(\l_t+\frac{1}{2}\right)\E[f(x^t)-f(x^*)]-\sum_{t=0}^{T-2}2\h_b\l_{t+1}\E[f(x^t)-f(x^*)]\\
        &=2\a\left(\l_{T-1}+\frac{1}{2}\right)\E[f(x^{T-1})-f(x^*)]+\sum_{t=0}^{T-2}2\a\left(\l_t+\frac{1}{2}\right)\E[f(x^t)-f(x^*)]\\
        &\q-\sum_{t=0}^{T-2}2\h_b\l_{t+1}\E[f(x^t)-f(x^*)]\\
        &=2\a\left(\l_{T-1}+\frac{1}{2}\right)\E[f(x^{T-1})-f(x^*)]+\sum_{t=0}^{T-2}\left[2\a\left(\l_t+\frac{1}{2}\right)-2\h_b\l_{t+1}\right]\E[f(x^t)-f(x^*)]\\
        &\geq\left[2\a\left(\l_{T-1}+\frac{1}{2}\right)-2\h_b\l_T\right]\E[f(x^{T-1})-f(x^*)]\\
        &\q+\sum_{t=0}^{T-2}\left[2\a\left(\l_t+\frac{1}{2}\right)-2\h_b\l_{t+1}\right]\E[f(x^t)-f(x^*)]\\
        &=\sum_{t=0}^{T-1}\left[2\a\left(\l_t+\frac{1}{2}\right)-2\h_b\l_{t+1}\right]\E[f(x^t)-f(x^*)].\numberthis\label{eq:thm:ima-sps-max-bound}
    \end{align*}
    Combining \cref{eq:thm:ima-sps-max-main} and \cref{eq:thm:ima-sps-max-bound} we have 
    \begin{align*}
        \label{eq:thm:ima-sps-max-lastt}
        &\sum_{t=0}^{T-1}\left[2\a\left(\l_t+\frac{1}{2}\right)-2\h_b\l_{t+1}\right]\E[f(x^t)-f(x^*)]\\
        &\leq\|x^0-x^*\|^2+\sum_{t=0}^{T-1}2\left(\h_b\left(1+\l_t\right)-\a\left(\l_t+\frac{1}{2}\right)\right)\s^2.\numberthis
    \end{align*}
    In \cref{eq:thm:ima-sps-max-lastt}, setting $\l_t=\l,\forall t>0$ with $0\leq\l<\frac{\a}{2(\h_b-\a)}$ (in order to ensure that quantity $2\a\left(\l+\frac{1}{2}\right)-2\h_b\l$ in the LHS is positive) we get 
    \begin{align}
        \label{eq:thm:ima-sps-max-last}
        \left[2\a\left(\l+\frac{1}{2}\right)-2\h_b\l\right]\sum_{t=0}^{T-1}\E[f(x^t)-f(x^*)]\leq\|x^0-x^*\|^2+2T\left(\h_b\left(1+\l\right)-\a\left(\l+\frac{1}{2}\right)\right)\s^2.
    \end{align}
    Now dividing \cref{eq:thm:ima-sps-max-last} by $T\left[2\a\left(\l+\frac{1}{2}\right)-2\h_b\l\right]$ and using Jensen's inequality we get 
    \begin{align*}
        \E[f(\overline{x}^T)-f(x^*)]
        &\leq\frac{1}{T}\sum_{t=0}^{T-1}\E[f(x^t)-f(x^*)]\\
        &\leq\frac{\|x^0-x^*\|^2+T\left(2\h_b+2\h_b\l-2\a\l-\a\right)\s^2}{(2\a\l+\a-2\h_b\l)T}\\
        &=\frac{\|x^0-x^*\|^2}{(2\a\l+\a-2\h_b\l)T}+\frac{2\h_b+2\h_b\l-2\a\l-\a}{2\a\l+\a-2\h_b\l}\s^2,
    \end{align*}
    as wanted.
\end{proof}

\medskip
\noindent
Now the proof of \Cref{thm:shb-sps-max} follows immediately from \Cref{thm:ima-sps-max} setting $\l=\frac{\b}{1-\b}$ (from \Cref{pro:ima-ss}) and $\g_b=\h_b$ (from \Cref{cor:trans-app}). The bound on $\b=\frac{\l}{1+\l}$ follows from the bound on $\l$ and from the fact that the function $f(x)=\frac{x}{1+x}$ is strictly increasing for $x\geq0$.

\subsection{Proof of \texorpdfstring{\Cref{thm:shb-dec-sps}}{Theorem 3.5}}

Here we state and prove the corresponding theorem for DecSPS in the \ref{eq:ima} setting. 

\begin{theorem}
    \label{thm:ima-dec-sps}
    Assume that each $f_i$ is convex and $L_i$-smooth. Then, the iterates of \ref{eq:ima} with 
    \begin{align*}
        \h_t=\frac{1}{c_t}\min\left\{\frac{f_{S_t}(x^t)-\ell_{S_t}^*}{\|\nabla f_{S_t}(x^t)\|^2},
        c_{t-1}\h_{t-1}\right\}\text{ and }\l_t\geq0,
    \end{align*}
    where $(c_t)$ is an increasing sequences with $c_t\geq1$, $c_{-1}=c_0$, $\h_{-1}=\h_b$ and $(\l_t)$ is a decreasing sequence, converge as  
    \begin{align*}
        \E[f(\overline{x}^T)-f(x^*)]\leq\frac{2\l_1[f(x^0)-f(x^*)]}{T}+
        \frac{c_{T-1}D^2_z}{T\a}+\frac{\s^2}{T}\sum_{t=0}^{T-1}\frac{1}{c_t},
    \end{align*}
    where $\overline{x}^T=\frac{1}{T}\sum_{t=0}^{T-1}x^t$, $\a=\min\left\{\frac{1}{2L},\h_bc_0\right\}$ and $D^2_z=\max_{t\in[T]}\|z^t-x^*\|^2$. 
\end{theorem}

\begin{proof}
    We have 
    \begin{align*}
        \|z^{t+1}-x^*\|^2\overset{(\ref{eq:ima-1})}{=}\|z^t-x^*\|^2-2\h_t\langle\nabla f_{S_t}(x^t),z^t-x^*\rangle+\h_t^2\|\nabla f_{S_t}(x^t)\|^2.
    \end{align*}
    Rearranging and dividing by $\h_t>0$ we get 
    \begin{align*}
        2\langle\nabla f_{S_t}(x^t),z^t-x^*\rangle\leq\frac{\|z^t-x^*\|^2}{\h_t}-\frac{\|z^{t+1}-x^*\|^2}{\h_t}+\h_t\|\nabla f_{S_t}(x^t)\|^2.
    \end{align*}
    Summing the above for $t=0,\dots,T-1$ and telescoping we have 
    \begin{align*}
        &2\sum_{t=0}^{T-1}\langle\nabla f_{S_t}(x^t),z^t-x^*\rangle\\
        &\leq\sum_{t=0}^{T-1}\left[\frac{\|z^t-x^*\|^2}{\h_t}-\frac{\|z^{t+1}-x^*\|^2}{\h_t}+\h_t\|\nabla f_{S_t}(x^t)\|^2\right]\\
        &=\sum_{t=0}^{T-1}\frac{\|z^t-x^*\|^2}{\h_t}-\sum_{t=0}^{T-1}\frac{\|z^{t+1}-x^*\|^2}{\h_t}+\sum_{t=0}^{T-1}\h_t\|\nabla f_{S_t}(x^t)\|^2\\
        &=\frac{\|z^0-x^*\|^2}{\h_0}+\sum_{t=1}^{T-1}\frac{\|z^t-x^*\|^2}{\h_t}-\sum_{t=0}^{T-1}\frac{\|z^{t+1}-x^*\|^2}{\h_t}+\sum_{t=0}^{T-1}\h_t\|\nabla f_{S_t}(x^t)\|^2\\
        &=\frac{\|z^0-x^*\|^2}{\h_0}+\sum_{t=0}^{T-2}\frac{\|z^{t+1}-x^*\|^2}{\h_{t+1}}-\sum_{t=0}^{T-1}\frac{\|z^{t+1}-x^*\|^2}{\h_t}+\sum_{t=0}^{T-1}\h_t\|\nabla f_{S_t}(x^t)\|^2\\
        &=\frac{\|z^0-x^*\|^2}{\h_0}+\sum_{t=0}^{T-2}\frac{\|z^{t+1}-x^*\|^2}{\h_{t+1}}-\sum_{t=0}^{T-2}\frac{\|z^{t+1}-x^*\|^2}{\h_t}-\frac{\|z^T-x^*\|^2}{\h_{T-1}}+\sum_{t=0}^{T-1}\h_t\|\nabla f_{S_t}(x^t)\|^2\\
        &=\frac{\|z^0-x^*\|^2}{\h_0}+\sum_{t=0}^{T-2}\left(\frac{1}{\h_{t+1}}-
        \frac{1}{\h_t}\right)\|z^{t+1}-x^*\|^2-\frac{\|z^T-x^*\|^2}{\h_{T-1}}+
        \sum_{t=0}^{T-1}\h_t\|\nabla f_{S_t}(x^t)\|^2\\
        &\leq\left(\frac{1}{\h_0}+\sum_{t=0}^{T-2}\left[\frac{1}{\h_{t+1}}-
        \frac{1}{\h_t}\right]\right)D^2_z+\sum_{t=0}^{T-1}\h_t\|\nabla f_{S_t}(x^t)\|^2\\
        &=\frac{D^2_z}{\h_{T-1}}+\sum_{t=0}^{T-1}\h_t\|\nabla f_{S_t}(x^t)\|^2\numberthis\label{eq:main-bound}.
    \end{align*}
    Taking expectation in \cref{eq:main-bound} we have 
    \begin{align}
        \label{eq:thm:ima-dec-sps-pre}
        2\sum_{t=0}^{T-1}\E\left[\langle\nabla f(x^t),z^t-x^*\rangle\right]
        \leq\E\left[\frac{D^2_z}{\h_{T-1}}+\sum_{t=0}^{T-1}\h_t\|\nabla f_{S_t}(x^t)\|^2\right].
    \end{align}
    Now the LHS of \cref{eq:thm:ima-dec-sps-pre} can be bounded as follows 
    \begin{align*}
        2\sum_{t=0}^{T-1}&\E\left[\langle\nabla f(x^t),z^t-x^*\rangle\right]\\
        \overset{(z^0=x^0)}&{=}2\E\left[\langle\nabla f(x^0),x^0-x^*\rangle\right]+2\sum_{t=1}^{T-1}\E\left[\langle\nabla f(x^t),z^t-x^*\rangle\right]\\
        \overset{(\ref{eq:z-def})}&{=}2\E\left[\langle\nabla f(x^0),x^0-x^*\rangle\right]+\sum_{t=1}^{T-1}\E\left[2\langle\nabla f(x^t),x^t-x^*\rangle+2\l_t\langle\nabla f(x^t),x^t-x^{t-1}\rangle\right]\\
        \overset{\text{convexity}}&{\geq}2\E\left[f(x^0)-f(x^*)\right]+\sum_{t=1}^{T-1}\E\left[2(f(x^t)-f(x^*))+2\l_t(f(x^t)-f(x^{t-1}))\right]\\
        &=2\E\left[f(x^0)-f(x^*)\right]+\sum_{t=1}^{T-1}\left[2(1+\l_t)\E[f(x^t)-f(x^*)]-2\l_t\E[f(x^{t-1})-f(x^*)]\right]\\
        &=2\E\left[f(x^0)-f(x^*)\right]+\sum_{t=1}^{T-1}2(1+\l_t)\E[f(x^t)-f(x^*)]-\sum_{t=1}^{T-1}2\l_t\E[f(x^{t-1})-f(x^*)]\\
        &=2\E\left[f(x^0)-f(x^*)\right]+\sum_{t=1}^{T-1}2(1+\l_t)\E[f(x^t)-f(x^*)]-\sum_{t=0}^{T-2}2\l_{t+1}\E[f(x^t)-f(x^*)]\\
        &=2\E\left[f(x^0)-f(x^*)\right]+2\sum_{t=1}^{T-1}\E[f(x^t)-f(x^*)]+\sum_{t=1}^{T-1}2\l_t\E[f(x^t)-f(x^*)]\\
        &\q-\sum_{t=0}^{T-2}2\l_{t+1}\E[f(x^t)-f(x^*)]\\
        &=2\E\left[f(x^0)-f(x^*)\right]+2\sum_{t=1}^{T-1}\E[f(x^t)-f(x^*)]+2\l_{T-1}\E[f(x^{T-1})-f(x^*)]\\
        &\q+\sum_{t=1}^{T-2}2\l_t\E[f(x^t)-f(x^*)]-\sum_{t=1}^{T-2}2\l_{t+1}\E[f(x^t)-f(x^*)]-2\l_1\E[f(x^0)-f(x^*)]\\
        &=2\sum_{t=0}^{T-1}\E[f(x^t)-f(x^*)]+2\l_{T-1}\E[f(x^{T-1})-f(x^*)]\\
        &\q+2\sum_{t=1}^{T-2}(\l_t-\l_{t+1})\E[f(x^t)-f(x^*)]-2\l_1\E[f(x^0)-f(x^*)]\\
        &\geq2\sum_{t=0}^{T-1}\E[f(x^t)-f(x^*)]-2\l_1[f(x^0)-f(x^*)]\numberthis\label{eq:thm:ima-dec-sps-bound},
    \end{align*}
    using the fact that $(\l_t)$ is decreasing and the fact that $\E[f(x^t)-f(x^*)]\geq0$. Combining \cref{eq:thm:ima-dec-sps-pre} and \cref{eq:thm:ima-dec-sps-bound} we have 
    \begin{align}
        \label{eq:thm:ima-dec-sps-important}
        2\sum_{t=0}^{T-1}\E[f(x^t)-f(x^*)]-2\l_1[f(x^0)-f(x^*)]
        \leq\E\left[\frac{D^2_z}{\h_{T-1}}+\sum_{t=0}^{T-1}\E\h_t\|\nabla f_{S_t}(x^t)\|^2\right].
    \end{align}
    Now for the RHS of \cref{eq:thm:ima-dec-sps-important} we have 
    \begin{align*}
        \frac{D^2_z}{\h_{T-1}}+\sum_{t=0}^{T-1}\h_t\|\nabla f_{S_t}(x^t)\|^2
        &\overset{(\ref{eq:lem:decsps-bound})}{\leq}\frac{D^2_z}{\h_{T-1}}+\sum_{t=0}^{T-1}\frac{1}{c_t}[f_{S_t}(x^t)-\ell_{S_t}^*]\\
        &=\frac{D^2_z}{\h_{T-1}}+\sum_{t=0}^{T-1}\frac{1}{c_t}[f_{S_t}(x^t)-f_{S_t}(x^*)]+\sum_{t=0}^{T-1}\frac{1}{c_t}[f_{S_t}(x^*)-\ell_{S_t}^*]\\
        &\leq\frac{c_{T-1}D^2_z}{\a}+\sum_{t=0}^{T-1}\frac{1}{c_t}[f_{S_t}(x^t)-f_{S_t}(x^*)]+\sum_{t=0}^{T-1}\frac{1}{c_t}[f_{S_t}(x^*)-\ell_{S_t}^*],\numberthis\label{eq:thm:ima-dec-sps-tech-bound}
    \end{align*}
    where the last inequality follows from the fact that $\h_t$ is decreasing and $\h_t\geq\frac{\a}{c_t}$ with $\a=\min\{\frac{1}{2L_{\max}},\h_bc_0\}$, from \Cref{lem:decsps-bound}. Thus, combining \cref{eq:thm:ima-dec-sps-important} and \cref{eq:thm:ima-dec-sps-tech-bound} we have  
    \begin{align}
        \label{eq:thm:ima-dec-sps-spec-pre}
        2\sum_{t=0}^{T-1}\E[f(x^t)-f(x^*)]-2\l_1[f(x^0)-f(x^*)]
        \leq\frac{c_{T-1}D^2_z}{\a}+\sum_{t=0}^{T-1}\frac{1}{c_t}\E[f(x^t)-f(x^*)]+\sum_{t=0}^{T-1}\frac{\s^2}{c_t}.
    \end{align}
    Rearranging \cref{eq:thm:ima-dec-sps-spec-pre} we get 
    \begin{align}
        \label{eq:thm:ima-dec-sps-spec}
        \sum_{t=0}^{T-1}\left(2-\frac{1}{c_t}\right)\E[f(x^t)-f(x^*)]\leq2\l_1
        [f(x^0)-f(x^*)]+\frac{c_{T-1}D^2_z}{\a}+\s^2\sum_{t=0}^{T-1}\frac{1}{c_t}
    \end{align}
    and using the fact that $2-\frac{1}{c_t}\geq1>0$ \cref{eq:thm:ima-dec-sps-spec} reduces to 
    \begin{align}
        \label{eq:thm:ima-dec-sps-red}
        \sum_{t=0}^{T-1}\E[f(x^t)-f(x^*)]
        \leq2\l_1[f(x^0)-f(x^*)]+\frac{c_{T-1}D^2_z}{\a}+\s^2\sum_{t=0}^{T-1}\frac{1}{c_t}
    \end{align}
    Dividing \cref{eq:thm:ima-dec-sps-red} by $T$ and using Jensen's inequality we get 
    \begin{align*}
        \E[f(\overline{x}^T)-f(x^*)]
        &\leq\frac{1}{T}\sum_{t=0}^{T-1}\E[f(x^t)-f(x^*)]\\
        &\leq\frac{2\l_1[f(x^0)-f(x^*)]}{T}+\frac{c_{T-1}D^2_z}{T\a}+\frac{\s^2}{T}\sum_{t=0}^{T-1}\frac{1}{c_t},
    \end{align*}
    as wanted. 
\end{proof}

\medskip
\noindent
If $\s^2>0$ like in the standard SGD analysis under decreasing step-sizes, the choice $c_t=\sqrt{t+1}$ leads to the asymptotic rate $O(1/\sqrt{t})$. 

\begin{corollary}
    \label{cor:ima-dec-sps}
    Under the setting of \Cref{thm:ima-dec-sps}, for $c_t=\sqrt{t+1}$ ($c_{-1}=c_0$) we have 
        \begin{align*}
        \E[f(\overline{x}^T)-f(x^*)]
        &\leq\frac{2\l_1[f(x^0)-f(x^*)]}{T}+\frac{D^2_z}{\a\sqrt{T}}+\frac{2\s^2}{\sqrt{T}}.
    \end{align*}
\end{corollary}

\begin{proof}
    It follows from the well known inequality 
    \begin{align*}
        \sum_{t=0}^{T-1}\frac{1}{\sqrt{t+1}}\leq2\sqrt{T}.
    \end{align*}
\end{proof}

\medskip
\noindent
Now the proof of \Cref{thm:shb-dec-sps} follows immediately from \Cref{cor:ima-dec-sps} setting $\l_t=\l=\frac{\b}{1-\b}$ (from \Cref{pro:ima-ss}). Moreover we have 
\begin{align*}
    \|z^t-x^*\|^2
    \overset{(\ref{eq:z-def})}&{=}\|(1+\l)x^t-\l x^{t-1}-x^*\|^2\\
    &\leq(1+\l)\|x^t-x^*\|^2+\l\|x^{t-1}-x^*\|^2\\
    &\leq(1+2\l)\max_t\|x^t-x^*\|^2,\numberthis\label{eq:dz-to-d}
\end{align*}
thus $D^2_z\leq(1+2\l)\max_t\|x^t-x^*\|^2=\frac{1+\b}{1-\b}D^2$, where $D^2=\max_{t\in[T]}\|x^t-x^*\|$.

\subsection{Proof of \texorpdfstring{\Cref{thm:shb-ada-sps}}{Theorem 3.6}}

Here we state and prove the corresponding theorem for AdaSPS in the \ref{eq:ima} setting. 

\begin{theorem}
    \label{thm:ima-ada-sps}
    Assume that each $f_i$ is convex and $L_i$-smooth. Then, the iterates of \ref{eq:ima} with 
    \begin{align*}
        \h_t=\min\left\{\frac{f_{S_t}(x^t)-\ell_{S_t}^*}{c\|\nabla f_{S_t}(x^t)\|^2}\frac{1}{\sqrt{\sum_{s=0}^tf_{S_s}(x^s)-\ell_{S_s}^*}},\h_{t-1}\right\}\text{ and }\l_t\geq0,
    \end{align*}
    where $\h_{-1}=+\infty$, $c>0$ and $(\l_t)$ a decreasing sequence, converge as 
    \begin{align*}
        \E[f(\overline{x}^T)-f(x^*)]\leq\frac{\t^2}{T}+\frac{\t\s}{\sqrt{T}},
    \end{align*}
    where $\overline{x}^T=\frac{1}{T}\sum_{t=0}^{T-1}x^t$, $\t=\left(\l_1\sqrt{f(x^0)-f(x^*)}+\frac{cLD^2_z}{2}+\frac{1}{2c}\right)$ and $D^2_z=\max_t\|z^t-x^*\|$. 
\end{theorem}

\begin{proof}
    Using \Cref{eq:thm:ima-dec-sps-important} from the proof of \Cref{thm:shb-dec-sps}, we have 
    \begin{align}
        \label{eq:thm:ima-ada-sps-important}
        2\sum_{t=0}^{T-1}\E[f(x^t)-f(x^*)]-2\l_1[f(x^0)-f(x^*)]
        \leq\E\left[\frac{D^2_z}{\h_{T-1}}+\sum_{t=0}^{T-1}\E\h_t\|\nabla f_{S_t}(x^t)\|^2\right].
    \end{align}
    By \Cref{lem:adasps-bounds} we have 
    \begin{align}
        \label{eq:thm:ima-ada-sps-lem1}
        \frac{D^2_z}{\h_{T-1}}
        \leq cLD^2_z\sqrt{\sum_{s=0}^{T-1}f_{S_s}(x^s)-\ell_{S_s}^*},
    \end{align}
    and by \Cref{lem:adasps-trick} we have 
    \begin{align}
        \label{eq:thm:ima-ada-sps-lem2}
        \sum_{t=0}^{T-1}\h_t\|\nabla f_{S_t}(x^t)\|^2
        \leq\sum_{t=0}^{T-1}\frac{f_{S_t}(x^t)-\ell_{S_t}^*}{2c\sqrt{\sum_{s=0}^tf_{S_s}(x^s)-\ell_{S_s}^*}}
        \leq\frac{1}{c}\sqrt{\sum_{s=0}^{T-1}f_{S_s}(x^s)-\ell_{S_s}^*}.
    \end{align}
    Now combining \Cref{eq:thm:ima-ada-sps-important,eq:thm:ima-ada-sps-lem1,eq:thm:ima-ada-sps-lem2} and using Jensen's inequality we get 
    \begin{align*}
        &2\sum_{t=0}^{T-1}\E[f(x^t)-f(x^*)]-2\l_1[f(x^0)-f(x^*)]\\
        &\leq\E\left[cLD^2_z\sqrt{\sum_{s=0}^{T-1}f_{S_s}(x^s)-\ell_{S_s}^*}+\frac{1}{c}\sqrt{\sum_{s=0}^{T-1}f_{S_s}(x^s)-\ell_{S_s}^*}\right]\\
        &=\left(cLD^2_z+\frac{1}{c}\right)\E\left[\sqrt{\sum_{s=0}^{T-1}f_{S_s}(x^s)-\ell_{S_s}^*}\right]\\
        &=\left(cLD^2_z+\frac{1}{c}\right)\E\left[\sqrt{\sum_{s=0}^{T-1}(f_{S_s}(x^s)-f_{S_s}(x^*)+f_{S_s}(x^*)-\ell_{S_s}^*)}\right]\\
        &\leq\left(cLD^2_z+\frac{1}{c}\right)\sqrt{\sum_{s=0}^{T-1}\E[f(x^s)-f(x^*)]+\s^2}.\numberthis\label{eq:thm:ima-ada-sps-main}
    \end{align*}
    Rearranging \cref{eq:thm:ima-ada-sps-main} we have 
    \begin{align}
        \label{eq:thm:ima-ada-sps-main2}
        2\sum_{t=0}^{T-1}\E[f(x^t)-f(x^*)]\leq2\l_1[f(x^0)-f(x^*)]+\left(cLD^2_z+\frac{1}{c}\right)\sqrt{\sum_{s=0}^{T-1}\E[f(x^s)-f(x^*)]+\s^2}.
    \end{align}
    Now let us choose $c_q=\sqrt{f(x^0)-f(x^*)}$. Then 
    \begin{align*}
        2\l_1[f(x^0)-f(x^*)]
        &=2\l_1c_q\sqrt{f(x^0)-f(x^*)}\\
        &\leq2\l_1c_q\sqrt{\sum_{s=0}^{T-1}\E[f(x^s)-f(x^*)]+\s^2},\numberthis\label{eq:thm:ima-ada-sps-l-b}
    \end{align*}
    Combining \cref{eq:thm:ima-ada-sps-main2} and \cref{eq:thm:ima-ada-sps-l-b} we get 
    \begin{align}
        \label{eq:thm:ima-ada-sps-inter}
        \sum_{t=0}^{T-1}\E[f(x^t)-f(x^*)]
        \leq\left(\l_1c_q+\frac{cLD^2_z}{2}+\frac{1}{2c}\right)\sqrt{\sum_{s=0}^{T-1}\E[f(x^s)-f(x^*)]+\s^2}.
    \end{align}
    Squaring both sides of \cref{eq:thm:ima-ada-sps-inter} we have 
    \begin{align}
        \label{eq:thm:ima-ada-sps-inter2}
        \left(\sum_{t=0}^{T-1}\E[f(x^t)-f(x^*)]\right)^2
        \leq\t^2\left(\sum_{t=0}
        ^{T-1}\E[f(x^t)-f(x^*)]+T\s^2\right),
    \end{align}
    where $\t=\left(\l_1c_q+\frac{cLD^2_z}{2}+\frac{1}{2c}\right)$. Now we use \Cref{lem:adasps-finish} in \cref{eq:thm:ima-ada-sps-inter} to get 
    \begin{align}
        \label{eq:thm:ima-ada-sps-inter3}
        \sum_{t=0}^{T-1}\E[f(x^t)-f(x^*)]\leq\t^2+\t\s\sqrt{T}.
    \end{align}
    Finally by Jensen's inequality and \cref{eq:thm:ima-ada-sps-inter3} we have 
    \begin{align*}
        \E[f(\overline{x}^T)-f(x^*)]\leq\frac{\t^2}{T}+\frac{\t\s}{\sqrt{T}},
    \end{align*}
    as wanted. 
\end{proof}

\medskip
\noindent
Now the proof of \Cref{thm:shb-ada-sps} follows immediately from \Cref{thm:ima-ada-sps} setting $\l=\frac{\b}{1-\b}$ (from \Cref{pro:ima-ss}). Moreover, using \cref{eq:dz-to-d} we have $D^2_z\leq\frac{1+\b}{1-\b}D^2$. 

\newpage
\section{Beyond the bounded iterates assumption}

The two main theorems on the convergence of MomDecSPS/MomAdaSPS require the bounded iterates assumption. As mentioned in the main paper, this is a standard assumption for several adaptive step-sizes, see: \citep{reddi2019convergence,ward2020adagrad,orvieto2022dynamics,jiang2023adaptive}. However, one can artificially remove this assumption by adding a projection step in the update rule of IMA onto a compact and convex subset $\mathcal{D}\subseteq\R^d$ as described in this section. 

Consider the \textit{constrained} finite-sum optimization problem,
\begin{equation}
    \label{eq:main-problem-constr}
    \min_{x\in\mathcal{D}}\left[f(x)=\frac{1}{n}\sum_{i=1}^nf_i(x)\right],
\end{equation}
where each $f_i:\R^d\to\R$ is convex, smooth, and lower bounded by $\ell_i^*$ and $\mathcal{D}\subseteq\R^d$ is a compact and convex subset. Let $X^*_{\mathcal{D}}\subseteq\mathcal{D}$ be the set of minimizers of (\ref{eq:main-problem-constr}). We assume that $X^*_{\mathcal{D}}\neq\emptyset$ and we fix $x^*\in X^*_{\mathcal{D}}$. 

The new update rule of IMA, takes the following form (Projected IMA): 
\begin{align}
    &z^{t+1}=\tn{proj}_{\mathcal{D}}\left[z^t-\h_t\nabla f_{S_t}(x^t)\right]\\
    &x^{t+1}=\frac{\l_{t+1}}{\l_{t+1}+1}x^t+\frac{1}{\l_{t+1}+1}z^{t+1},
\end{align}
with $x^0=z^0\in\mathcal{D}$, where $\tn{proj}_{\mathcal{D}}(x)\in\arg\min_{d\in\mathcal{D}}\|d-x\|^2$. Note that the projection step is only needed for the update of $z^{t+1}$ because $x^{t+1}$ is already a convex combination of elements in the convex set $\mathcal{D}$, namely $z^{t+1}$ and $x^t$ by induction. We highlight the fact that Projected IMA is not necessarily equivalent to Projected SHB. Now the proofs of \Cref{thm:ima-dec-sps,thm:ima-ada-sps} will go through using the non-expansiveness of the projection operator, as explained in the following lemma: 

\begin{lemma}[Non-expansiveness]
    \label{lem:non-exp}
    For all $x,y\in\R^d$ it holds 
    \begin{align*}
        \|\tn{proj}_{\mathcal{D}}(x)-\tn{proj}_{\mathcal{D}}(y)\|\leq\|x-y\|
    \end{align*}
\end{lemma}

\begin{proof}
    See Example 8.14 and Lemma 8.16 in \citet{garrigos2023handbook}. 
\end{proof}

Now we have: 

\begin{theorem}[Pojected IMA version of \Cref{thm:ima-dec-sps}]
    Assume that each $f_i$ is convex and $L_i$-smooth. Then, the iterates of Projected IMA with 
    \begin{align*}
        \h_t=\frac{1}{c_t}\min\left\{\frac{f_{S_t}(x^t)-\ell_{S_t}^*}{\|\nabla f_{S_t}(x^t)\|^2},
        c_{t-1}\h_{t-1}\right\}\text{ and }\l_t\geq0,
    \end{align*}
    where $(c_t)$ is an increasing sequences with $c_t\geq1$, $c_{-1}=c_0$, $\h_{-1}=\h_b$ and $(\l_t)$ is a decreasing sequence, converge as  
    \begin{align*}
        \E[f(\overline{x}^T)-f(x^*)]\leq\frac{2\l_1[f(x^0)-f(x^*)]}{T}+
        \frac{c_{T-1}D^2}{T\a}+\frac{\s^2}{T}\sum_{t=0}^{T-1}\frac{1}{c_t},
    \end{align*}
    where $\overline{x}^T=\frac{1}{T}\sum_{t=0}^{T-1}x^t$, $\a=\min\left\{\frac{1}{2L},\h_bc_0\right\}$ and $D=\tn{diam }\mathcal{D}$. 
\end{theorem}

\begin{proof}
    We have 
    \begin{align*}
        \|z^{t+1}-x^*\|^2
        &=\|\tn{proj}_{\mathcal{D}}(z^t-\h_t\nabla f_{S_t}(x^t))-\tn{proj}_{\mathcal{D}}(x^*)\|\\
        \overset{\text{Lem. \ref{lem:non-exp}}}&{\leq}\|z^t-\h_t\nabla f_{S_t}(x^t)-x^*\|\\
        &=\|z^t-x^*\|^2-2\h_t\langle\nabla f_{S_t}(x^t),z^t-x^*\rangle+\h_t^2\|\nabla f_{S_t}(x^t)\|^2.
    \end{align*}
    Now we continue exactly like the rest of the proof of \Cref{thm:ima-dec-sps}. 
\end{proof}

\begin{theorem}[Pojected IMA version of \Cref{thm:ima-ada-sps}]
    Assume that each $f_i$ is convex and $L_i$-smooth. Then, the iterates of Projected IMA with 
    \begin{align*}
        \h_t=\min\left\{\frac{f_{S_t}(x^t)-\ell_{S_t}^*}{c\|\nabla f_{S_t}(x^t)\|^2}\frac{1}{\sqrt{\sum_{s=0}^tf_{S_s}(x^s)-\ell_{S_s}^*}},\h_{t-1}\right\}\text{ and }\l_t\geq0,
    \end{align*}
    where $\h_{-1}=+\infty$, $c>0$ and $(\l_t)$ a decreasing sequence, converge as 
    \begin{align*}
        \E[f(\overline{x}^T)-f(x^*)]\leq\frac{\t^2}{T}+\frac{\t\s}{\sqrt{T}},
    \end{align*}
    where $\overline{x}^T=\frac{1}{T}\sum_{t=0}^{T-1}x^t$, $\t=\left(\l_1\sqrt{f(x^0)-f(x^*)}+\frac{cLD^2}{2}+\frac{1}{2c}\right)$ and $D=\tn{diam }\mathcal{D}$. 
\end{theorem}

\begin{proof}
    Same as above. 
\end{proof}

\newpage
\section{Additional Experiments}
\label{sec:add-exps}
\newcommand{\wid}{0.27\textwidth}

\subsection{Extra Deterministic Experiment}
\label{sec:deterministic-app}

Here we have included an extra experiment in the deterministic setting, this one for logistic regression. We have performed a grid search to find the best $\b$ for \ref{eq:mopsmax} since no optimal choices for $\b$ are known for the general convex case. We use the same $\b$ for HB and ALR-MAG for direct comparison. For the step-size of HB we used $\g=2(1-\b)/L$ as recomended in \citet{ghadimi2015global}. The results are presented in \Cref{fig:det_lr}. In this problem we observe that our step-size (\ref{eq:mopsmax}) is the fastest, having similar performace with ALR-MAG and GD Polyak. For the deterministic setting we have used a 12 core AMD Ryzen 5 5600H CPU to run the experiments. 

\begin{figure}[H]
	\centering
    \begin{subfigure}{0.25\textwidth}
		\includegraphics[width=\textwidth]{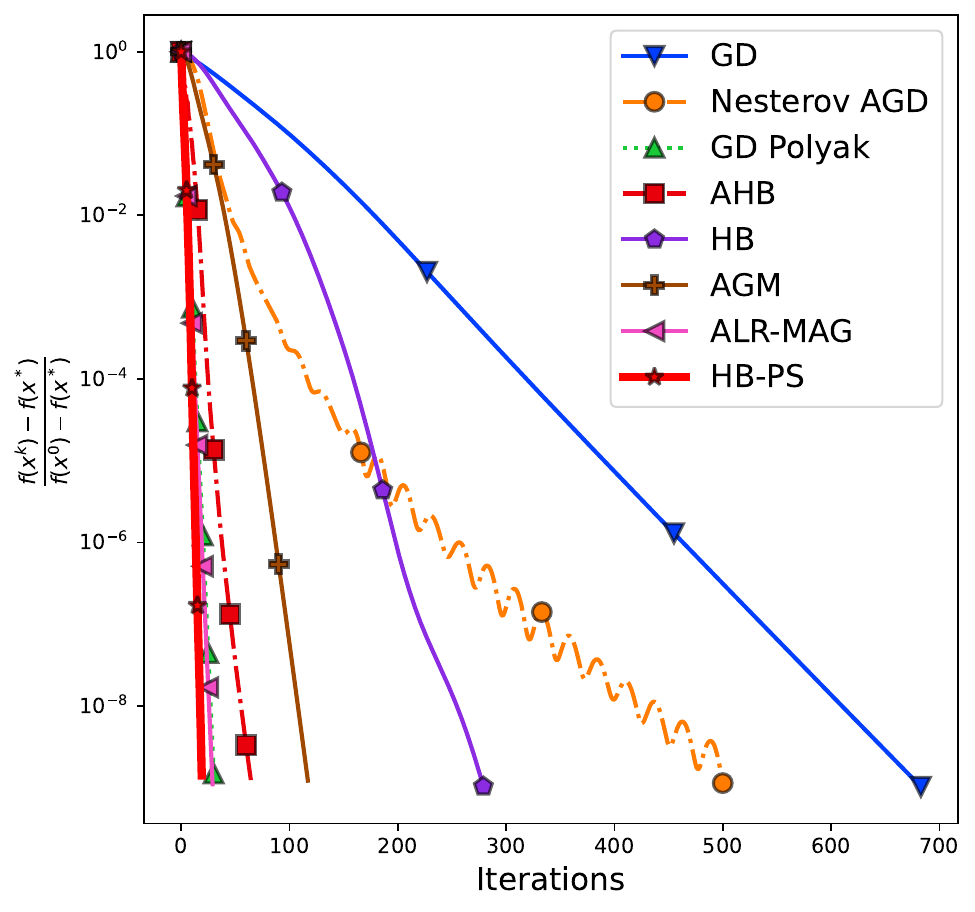}
		\caption{Relative Error Plot}
		\label{subfig:det_lr}
	\end{subfigure}
    ~
	\begin{subfigure}{0.25\textwidth}
		\includegraphics[width=\textwidth]{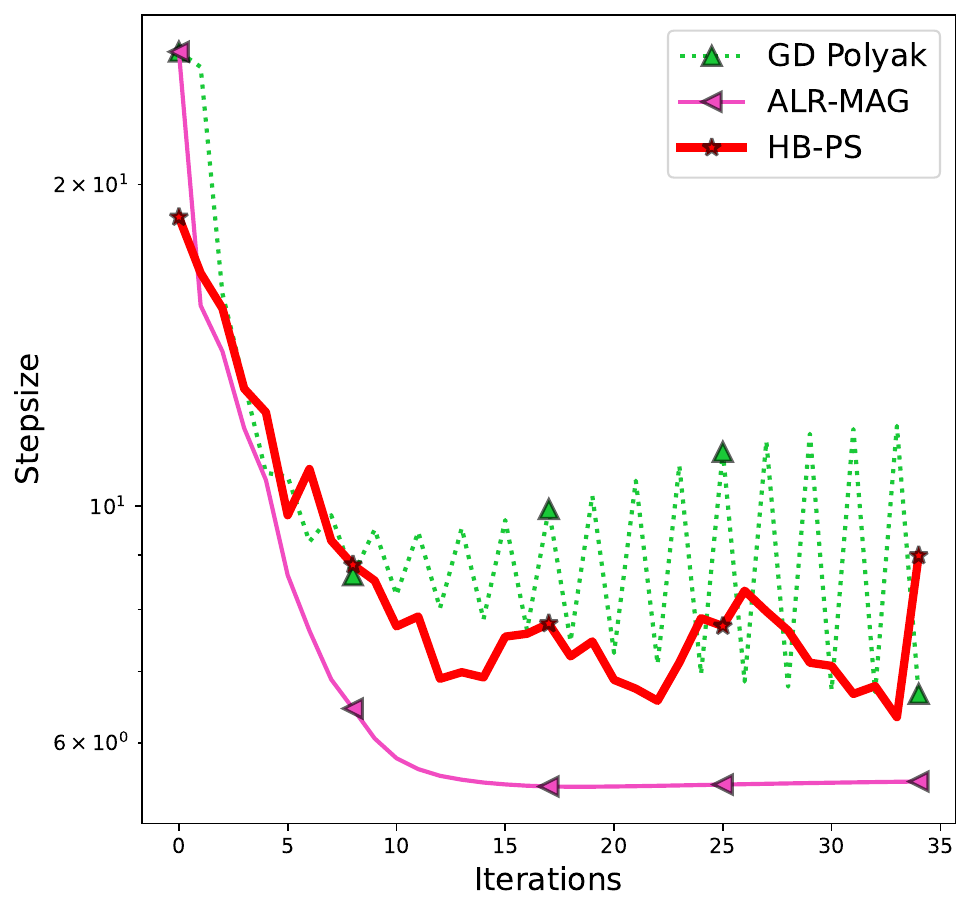}
		\caption{Step-size Plot}
		\label{subfig:det_lr_ss}
	\end{subfigure}

	\caption{Comparison of various deterministic algorithms for the logistic regression problem on synthetic data.}
	\label{fig:det_lr}
\end{figure}

\subsection{Stochastic Heavy Ball with Constant Step-size}
\label{sec:shb-const}

As we saw in \Cref{cor:shb-const}, when $\g_b\leq\frac{1}{2L_{\max}}$ then we have new theoretical guarantees for \ref{eq:shb} with the constant step-size $\g=\frac{1-\b}{2L_{\max}}$. The most recent analysis of constant \ref{eq:shb} is from \citet{liu2020improved} with step-size $\g=\frac{(1-\b)^2}{L}\min\left\{\frac{1}{4-\b+\b^2}, \frac{1}{2\sqrt{2\b+2\b^2}}\right\}$. As mentioned in the main text our step-size is larger when $\b\geq\sqrt{5}-2\approx0.236$. In \Cref{fig:shb_const} we provide a numerical comparison of these results for logistic regression with synthetic data. We observe that \ref{eq:shb} with the \citet{liu2020improved} step-size has faster convergence for the first iterations but it reaches a plateau much earlier. Moreover, note that as $\b\to1$ both step-sizes have similar performance. 

\begin{figure}[t]
	\centering
	\begin{subfigure}{0.2\textwidth}
		\includegraphics[width=\textwidth]{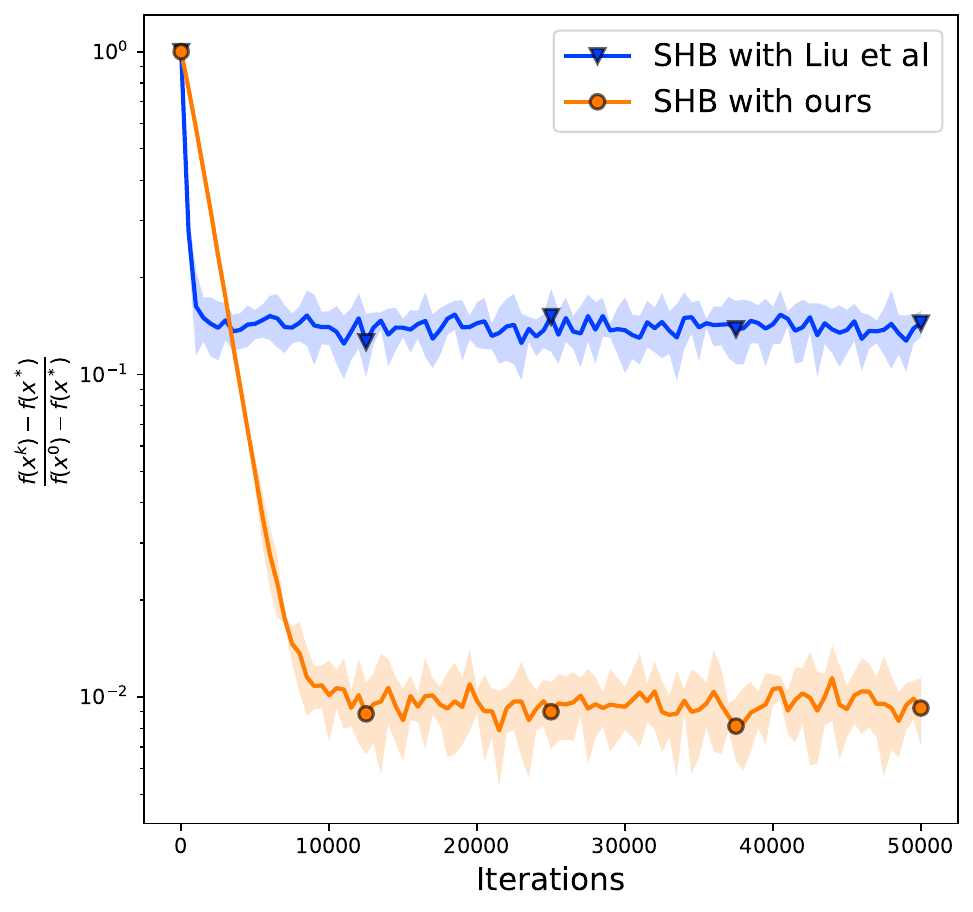}
		\caption{$\b=0.1$}
		\label{subfig:shb_const/b-0.1}
	\end{subfigure}
    ~
    \begin{subfigure}{0.2\textwidth}
		\includegraphics[width=\textwidth]{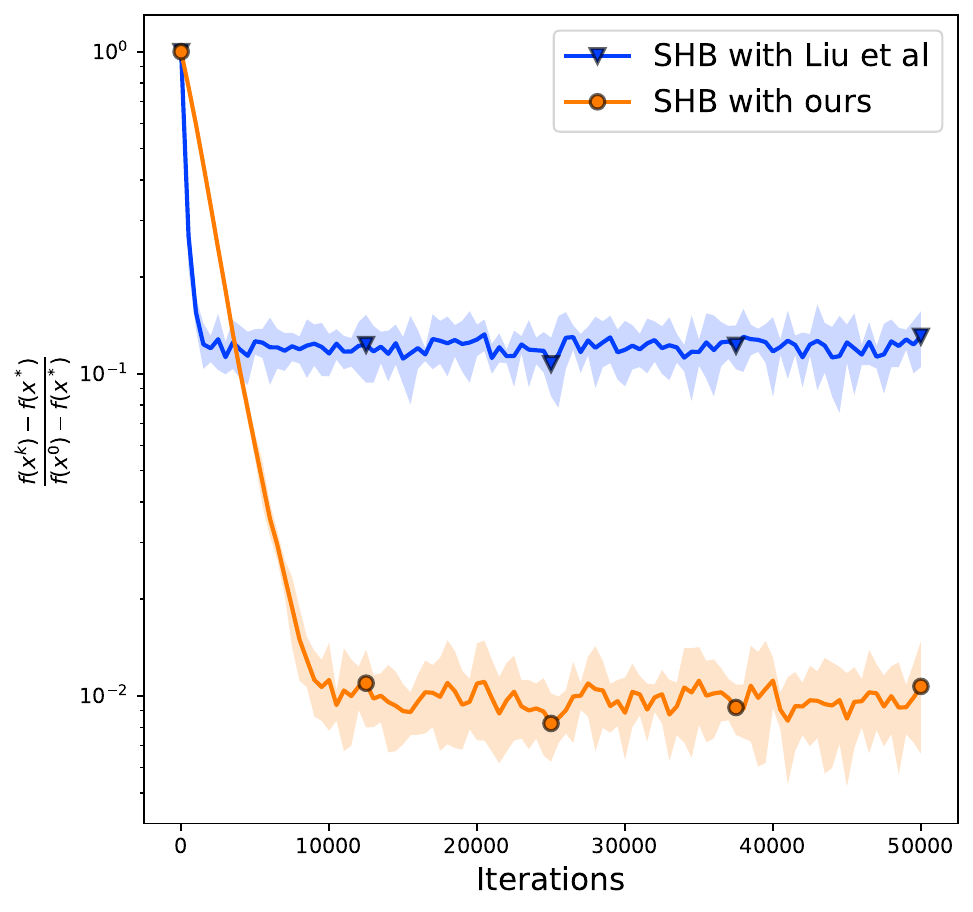}
		\caption{$\b=0.2$}
		\label{subfig:shb_const/b-0.2}
	\end{subfigure}
    ~
    \begin{subfigure}{0.2\textwidth}
		\includegraphics[width=\textwidth]{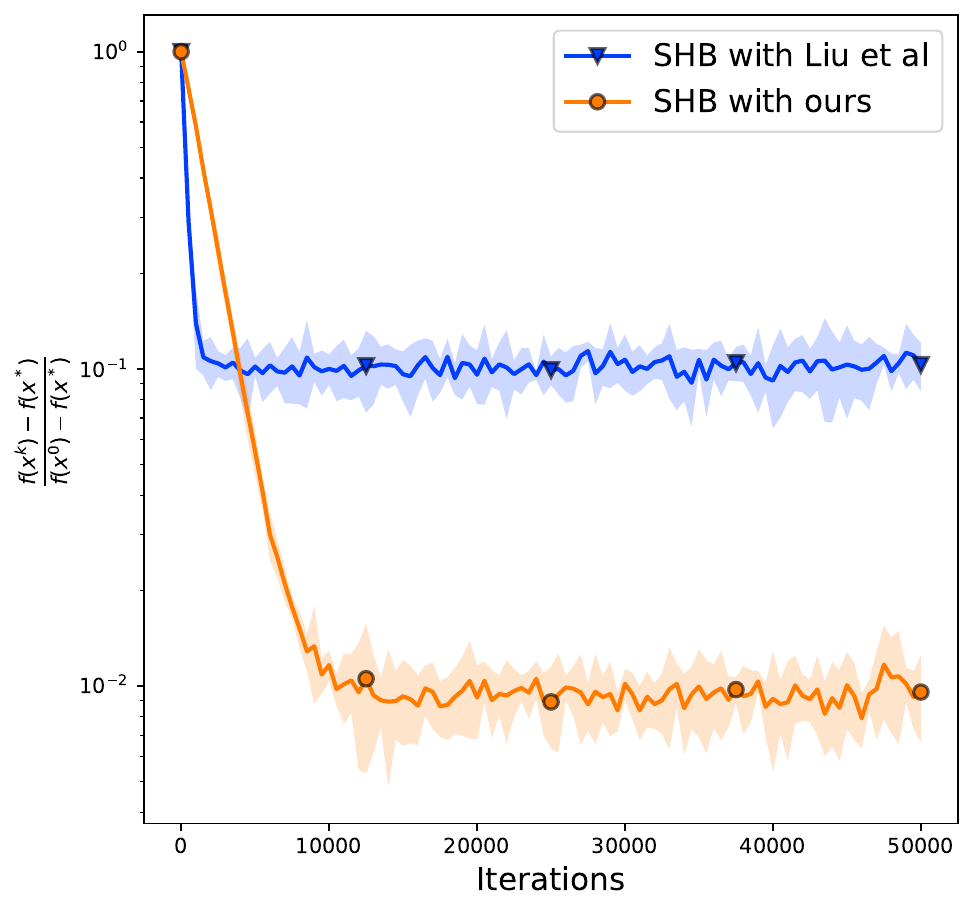}
		\caption{$\b=0.3$}
		\label{subfig:shb_const/b-0.3}
	\end{subfigure}
    ~
    \begin{subfigure}{0.2\textwidth}
		\includegraphics[width=\textwidth]{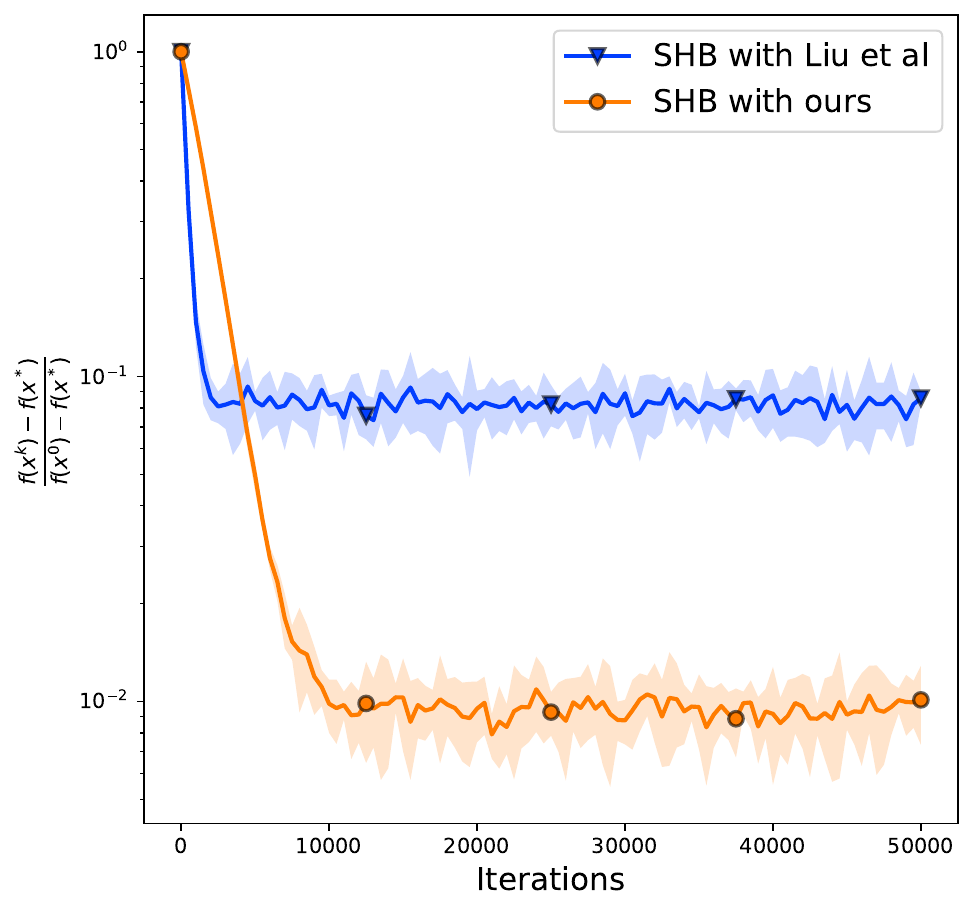}
		\caption{$\b=0.4$}
		\label{subfig:shb_const/b-0.4}
	\end{subfigure}
    ~
    \begin{subfigure}{0.2\textwidth}
		\includegraphics[width=\textwidth]{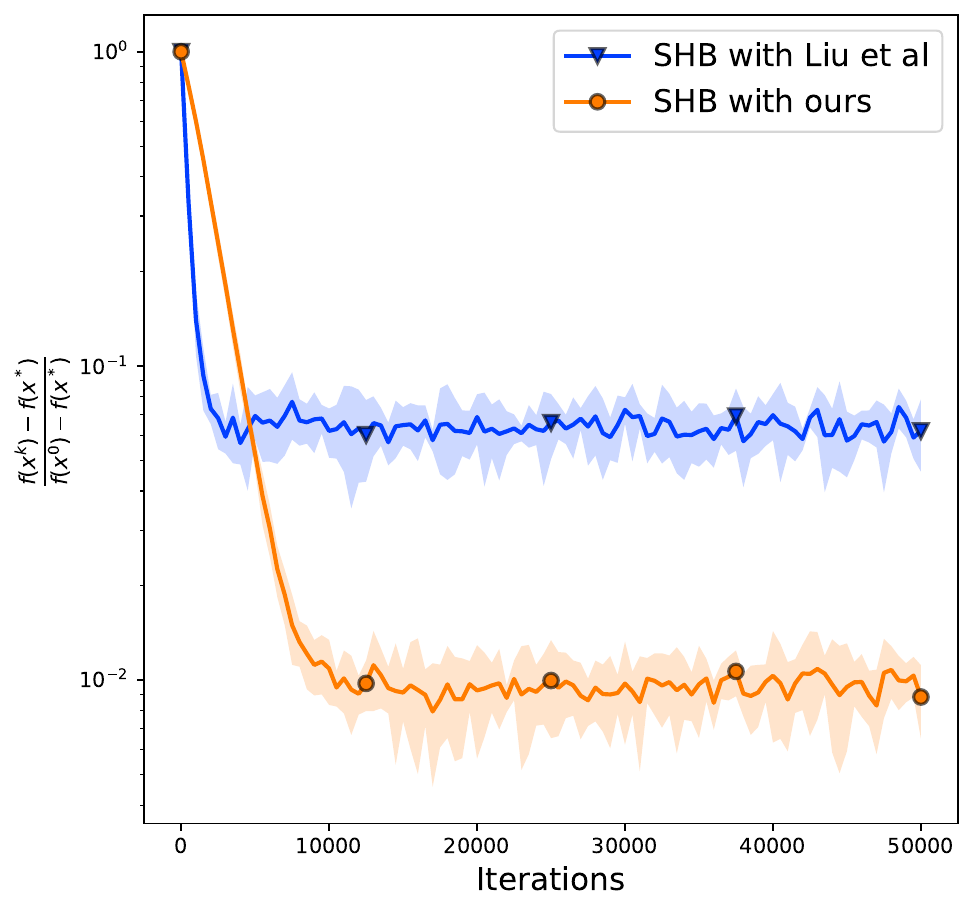}
		\caption{$\b=0.5$}
		\label{subfig:shb_const/b-0.5}
	\end{subfigure}
    ~
    \begin{subfigure}{0.2\textwidth}
		\includegraphics[width=\textwidth]{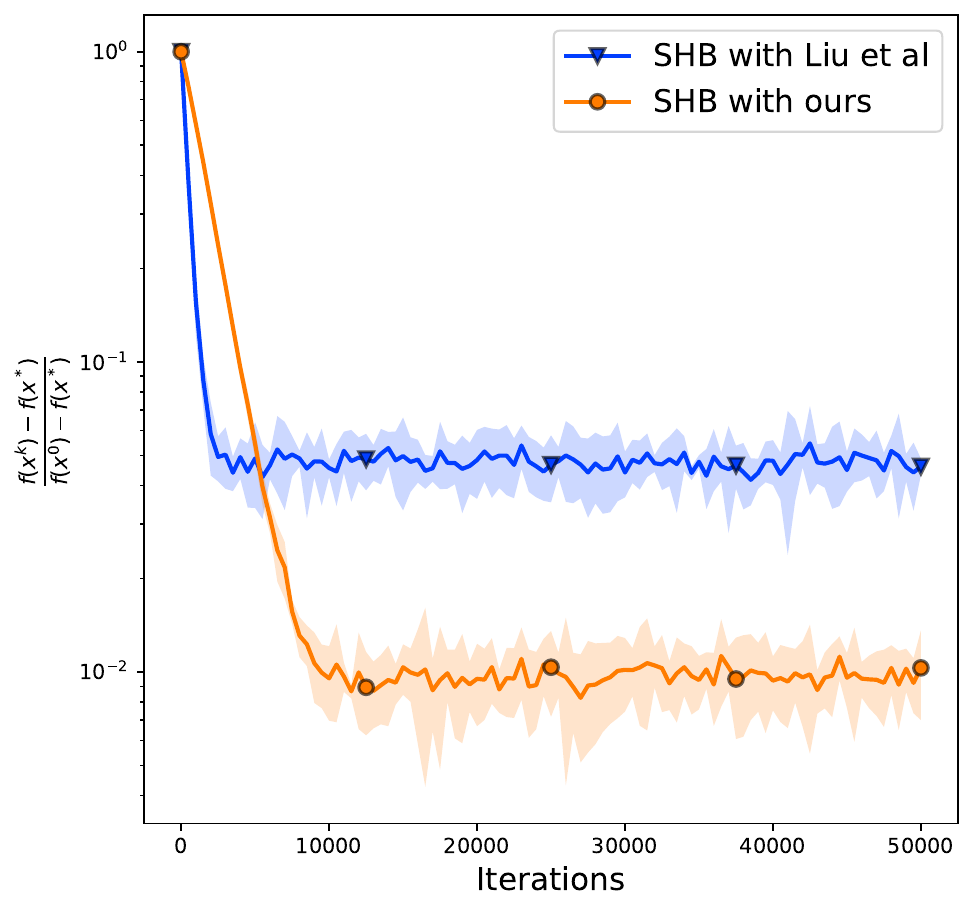}
		\caption{$\b=0.6$}
		\label{subfig:shb_const/b-0.6}
	\end{subfigure}
    ~
    \begin{subfigure}{0.2\textwidth}
		\includegraphics[width=\textwidth]{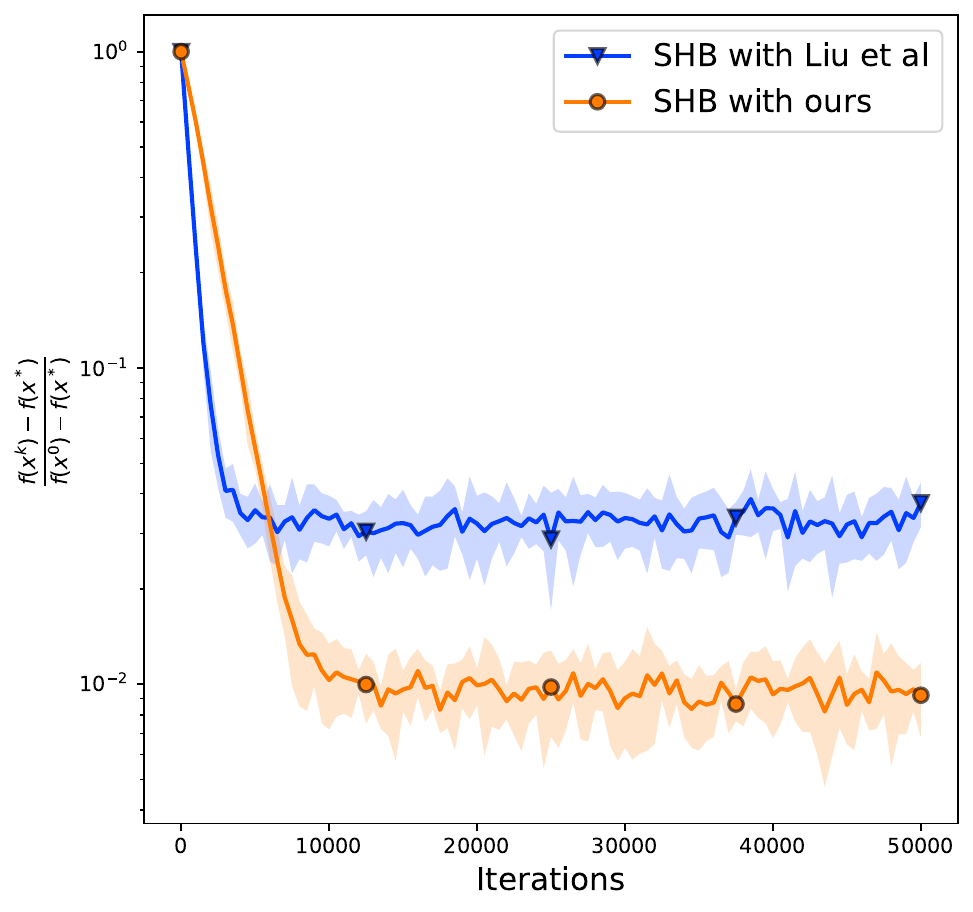}
		\caption{$\b=0.7$}
		\label{subfig:shb_const/b-0.7}
	\end{subfigure}
    ~
    \begin{subfigure}{0.2\textwidth}
		\includegraphics[width=\textwidth]{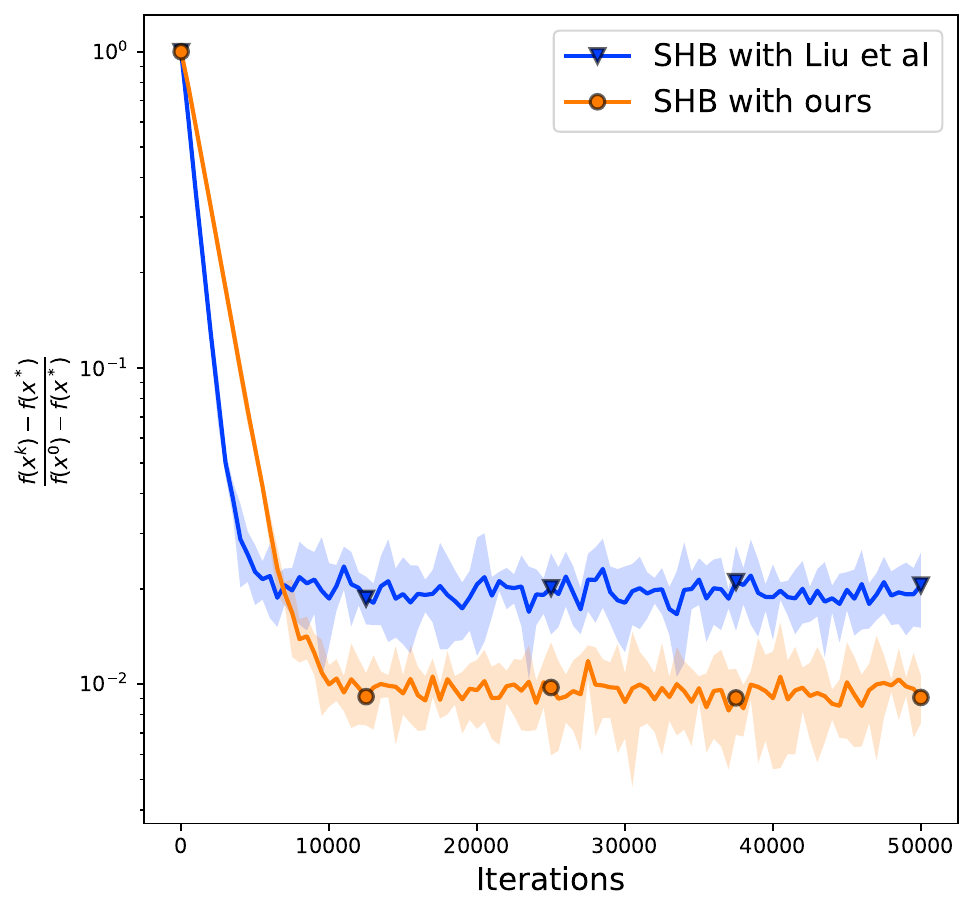}
		\caption{$\b=0.8$}
		\label{subfig:shb_const/b-0.8}
	\end{subfigure}
    ~
    \begin{subfigure}{0.2\textwidth}
		\includegraphics[width=\textwidth]{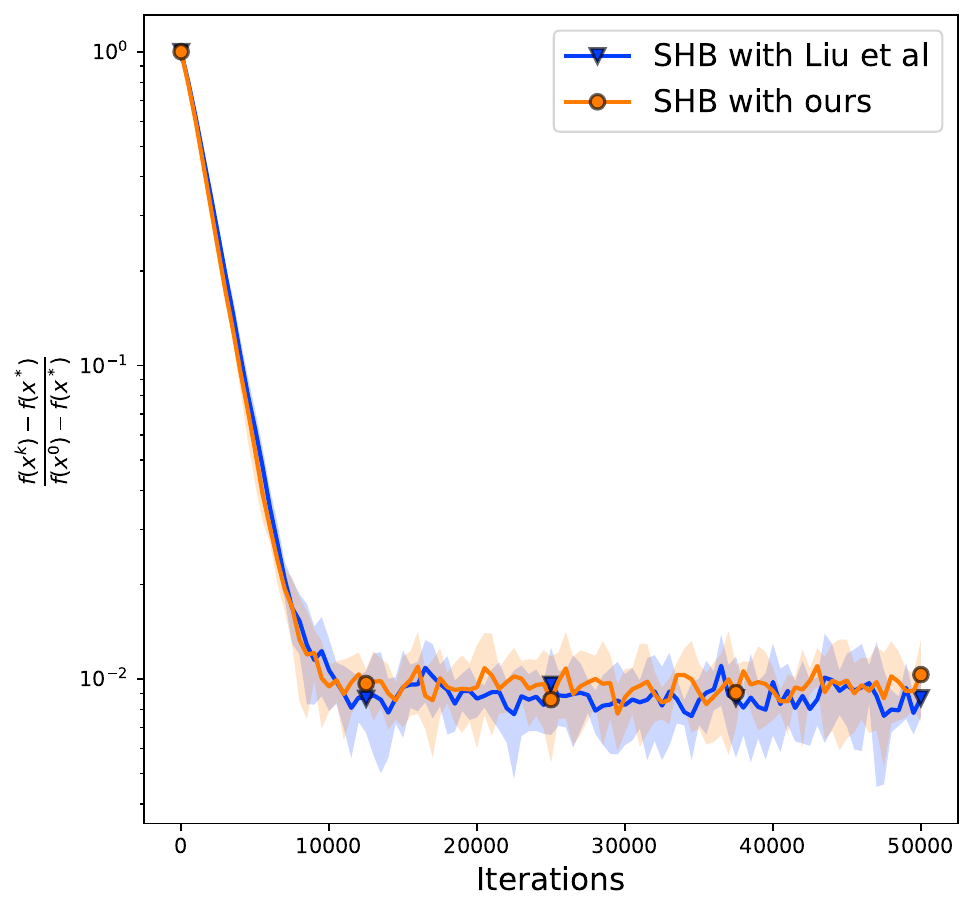}
		\caption{$\b=0.9$}
		\label{subfig:shb_const/b-0.9}
	\end{subfigure}
    
	\caption{Comparison of \ref{eq:shb} with constant step-size. \citep{liu2020improved} vs \Cref{cor:shb-const} for various momentum coefficients $\b$ on logistic regression with synthetic data.}
	\label{fig:shb_const}
\end{figure}

\subsection{Other Choices}
\label{sec:other-choices}

In this paper, we have provided convergence guarantees for \ref{eq:shb} update rule $x^{t+1}=x^t-\g_t\nabla f_{S_t}(x^t)+\b(x^t-x^{t-1})$ with the \ref{eq:mospsmax} step-size given by $\g_t=(1-\b)\min\left\{\frac{f_{S_t}(x^t)-\ell_{S_t}^*}{c\|\nabla f_{S_t}(x^t)\|^2}, \g_b\right\}$. However, as discussed in the main paper a more natural step-size would be to choose the well known SPS$_{\max}$ on the \ref{eq:shb} rule. We call this new update rule SPS$_{\max}$ with naive momentum. Here we numerically compare these two updates. As we see in \Cref{fig:mosps_vs_naive}, when $\b$ is \textquote{small} and close to $0$ then both SPS$_{\max}$ with naive momentum and \ref{eq:mospsmax} are close in performance. Of course, when $\b=0$ they are both equal to standard SPS$_{\max}$. However, as $\b$ gets larger the performance of SPS$_{\max}$ with naive momentum gets worse and less stable and at some point it diverges. For $\b\geq0.7$ there is no SPS$_{\max}$ with naive momentum, since it cannot be numerically computed. 

\begin{figure}[H]
	\centering
    \begin{subfigure}{0.2\textwidth}
		\includegraphics[width=\textwidth]{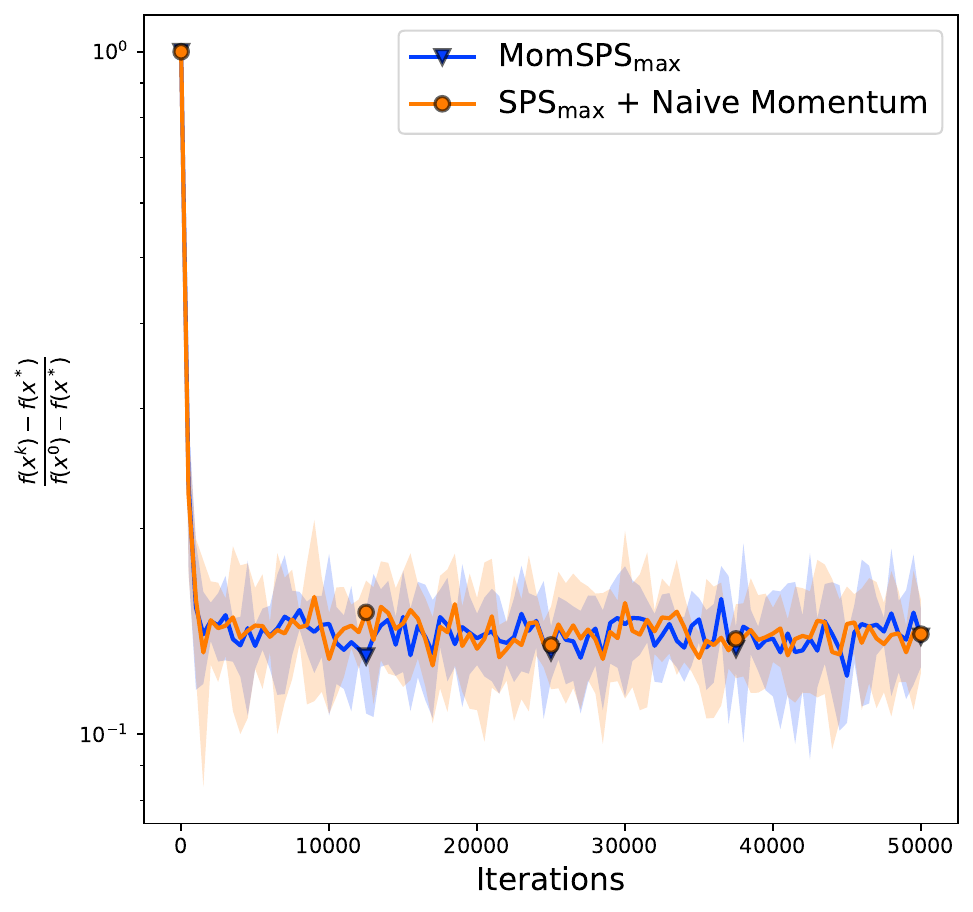}
		\caption{$\b=0.0$}
		\label{subfig:mosps_vs_naive/b-0.0.pdf}
	\end{subfigure}
    ~
	\begin{subfigure}{0.2\textwidth}
		\includegraphics[width=\textwidth]{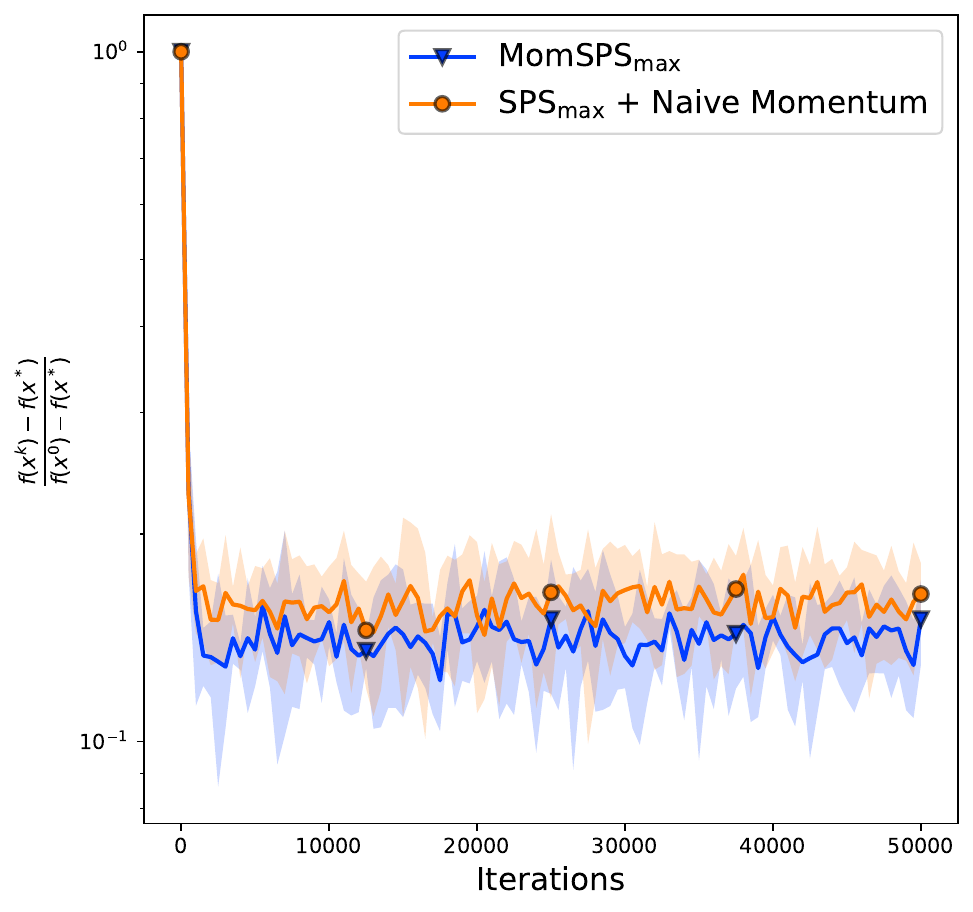}
		\caption{$\b=0.1$}
		\label{subfig:mosps_vs_naive/b-0.1.pdf}
	\end{subfigure}
    ~
    \begin{subfigure}{0.2\textwidth}
		\includegraphics[width=\textwidth]{figures/SHB+SPS/mosps_vs_naive/b-0.2.pdf}
		\caption{$\b=0.2$}
		\label{subfig:mosps_vs_naive/b-0.2.pdf}
	\end{subfigure}
    ~
	\begin{subfigure}{0.2\textwidth}
	   \includegraphics[width=\textwidth]{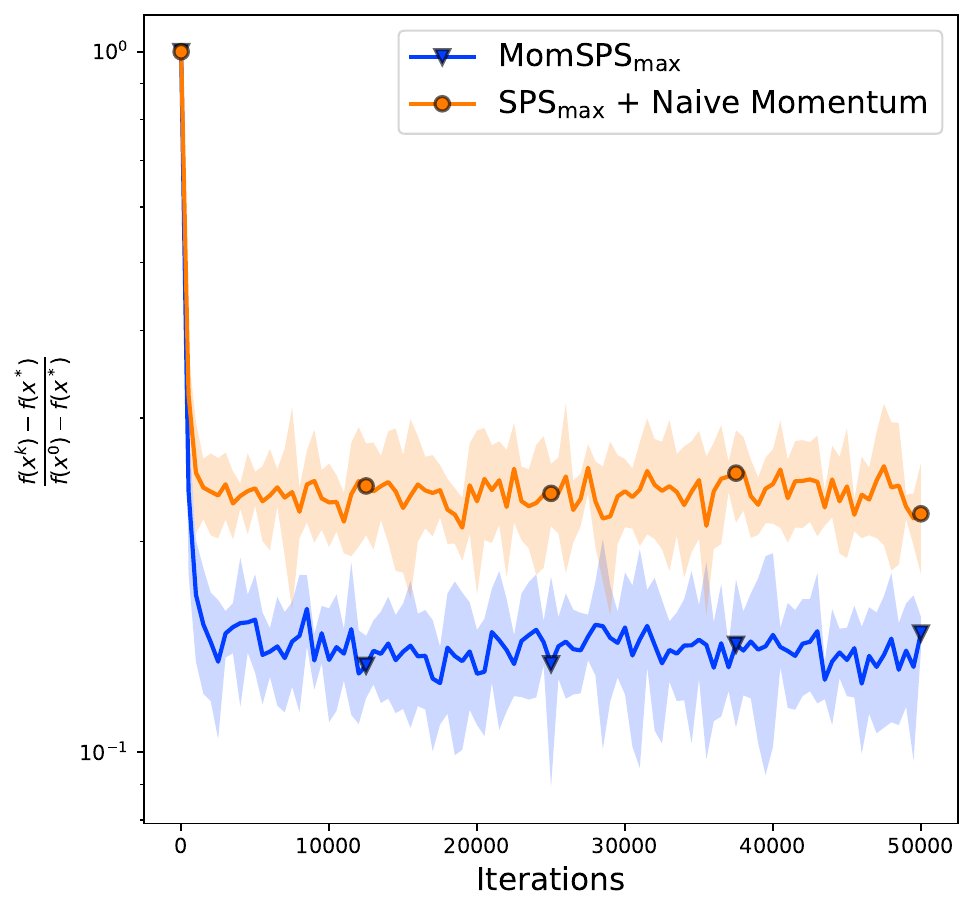}
		\caption{$\b=0.3$}
		\label{subfig:mosps_vs_naive/b-0.3.pdf}
	\end{subfigure}
    ~
    \begin{subfigure}{0.2\textwidth}
		\includegraphics[width=\textwidth]{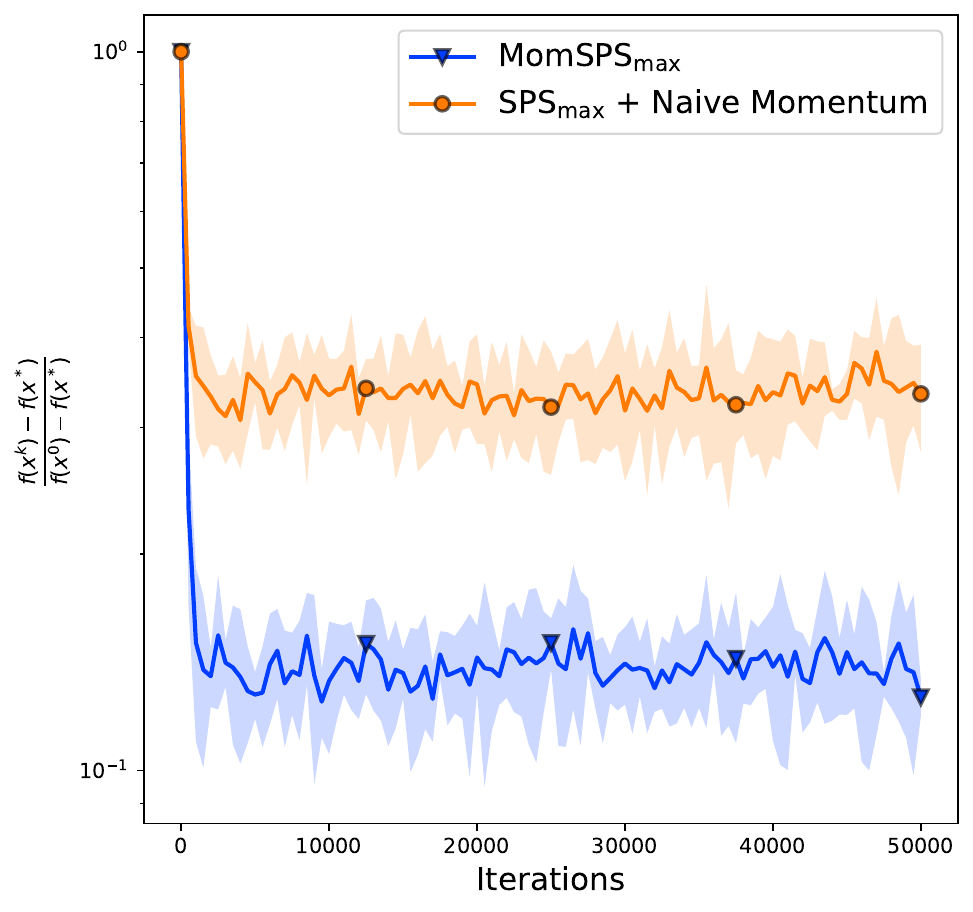}
		\caption{$\b=0.4$}
		\label{subfig:mosps_vs_naive/b-0.4.pdf}
	\end{subfigure}
    ~
	\begin{subfigure}{0.2\textwidth}
		\includegraphics[width=\textwidth]{figures/SHB+SPS/mosps_vs_naive/b-0.5.pdf}
		\caption{$\b=0.5$}
		\label{subfig:mosps_vs_naive/b-0.5.pdf}
	\end{subfigure}
    ~
    \begin{subfigure}{0.2\textwidth}
		\includegraphics[width=\textwidth]{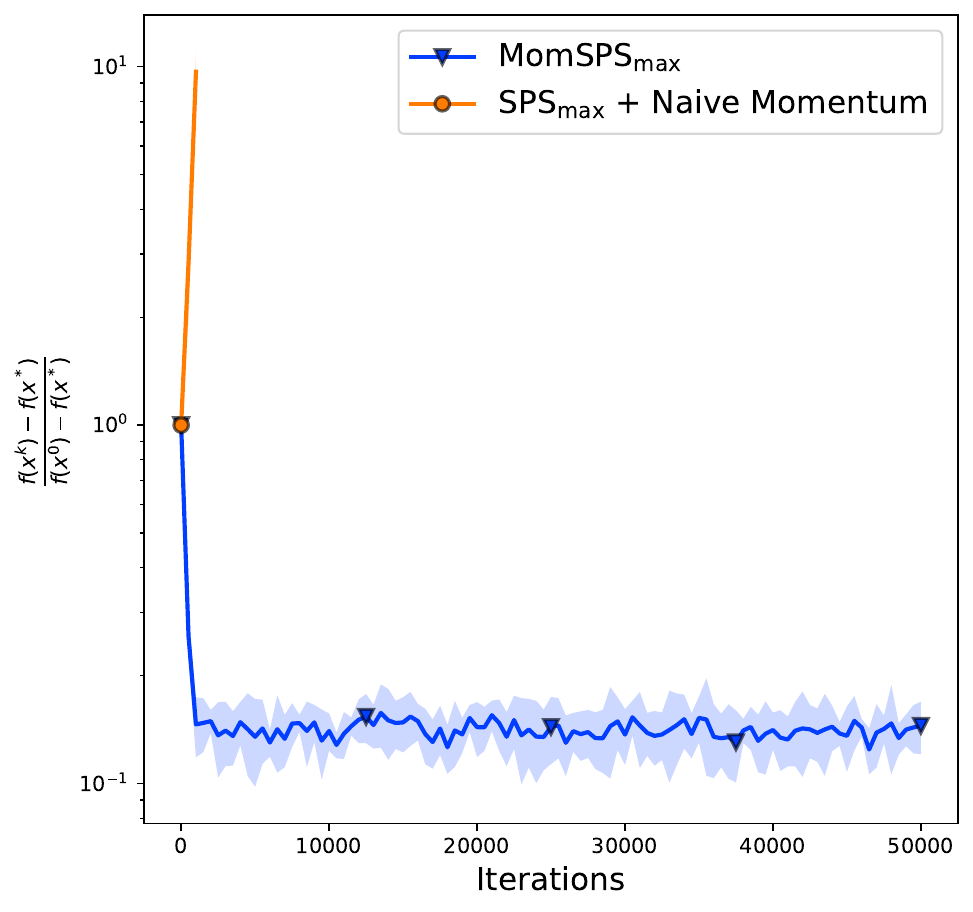}
		\caption{$\b=0.6$}
		\label{subfig:mosps_vs_naive/b-0.6.pdf}
	\end{subfigure}
    ~
	\begin{subfigure}{0.2\textwidth}
		\includegraphics[width=\textwidth]{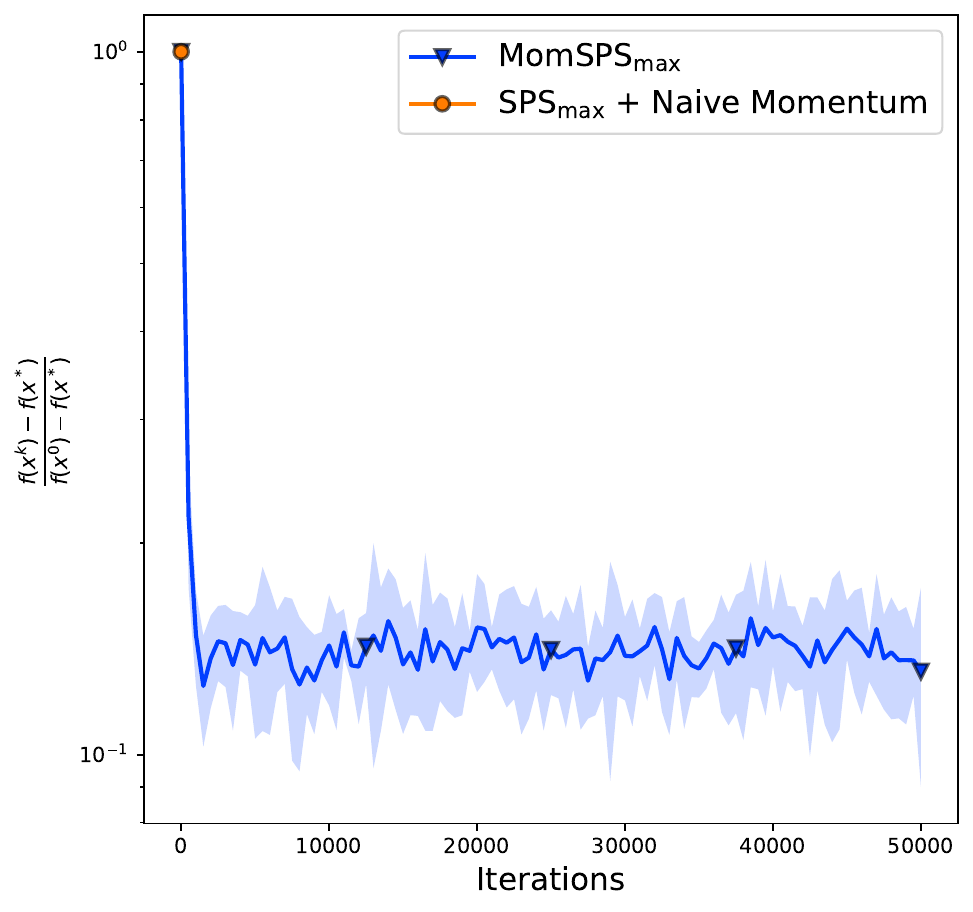}
		\caption{$\b=0.7$}
		\label{subfig:mosps_vs_naive/b-0.7.pdf}
	\end{subfigure}
    
	\caption{Comparison of \ref{eq:mospsmax} versus SPS$_{\max}$ with naive momentum for various momentum coefficients $\b$ on logistic regression with synthetic data.}
	\label{fig:mosps_vs_naive}
\end{figure}

Furthermore, recall that our decreasing variants \ref{eq:modecsps} and \ref{eq:moadasps} are respectively given by 
\begin{align*}
    &\g_t=\min\left\{\frac{(1-\b)[f_{S_t}(x^t)-\ell_{S_t}^*]}{c\sqrt{t+1}\|\nabla f_{S_t}(x^t)\|^2},\frac{\g_{t-1}\sqrt{t}}{\sqrt{t+1}}\right\},\tag{MomDecSPS}\\
    &\g_t=\min\left\{\frac{(1-\b)[f_{S_t}(x^t)-\ell_{S_t}^*]}{c\|\nabla f_{S_t}(x^t)\|^2\sqrt{\sum_{s=0}^tf_{S_s}(x^s)-\ell_{S_s}^*}},\g_{t-1}\right\}.\tag{MomAdaSPS}
\end{align*}
The above step-sizes are theoretically inspired by IMA, as explained in \Cref{sec:ima}, but again a natural question is what happens if we take the \textquote{correcting factor} $1-\b$ outside the minimum. We call these step-sizes Alternative MomDecSPS and Alternative MomAdaSPS and are given by 
\begin{align*}
    &\g_t=(1-\b)\min\left\{\frac{f_{S_t}(x^t)-\ell_{S_t}^*}{c\sqrt{t+1}\|\nabla f_{S_t}(x^t)\|^2},\frac{\g_{t-1}\sqrt{t}}{\sqrt{t+1}}\right\},\tag{Alt MomDecSPS}\\
    &\g_t=(1-\b)\min\left\{\frac{f_{S_t}(x^t)-\ell_{S_t}^*}{c\|\nabla f_{S_t}(x^t)\|^2\sqrt{\sum_{s=0}^tf_{S_s}(x^s)-\ell_{S_s}^*}},\g_{t-1}\right\}.\tag{Alt MomAdaSPS}
\end{align*}
In \Cref{fig:modecsps_vs_naive} and \Cref{fig:moadasps_vs_naive} we see that when $\b=0$ we have similiar performances, however when $\b>0$ then our proposed step-sizes are significantly better. A possible reason is that for both Alt MomDecSPS and Alt MomadaSPS it holds $\g_t\leq(1-\b)\g_{t-1}$ for all $t$. Thus inductively we get $\g_t\leq(1-\b)^t\g_0$. So when $\b>0$, $(1-\b)^t\to0$ as $t\to\infty$ which means that the step-size is too small and it barely updates $x^t$. 

\begin{figure}[H]
	\centering
    \begin{subfigure}{0.2\textwidth}
		\includegraphics[width=\textwidth]{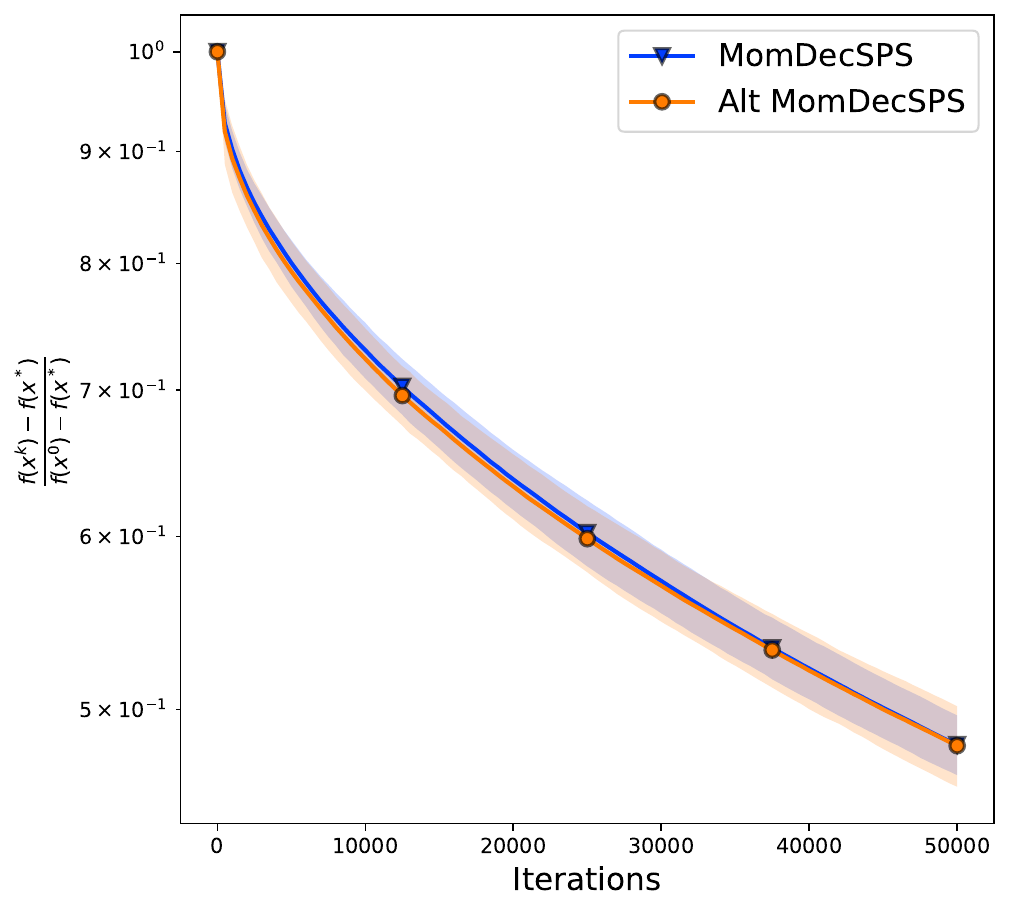}
		\caption{$\b=0.0$}
		\label{subfig:modecsps_vs_naive/b-0.0.pdf}
	\end{subfigure}
    ~
	\begin{subfigure}{0.2\textwidth}
	   \includegraphics[width=\textwidth]{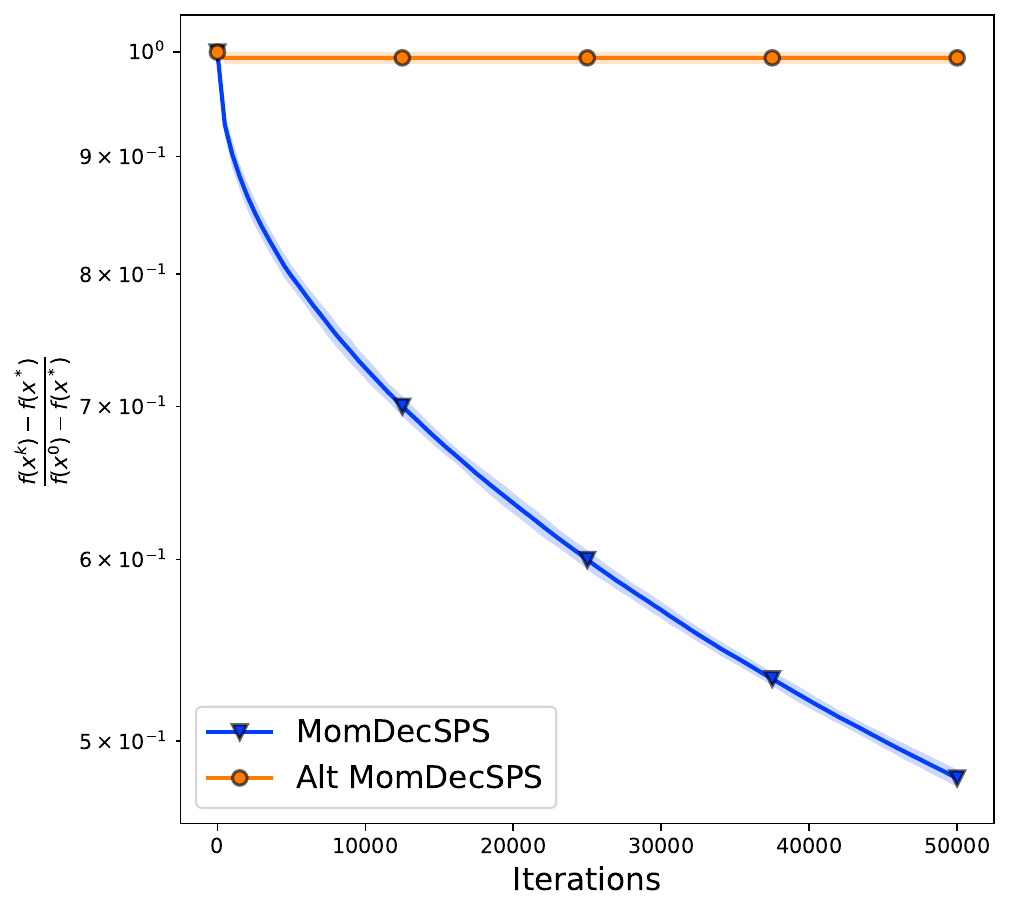}
		\caption{$\b=0.3$}
		\label{subfig:modecsps_vs_naive/b-0.3.pdf}
	\end{subfigure}
    ~
    \begin{subfigure}{0.2\textwidth}
		\includegraphics[width=\textwidth]{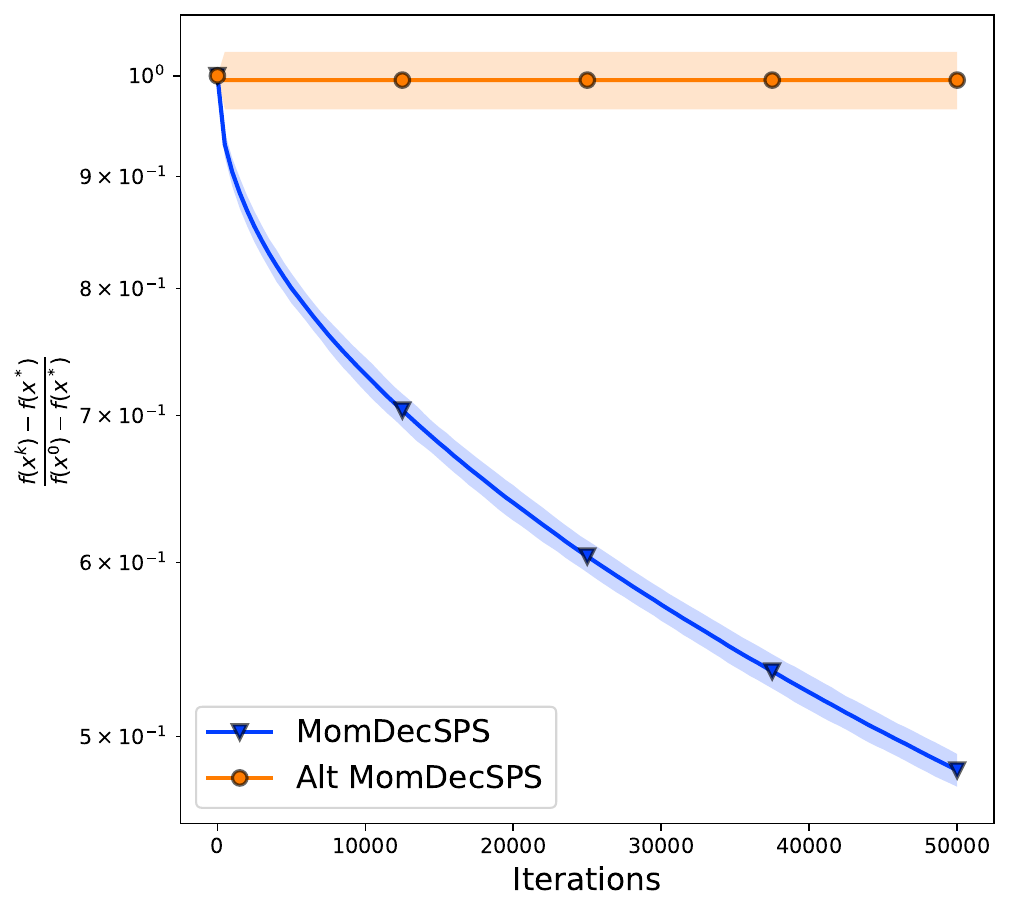}
		\caption{$\b=0.6$}
		\label{subfig:modecsps_vs_naive/b-0.6.pdf}
	\end{subfigure}
    ~
	\begin{subfigure}{0.2\textwidth}
		\includegraphics[width=\textwidth]{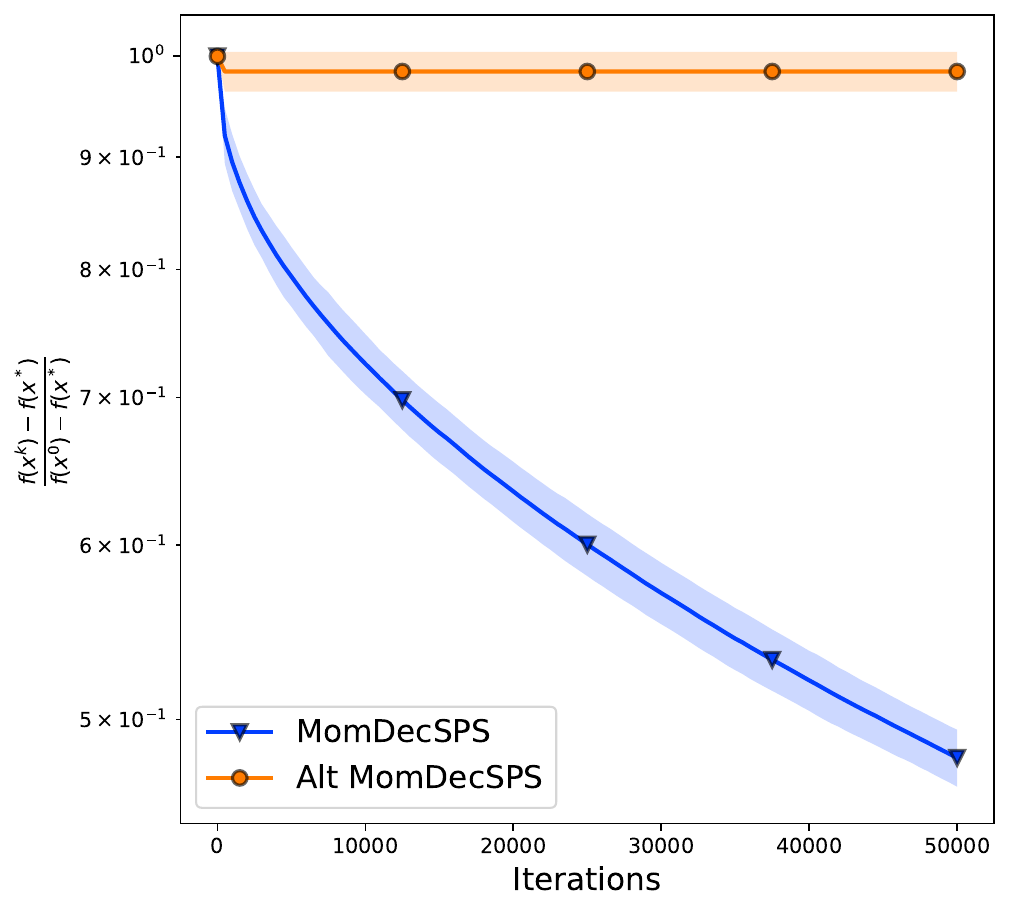}
		\caption{$\b=0.9$}
		\label{subfig:modecsps_vs_naive/b-0.9.pdf}
	\end{subfigure}
    
	\caption{Comparison of \ref{eq:modecsps} versus Alternative MomDecSPS for various momentum coefficients $\b$ on logistic regression with synthetic data.}
	\label{fig:modecsps_vs_naive}
\end{figure}

\begin{figure}[H]
	\centering
    \begin{subfigure}{0.2\textwidth}
		\includegraphics[width=\textwidth]{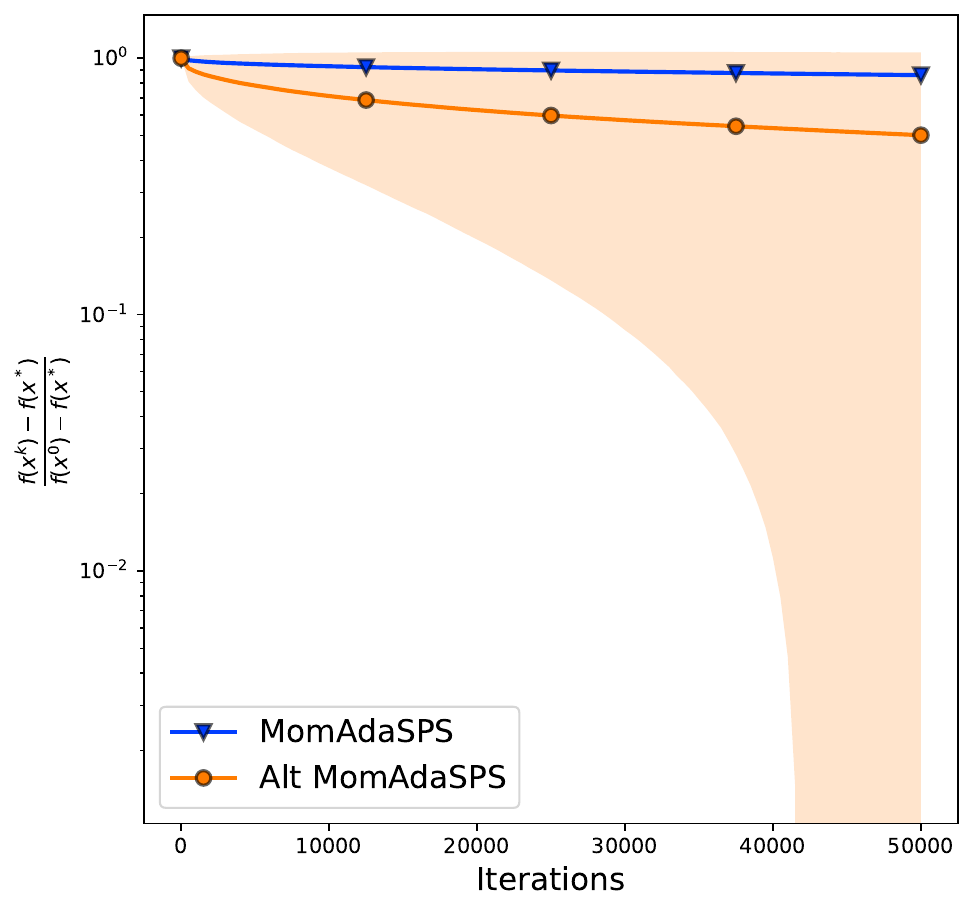}
		\caption{$\b=0.0$}
		\label{subfig:moadasps_vs_naive/b-0.0.pdf}
	\end{subfigure}
    ~
	\begin{subfigure}{0.2\textwidth}
	   \includegraphics[width=\textwidth]{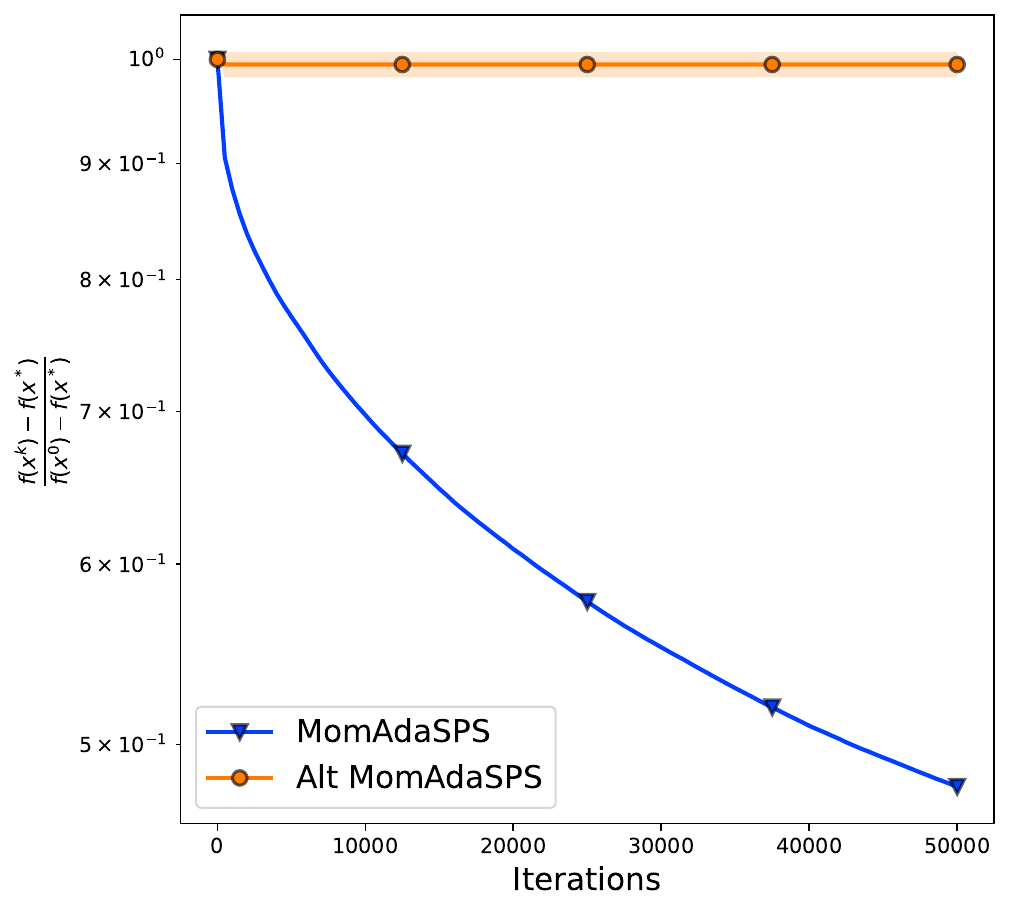}
		\caption{$\b=0.3$}
		\label{subfig:moadasps_vs_naive/b-0.3.pdf}
	\end{subfigure}
    ~
    \begin{subfigure}{0.2\textwidth}
		\includegraphics[width=\textwidth]{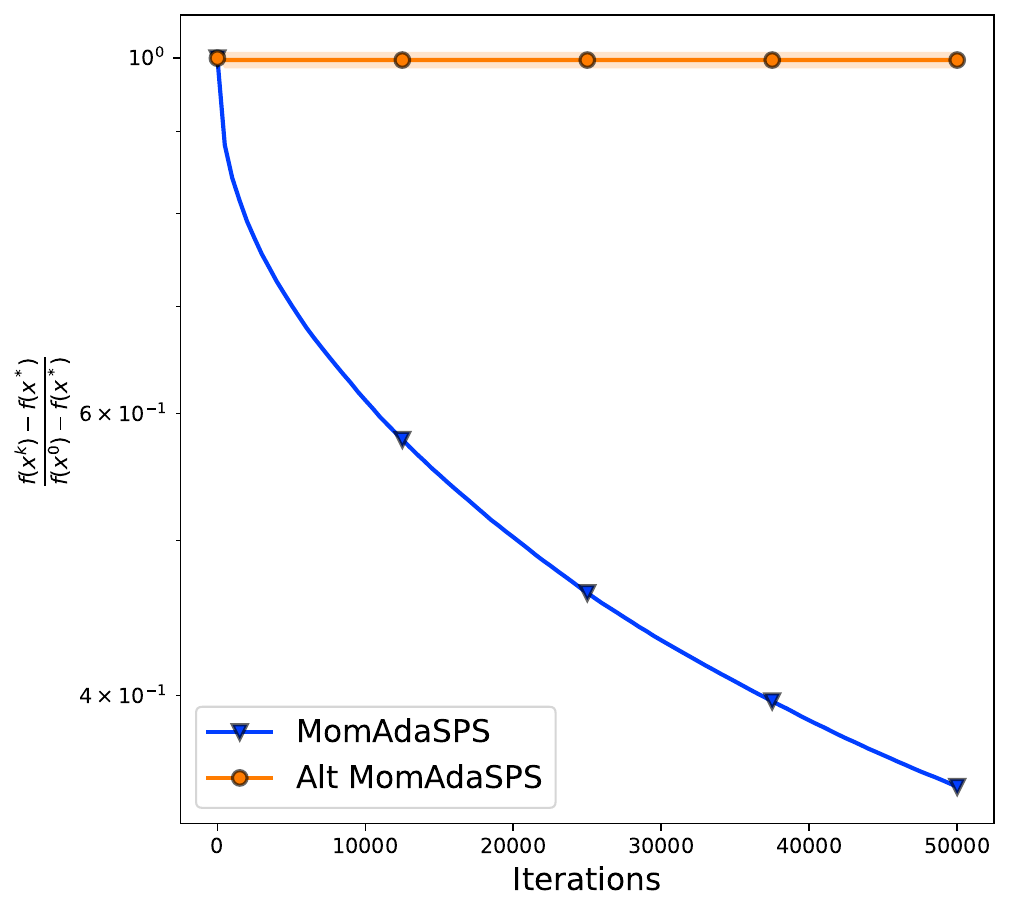}
		\caption{$\b=0.6$}
		\label{subfig:moadasps_vs_naive/b-0.6.pdf}
	\end{subfigure}
    ~
	\begin{subfigure}{0.2\textwidth}
		\includegraphics[width=\textwidth]{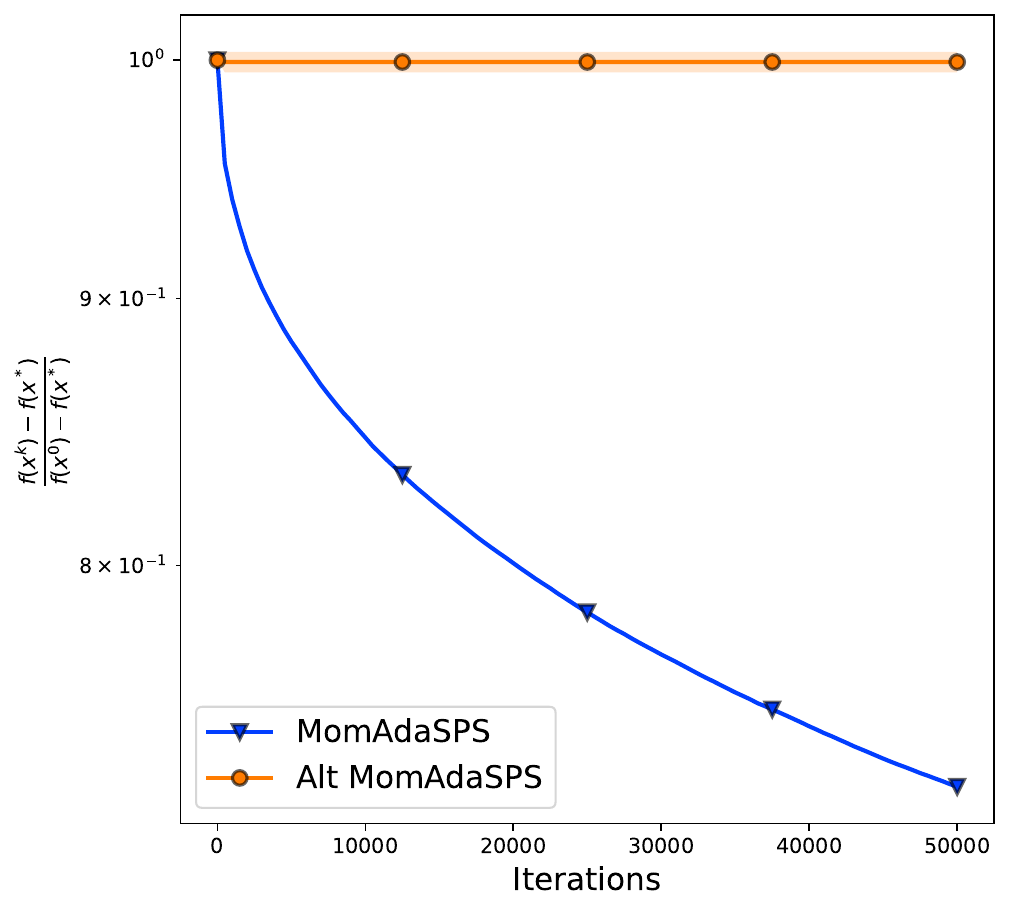}
		\caption{$\b=0.9$}
		\label{subfig:moadasps_vs_naive/b-0.9.pdf}
	\end{subfigure}
    
	\caption{Comparison of \ref{eq:moadasps} versus Alternative MomAdaSPS for various momentum coefficients $\b$ on logistic regression with synthetic data.}
	\label{fig:moadasps_vs_naive}
\end{figure}

\subsection{Is \texorpdfstring{$\g_b$}{gamma b} needed?}
\label{sec:upper-bound-need}

Recall the definition of \ref{eq:mopsmax}, $\g_t=(1-\b)\min\left\{\frac{f(x^t)-f(x^*)}{\|\nabla f(x^t)\|^2},\g_b\right\}$. The more natural choice would be without the upper bound $\g_b$. Moreover, as in \Cref{sec:other-choices}, we can have the naive version of both of the above step-sizes, i.e. with no correcting factor $1-\b$. We compare the following step-sizes on both the least squares problem, in \Cref{fig:det_ls_choices_b}, and logistic regression problem, in \Cref{fig:det_lr_choices_b}, in the deterministic setting on synthetic data. The method is \ref{eq:shb} with $\b=0.97$ for the least squares and $\b=0.3$ for the logistic regression. 
\begin{align*}
    &\g_t=(1-\b)\frac{f(x^t)-f(x^*)}{\|\nabla f(x^t)\|^2}\tag{MomPS}\\
    &\g_t=\frac{f(x^t)-f(x^*)}{\|\nabla f(x^t)\|^2}\tag{Alt MomPS}\\
    &\g_t=(1-\b)\min\left\{\frac{f(x^t)-f(x^*)}{\|\nabla f(x^t)\|^2},\g_b\right\}\tag{MomPS$_{\max}$}\\
    &\g_t=\min\left\{\frac{f(x^t)-f(x^*)}{\|\nabla f(x^t)\|^2},\g_b\right\}\tag{Alt MomPS$_{\max}$}.
\end{align*}
We see that if $\g_b$ is chosen large enough, all these step-sizes have similar performance. However, we highlight the fact that convergence guarantees are known only for \ref{eq:mopsmax}. 

\begin{figure}[H]
	\centering
    \begin{subfigure}{0.2\textwidth}
		\includegraphics[width=\textwidth]{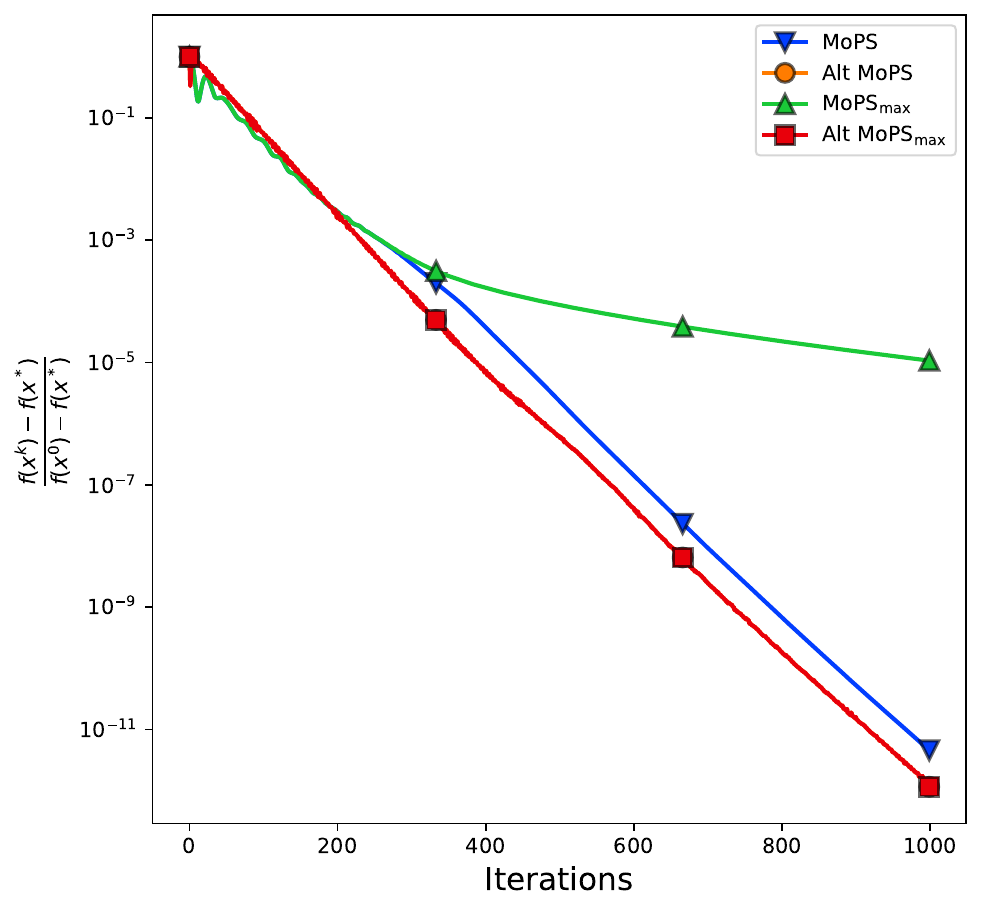}
		\caption{$\g_b=0.1$}
		\label{subfig:det_ls_choices_b-0.1}
	\end{subfigure}
    ~
	\begin{subfigure}{0.2\textwidth}
		\includegraphics[width=\textwidth]{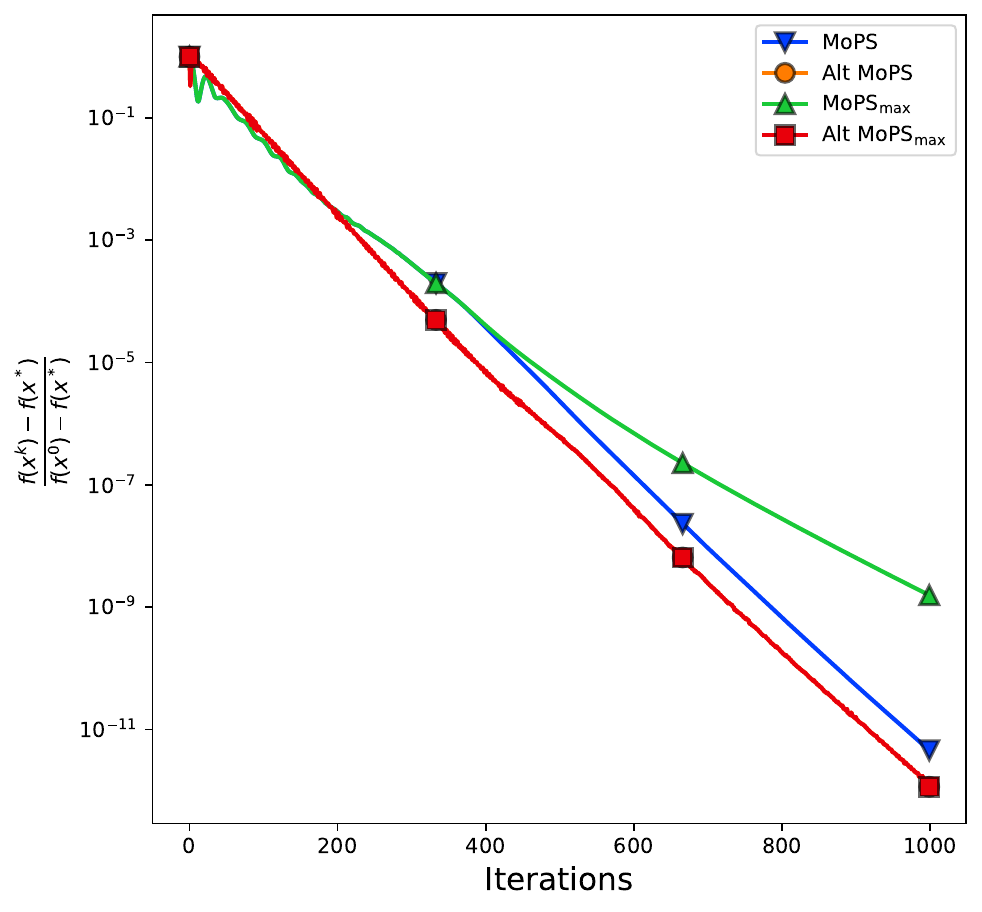}
		\caption{$\g_b=0.5$}
		\label{subfig:det_ls_choices_b-0.5}
	\end{subfigure}
    ~
    \begin{subfigure}{0.2\textwidth}
		\includegraphics[width=\textwidth]{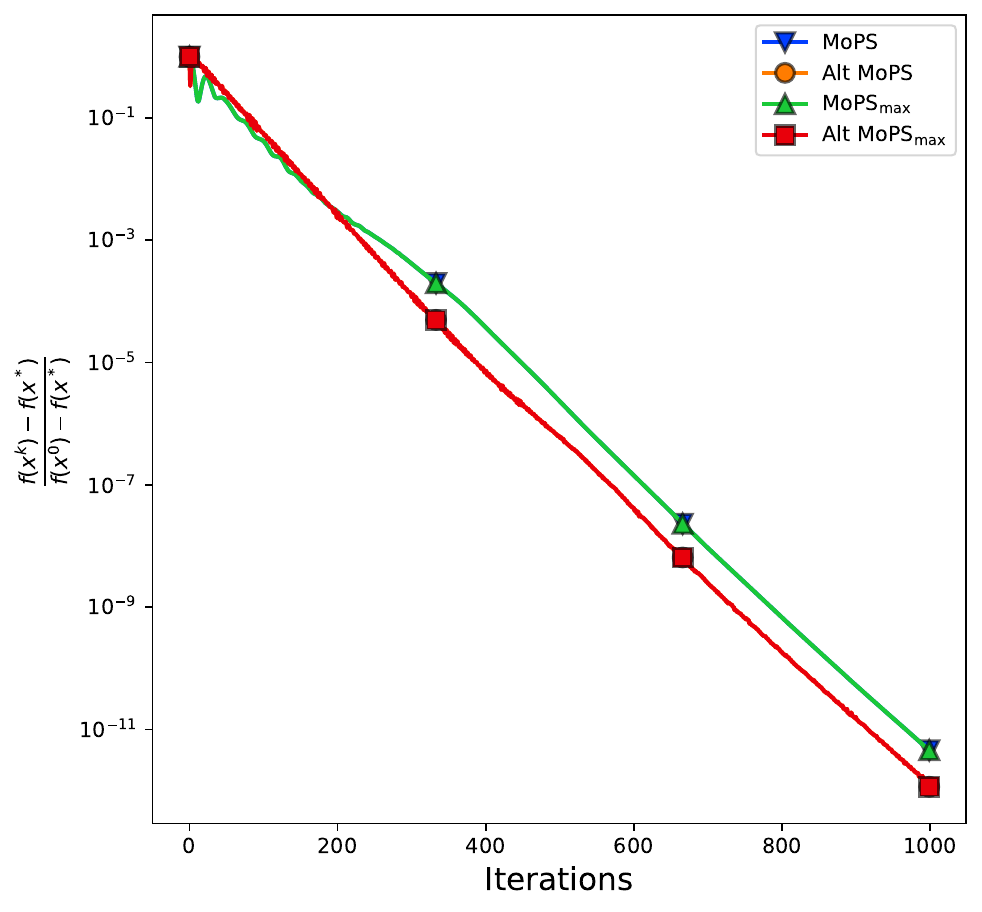}
		\caption{$\g_b=1$}
		\label{subfig:det_ls_choices_b-1}
	\end{subfigure}
    ~
    \begin{subfigure}{0.2\textwidth}
		\includegraphics[width=\textwidth]{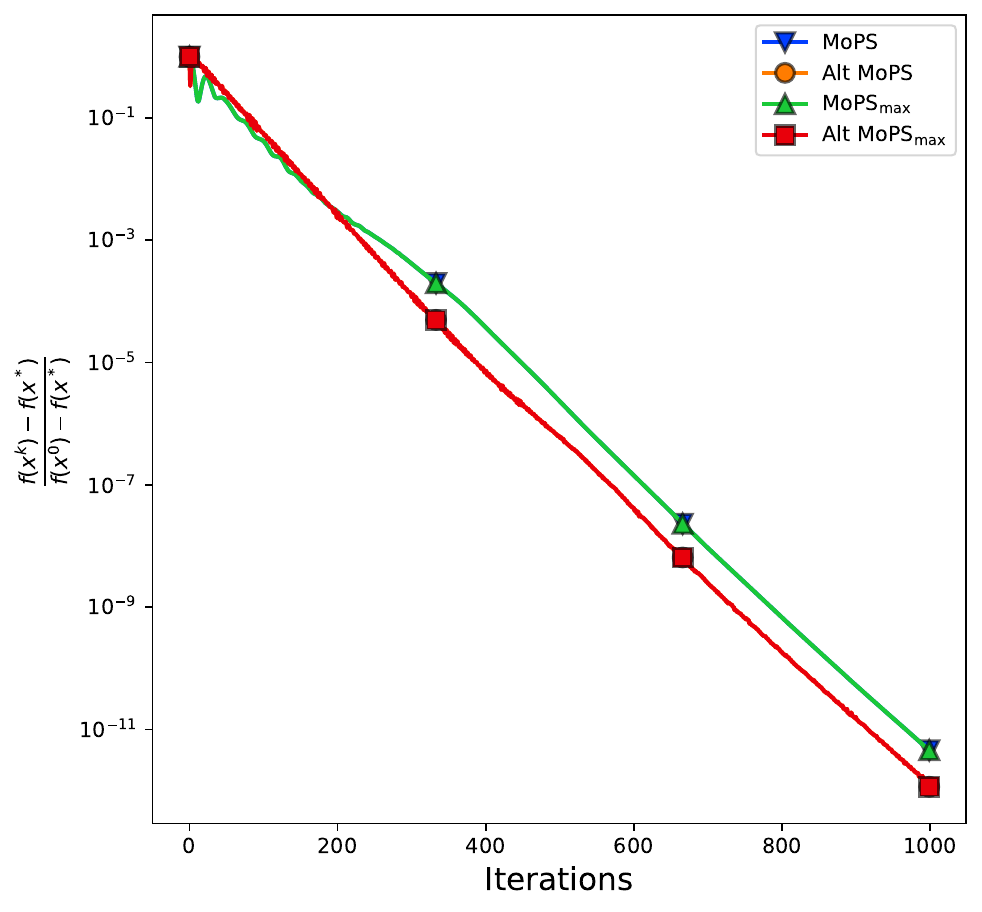}
		\caption{$\g_b=10$}
		\label{subfig:det_ls_choices_b-10}
	\end{subfigure}
    
	\caption{Comparison of various upper bounds for \ref{eq:mopsmax} and other alternatives on least squares with synthetic data.}
	\label{fig:det_ls_choices_b}
\end{figure}

\begin{figure}[H]
	\centering
    \begin{subfigure}{0.2\textwidth}
		\includegraphics[width=\textwidth]{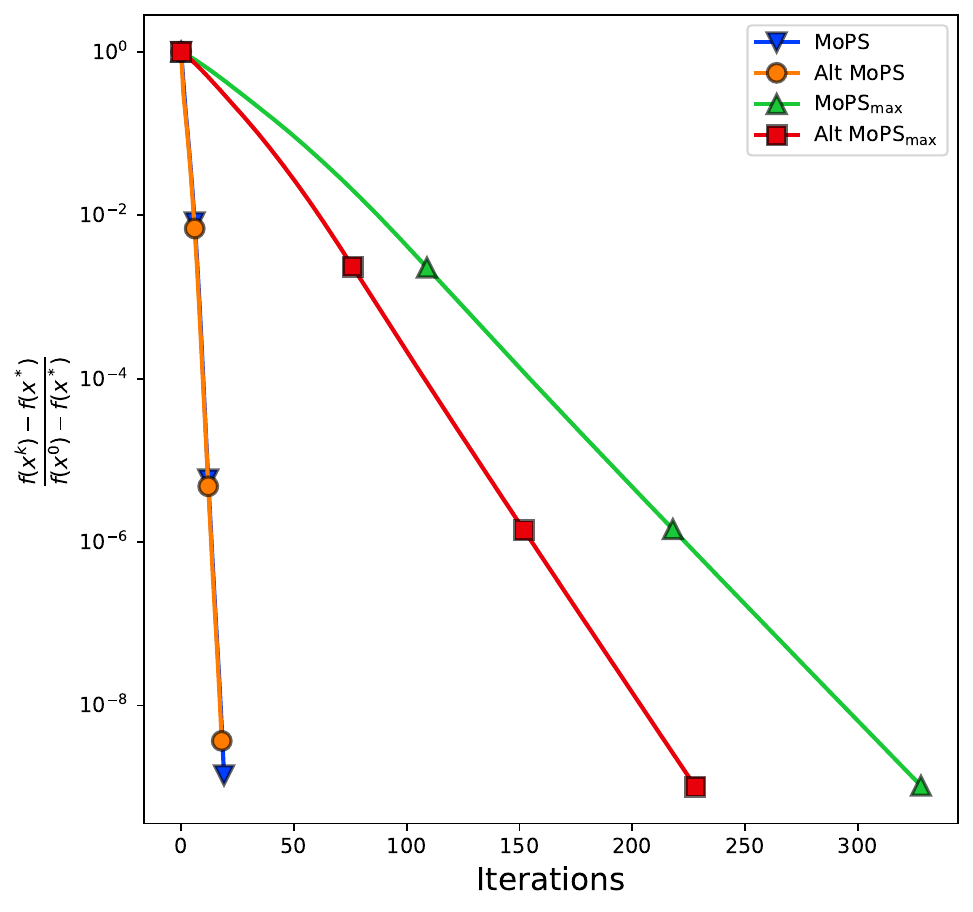}
		\caption{$\g_b=1$}
		\label{subfig:det_lr_choices_b-1}
	\end{subfigure}
    ~
	\begin{subfigure}{0.2\textwidth}
		\includegraphics[width=\textwidth]{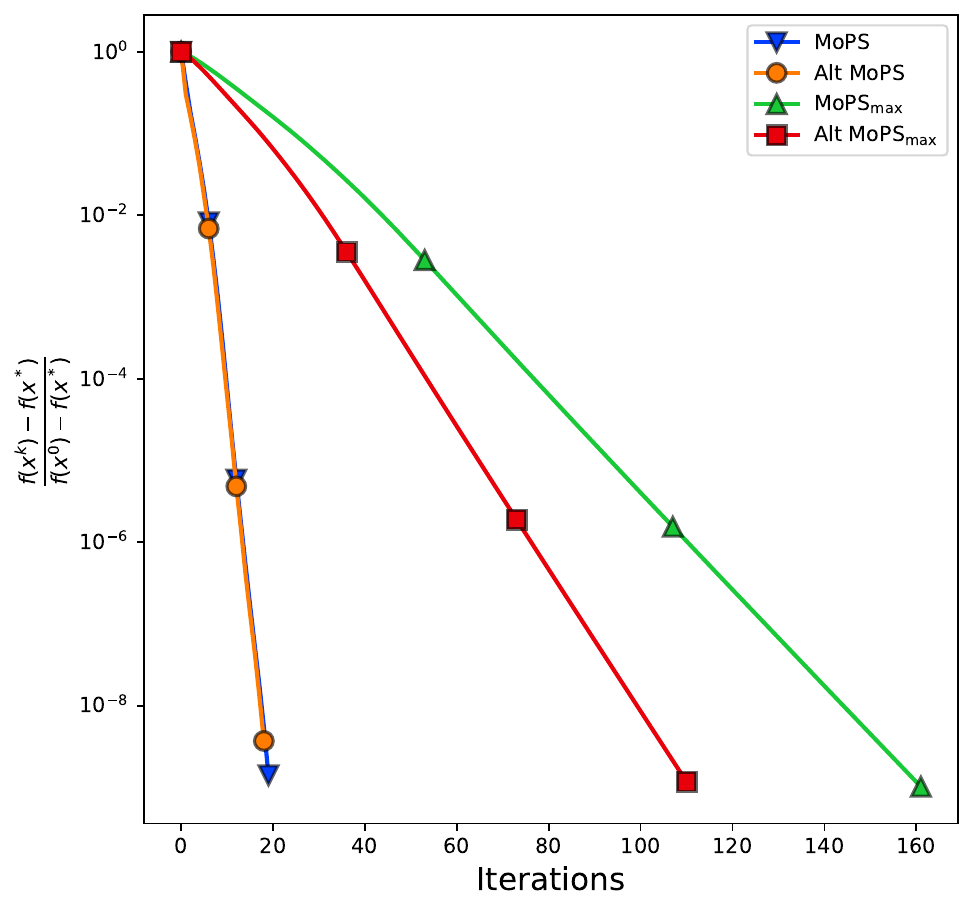}
		\caption{$\g_b=2$}
		\label{subfig:det_lr_choices_b-2}
	\end{subfigure}
    ~
    \begin{subfigure}{0.2\textwidth}
		\includegraphics[width=\textwidth]{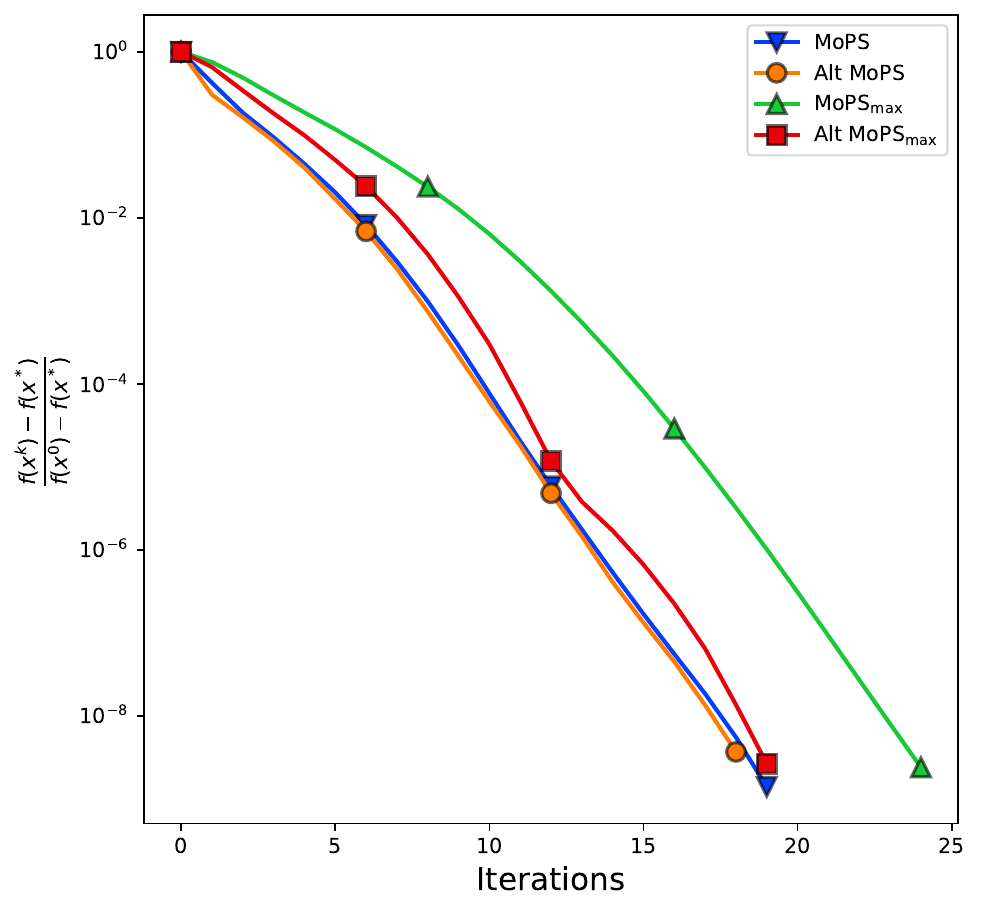}
		\caption{$\g_b=10$}
		\label{subfig:det_lr_choices_b-10}
	\end{subfigure}
    ~
    \begin{subfigure}{0.2\textwidth}
		\includegraphics[width=\textwidth]{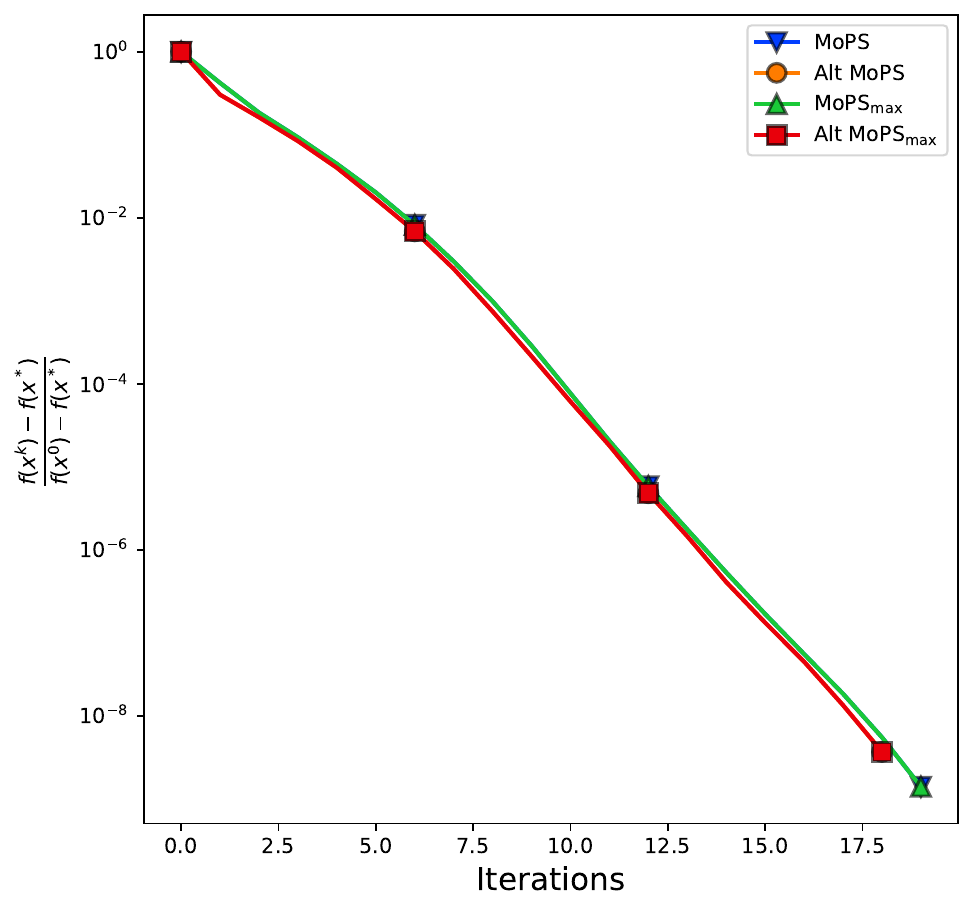}
		\caption{$\g_b=100$}
		\label{subfig:det_lr_choices_b-100}
	\end{subfigure}
    
	\caption{Comparison of various upper bounds for \ref{eq:mopsmax} and other alternatives on logistic regression with synthetic data.}
	\label{fig:det_lr_choices_b}
\end{figure}

\subsection{Plots for \texorpdfstring{$C_1$}{C1} and \texorpdfstring{$C_2$}{C2}}

Below it a plot of the constants $C_1$ and $C_2$ as functions of $\b$ with $\g_b=2$ and $\a=1$ (recall $\g_b\geq\a$) from \Cref{thm:shb-sps-max}. Both functions are hyperbolas and the vertical line is the constant $\b=\frac{\a}{2\g_b-\a}$. We see that in the interval $\b\in\left[0,\frac{\a}{2\g_b-\a}\right)$ both functions are increasing. 

\begin{figure}[H]
	\centering
	\includegraphics[width=0.3\textwidth]{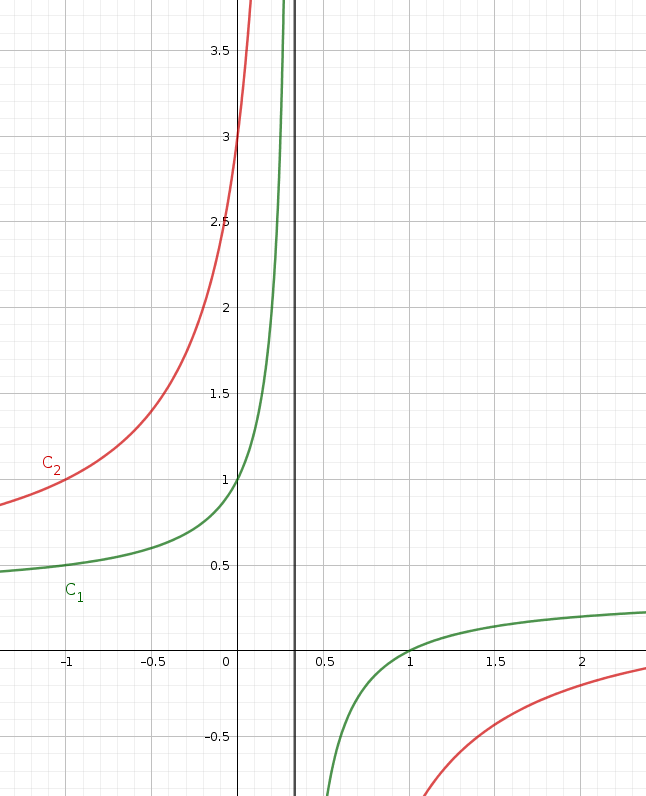}
    
	\caption{Plots for the constants $C_1$ and $C_2$ as functions of $\b$. Here we have chosen $\g_b=2$ and $\a=1$.}
	\label{fig:consts}
\end{figure}

\subsection{More convex experiments and parameter settings}
\label{sec:logreg-exp}

In this section, we list the parameters, architectures and hardware that we used for the deep learning experiments. The information is collected in \Cref{tab:logreg-params}. We also include some extra experiments in \Cref{fig:m_logreg_set_vehicle_bs_16_e_100,fig:m_logreg_set_letter_bs_256_e_100,fig:dec_m_logreg_set_vehicle_bs_85_e_100,fig:dec_m_logreg_set_glass_bs_32_e_100}. 

\begin{table}[H]
    \centering
    \begin{tabular}{ll}
        \toprule
        Hyper-parameter & Value \\
        \midrule
        Architecture & Logistic Regression \\
        GPUs & 1x NVIDIA GeForce RTX 3050 \\
        Batch-size & See caption of each plot \\
        Epochs & 100 \\
        Trials & 5 \\
        Weight Decay & 0.0 \\
        \bottomrule
    \end{tabular}
    \caption{Logistic regression experiment}
    \label{tab:logreg-params}
\end{table}

\begin{minipage}{0.48\textwidth}
\begin{figure}[H]
	\centering
	\includegraphics[width=\textwidth]{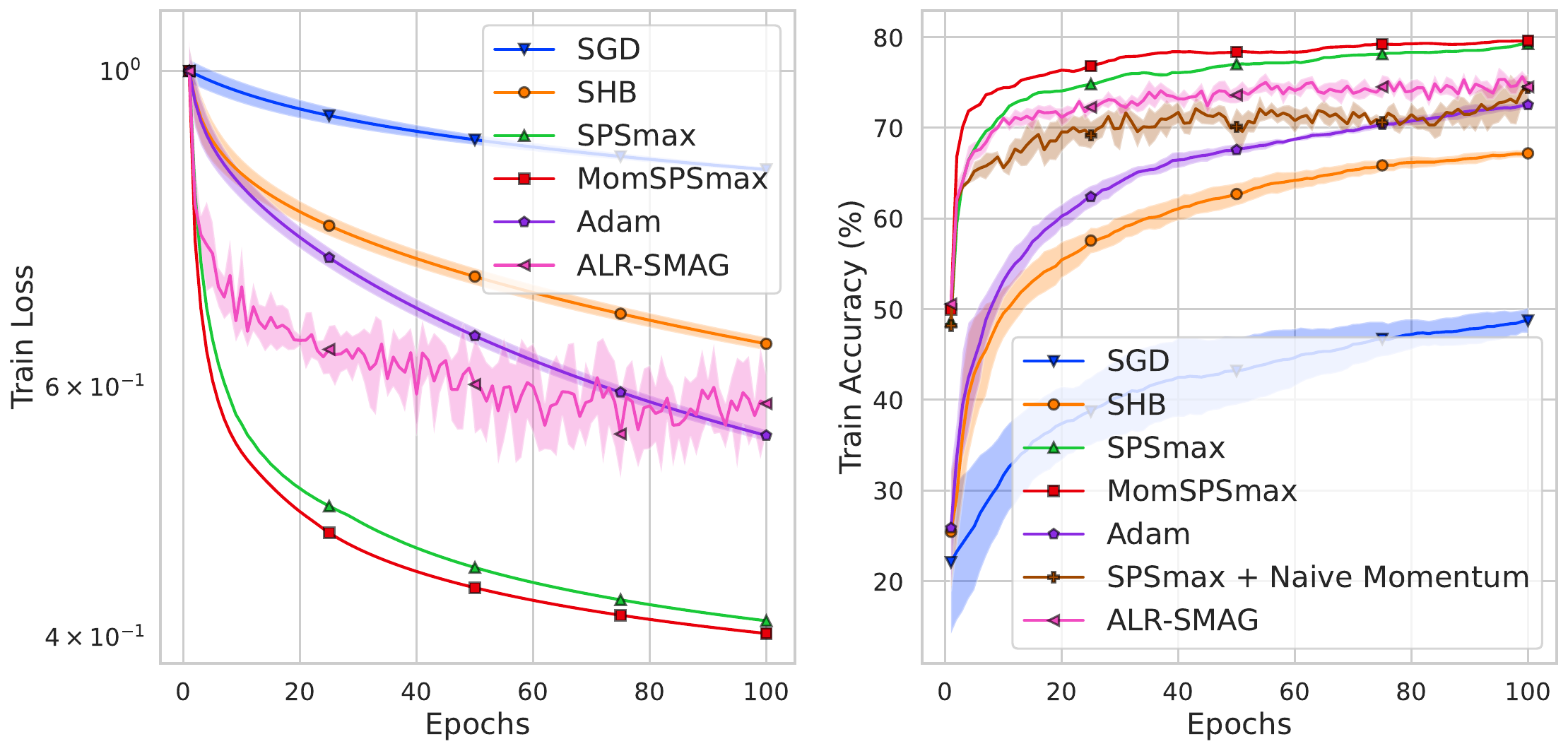}
	\caption{\small LibSVM dataset: vehicle, Batch size: 16}
	\label{fig:m_logreg_set_vehicle_bs_16_e_100}
\end{figure}
\end{minipage}
~
\begin{minipage}{0.48\textwidth}
\begin{figure}[H]
	\centering
	\includegraphics[width=\textwidth]{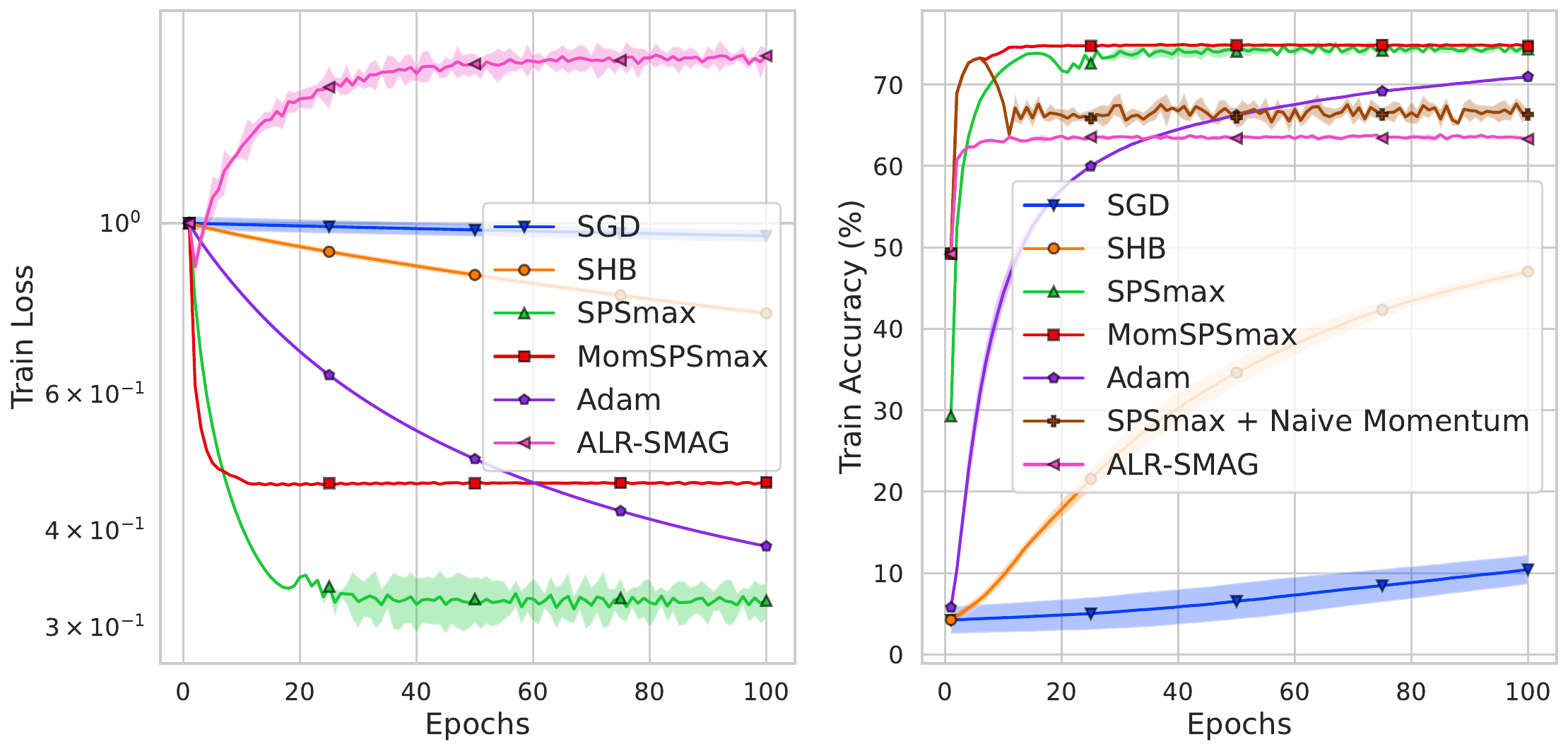}
	\caption{\small LibSVM dataset: letter, Batch size: 256}
	\label{fig:m_logreg_set_letter_bs_256_e_100}
\end{figure}
\end{minipage}

\begin{minipage}{0.48\textwidth}
\begin{figure}[H]
	\centering
	\includegraphics[width=\textwidth]{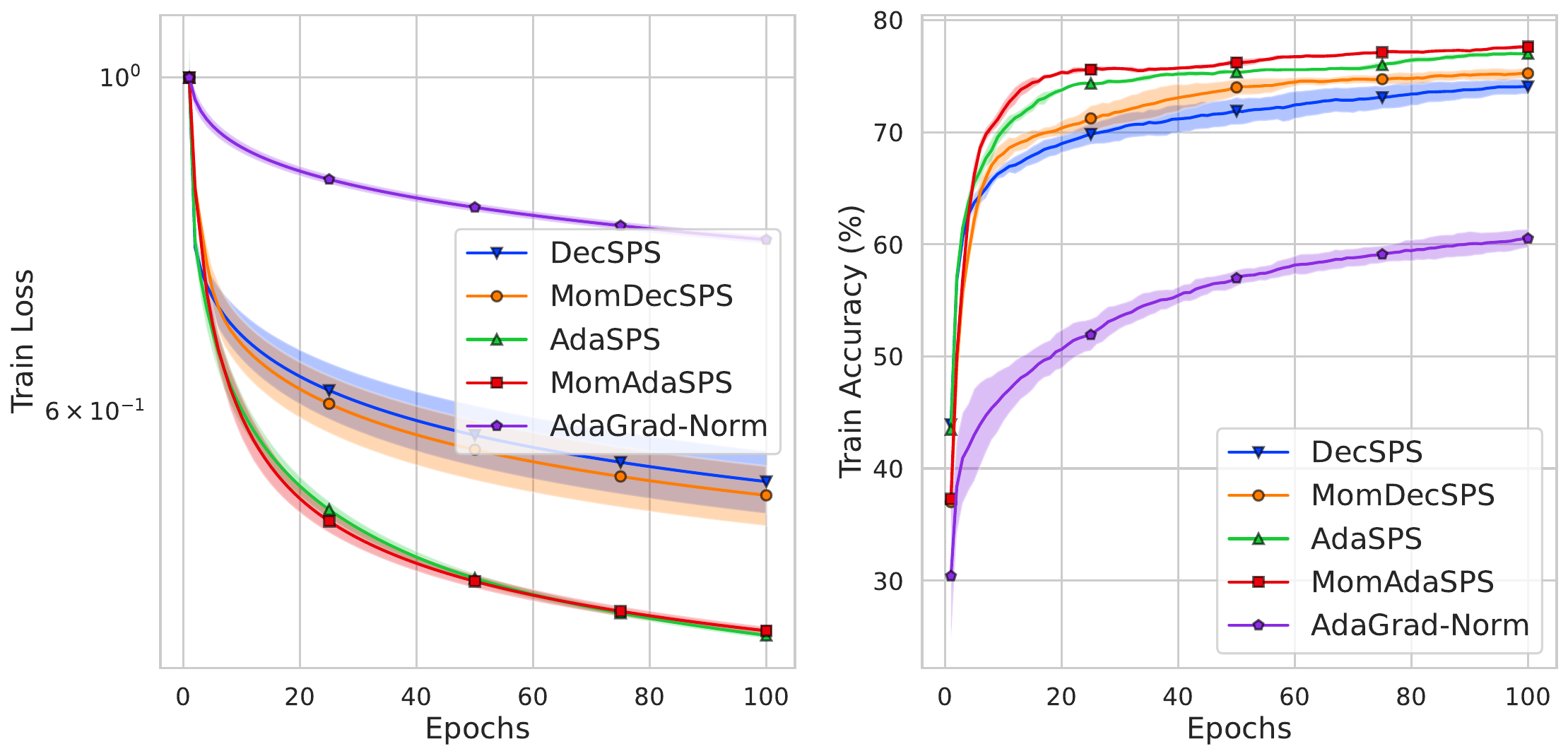}
	\caption{\small LibSVM dataset: vehicle, Batch size: 85}
	\label{fig:dec_m_logreg_set_vehicle_bs_85_e_100}
\end{figure}
\end{minipage}
~
\begin{minipage}{0.48\textwidth}
\begin{figure}[H]
	\centering
	\includegraphics[width=\textwidth]{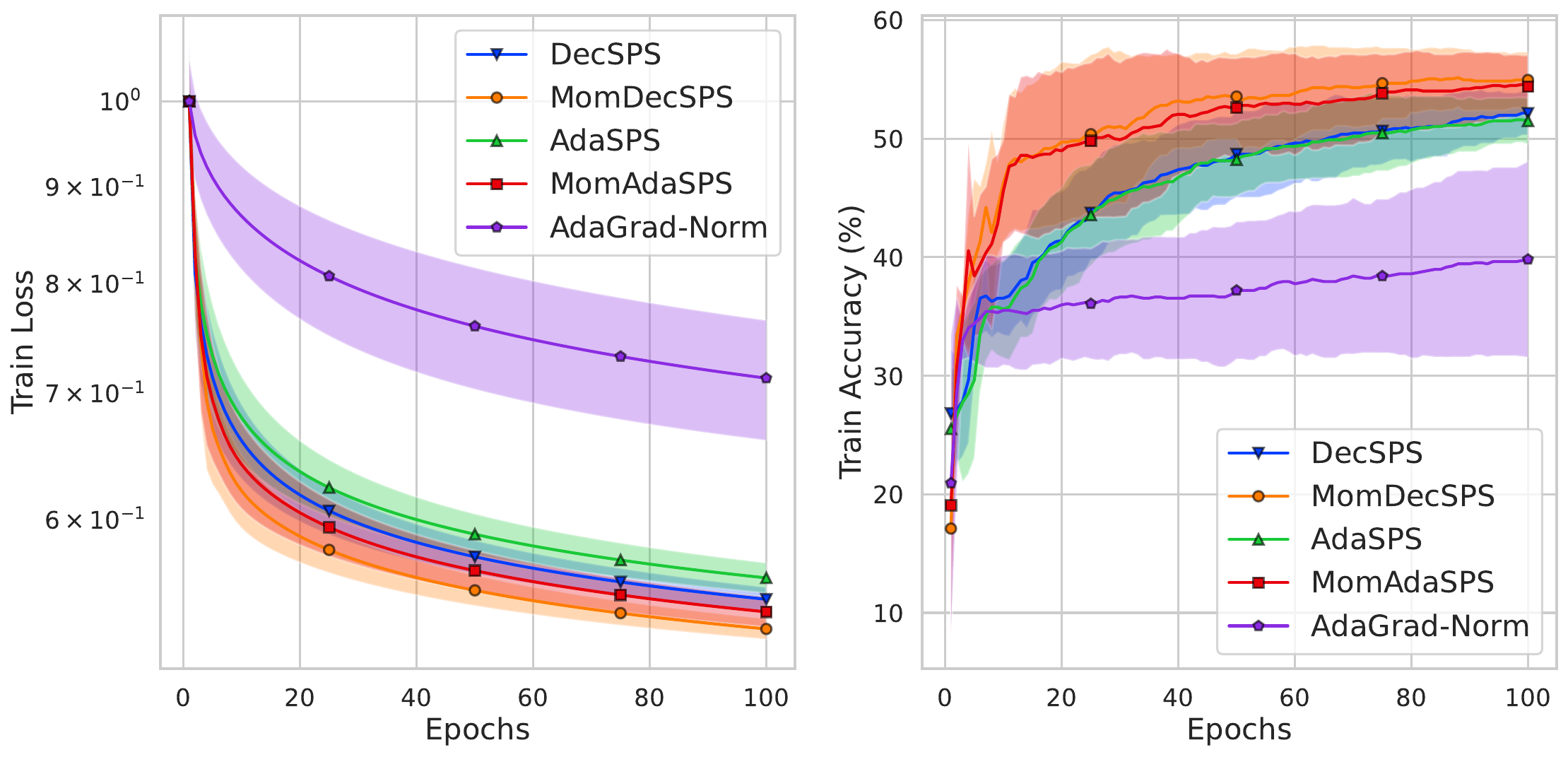}
	\caption{\small LibSVM dataset: glass, Batch size: 32}
	\label{fig:dec_m_logreg_set_glass_bs_32_e_100}
\end{figure}
\end{minipage}

\subsection{More deep learning experiments and parameter settings}
\label{sec:dl-exp}

In this section, we list the parameters, architectures and hardware that we used for the deep learning experiments. The information is collected in \Cref{tab:dl-params}. We also include some extra experiments in \Cref{fig:resnet18_set_cifar10_bs_256_e_100,fig:hresnet18_set_cifar100_bs_256_e_100}. For ALR-SMAG we use $c=0.1$ in the DL experiments. 

\begin{table}[H]
    \centering
    \begin{tabular}{ll}
        \toprule
        Hyper-parameter & Value \\
        \midrule
        Architecture & ResNet 10/34 \citep{he2016deep}\\
        GPUs & 1x Nvidia RTX 6000 Ada Generation\\
        Batch-size & 256 \\
        Epochs & 100 \\
        Trials & 5 \\
        Weight Decay & 0.0 \\
        \bottomrule
    \end{tabular}
    \caption{CIFAR10 experiment}
    \label{tab:dl-params}
\end{table}

\begin{minipage}{0.48\textwidth}
\begin{figure}[H]
	\centering
	\includegraphics[width=\textwidth]{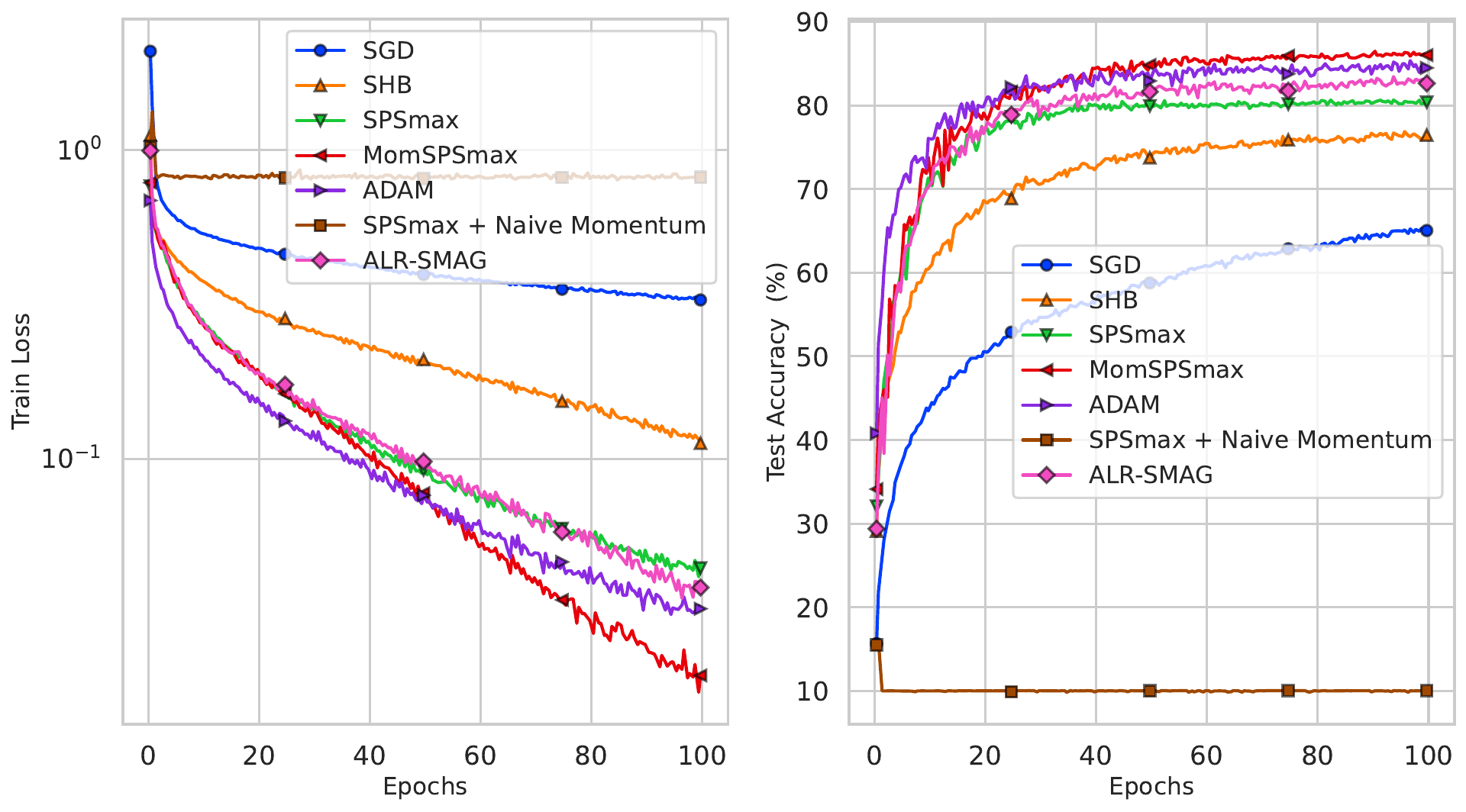}
    \caption{Resnet 18 on CIFAR 10}
    \label{fig:resnet18_set_cifar10_bs_256_e_100}
\end{figure}
\end{minipage}
~
\begin{minipage}{0.48\textwidth}
\begin{figure}[H]
	\centering
	\includegraphics[width=\textwidth]{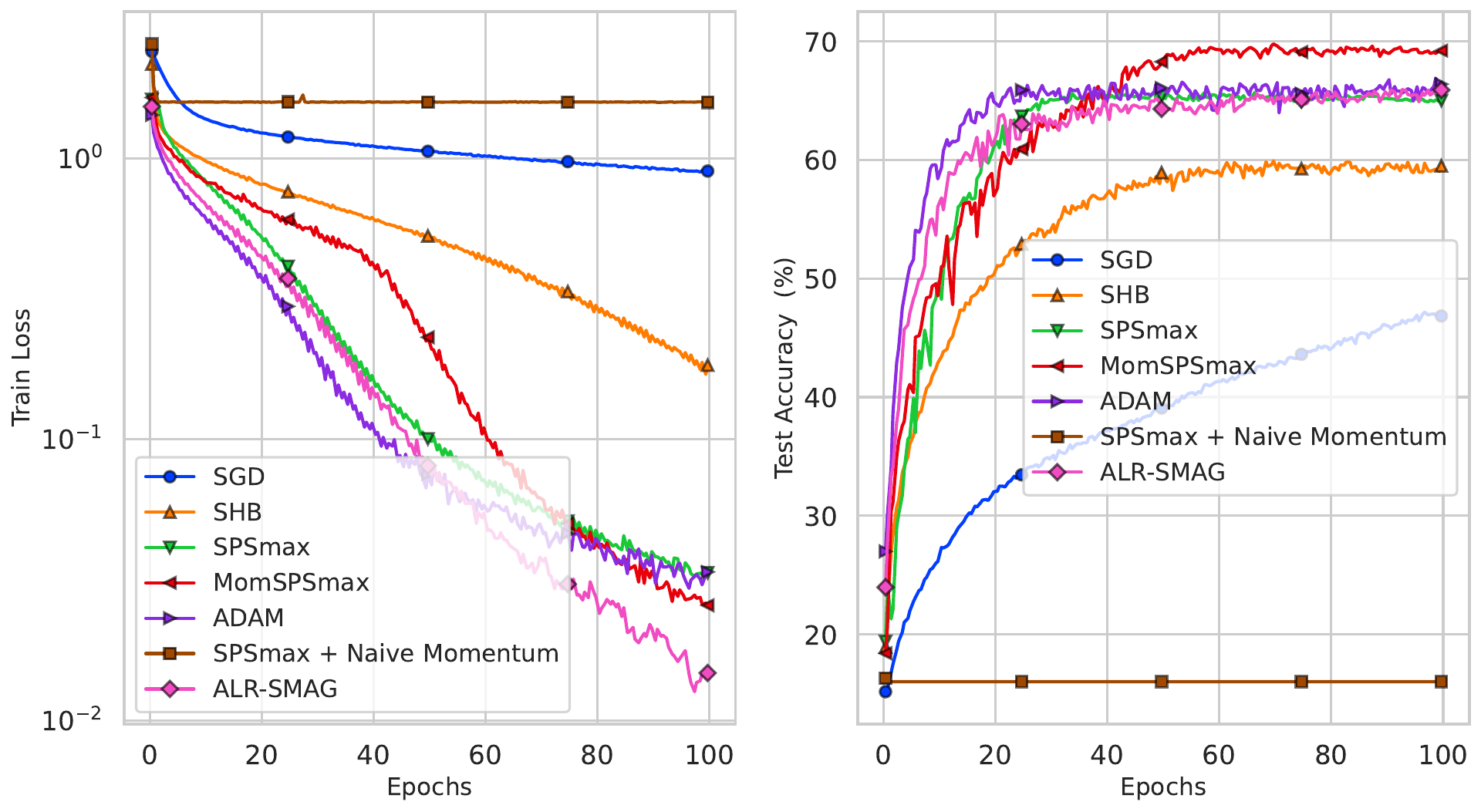}
    \caption{Resnet 18 on CIFAR 100}
	\label{fig:hresnet18_set_cifar100_bs_256_e_100}
\end{figure}
\end{minipage}

\end{document}

%% file: packages.tex
\usepackage{algorithm}
\usepackage[noEnd=false]{algpseudocodex} 

\usepackage{amsmath,amsthm,amsfonts} 
\newcommand\numberthis{\addtocounter{equation}{1}\tag{\theequation}}
\usepackage{thmtools} 
	\allowdisplaybreaks[1]
    
	\newtheorem*{theorem*}{Theorem}
    \newtheorem*{proposition*}{Proposition}
    \newtheorem*{lemma*}{Lemma}
	\newtheorem*{corollary*}{Corollary}
    \newtheorem*{application*}{Application}
    \newtheorem*{example*}{Example}
    \newtheorem*{conjecture*}{Conjecture}

    \newtheorem*{definition*}{Definition}
    \newtheorem*{construction*}{Construction}
    \newtheorem*{assumption*}{Assumption}
    \newtheorem*{remark*}{Remark}
    \newtheorem*{notation*}{Notation}
    \newtheorem*{problem*}{Problem}

    \definecolor{shadecolor}{gray}{0.9}
    \declaretheoremstyle[
    headfont=\normalfont\bfseries,
    notefont=\mdseries, notebraces={(}{)},
    bodyfont=\normalfont,
    postheadspace=0.5em,
    spaceabove=1pt,
    mdframed={
      skipabove=8pt,
      skipbelow=8pt,
      hidealllines=true,
      backgroundcolor={shadecolor},
      innerleftmargin=4pt,
      innerrightmargin=4pt}
    ]{shaded}

	\declaretheorem[style=shaded,
                    name=Theorem,
					style=plain,
					numberwithin=section,
					refname={theorem,theorems},
					Refname={Theorem,Theorems}]
					{theorem}
	\declaretheorem[style=shaded,
                    name=Proposition,
					sibling=theorem,
					style=plain,
					refname={proposition,propositions},
					Refname={Proposition,Propositions}]
					{proposition}
	\declaretheorem[style=shaded,
                    name=Lemma,
					sibling=theorem,
					style=plain,
					refname={lemma,lemmas},
					Refname={Lemma,Lemmas}]
					{lemma}
	\declaretheorem[style=shaded,
                    name=Corollary,
					sibling=theorem,
					style=plain,
					refname={corollary,corollaries},
					Refname={Corollary,Corollaries}]
					{corollary}

	\declaretheorem[style=shaded,
                    name=Definition,
					sibling=theorem,
					style=definition,
					refname={definition,definitions},
					Refname={Definition,Definitions}]
					{definition}
	
	\declaretheorem[style=shaded,
                    name=Assumption,
					sibling=theorem,
					style=definition,
					refname={assumption,assumptions},
					Refname={Assumption,Assumptions}]
					{assumption}
	\declaretheorem[style=shaded,
                    name=Remark,
					sibling=theorem,
					style=definition,
					refname={remark,remarks},
					Refname={Remark,Remarks}]
					{remark}

\usepackage{enumitem} 
\usepackage{subcaption} 
\usepackage{url} 
\usepackage{float}
\usepackage{csquotes}
\usepackage{wrapfig}

\usepackage{cleveref} 

\usepackage{colortbl}
\usepackage{pifont}
\newcommand{\cmark}{\ding{51}}%
\newcommand{\xmark}{\ding{55}}%

\usepackage[colorinlistoftodos,bordercolor=orange,backgroundcolor=orange!20,linecolor=orange,textsize=scriptsize]{todonotes}

\usepackage{booktabs}
\usepackage{aligned-overset}

%% file: commands.tex
\renewcommand{\a}{\alpha}
\renewcommand{\b}{\beta}
\newcommand{\g}{\gamma}
\renewcommand{\d}{\delta}

\newcommand{\h}{\eta}

\renewcommand{\l}{\lambda}
\newcommand{\m}{\mu}

\renewcommand{\r}{\rho}
\newcommand{\s}{\sigma}
\renewcommand{\t}{\tau}

\newcommand{\N}{\mathbb{N}}

\newcommand{\R}{\mathbb{R}}

\DeclareMathOperator{\E}{\mathbb{E}}

\newcommand{\tn}{\textnormal}
\newcommand{\q}{\quad}